\documentclass[hidelinks,12pt, reqno]{amsart}
\usepackage{amsmath}
\usepackage{amssymb,amscd}
\usepackage{graphicx}
\usepackage{extarrows}
\usepackage{latexsym} 
\usepackage{tikz-cd}
\usepackage{enumitem}
\usepackage{float}
\usepackage[colorlinks]{hyperref}
\usepackage{verbatim} 
\usepackage{mathrsfs}

\usepackage[T1]{fontenc}
\usepackage{fourier}

\usepackage{etoolbox}
\usepackage{titlecaps}
\patchcmd{\subsection}{\bfseries}{\titlecap}{}{}
\patchcmd{\subsection}{-.5em}{.5em}{}{}

\patchcmd{\subsubsection}{-.5em}{.5em}{}{}

\usepackage{xpatch}
\makeatletter   
\xpatchcmd{\@tocline}
{\hfil\hbox to\@pnumwidth{\@tocpagenum{#7}}\par}
{\ifnum#1<0\hfill\else\dotfill\fi\hbox to\@pnumwidth{\@tocpagenum{#7}}\par}
{}{}
\makeatother 

\makeatletter
\def\l@subsection{\@tocline{2}{0pt}{4pc}{6pc}{}}
\def\l@subsubsection{\@tocline{3}{0pt}{8pc}{8pc}{}}
 \makeatother
\hypersetup{
     colorlinks=true,
     linkcolor=brown,
     urlcolor=blue,
     }

\numberwithin{equation}{section}

\newcommand{\C}{\mathbb{C}}
\newcommand{\N}{\mathbb{N}}
\newcommand{\R}{\mathbb{R}}

\newcommand{\D}{\mathbb{D}}

\DeclareMathOperator{\pr}{Pr}

\DeclareMathOperator{\img}{img}
\DeclareMathOperator{\supp}{Supp}
\DeclareMathOperator{\coker}{Coker}
\DeclareMathOperator{\typ}{Type}
\DeclareMathOperator{\im}{Im}
\DeclareMathOperator{\re}{Re}
\DeclareMathOperator{\codim}{Codim}
\DeclareMathOperator{\dom}{Dom}

\newtheorem{theorem}{Theorem}[section]
\newtheorem{corollary}[theorem]{Corollary}
\newtheorem{lemma}[theorem]{Lemma}
\newtheorem{prop}[theorem]{Proposition}

\theoremstyle{definition}
\newtheorem{definition}[theorem]{Definition}
\newtheorem{remark}[theorem]{Remark}
\newtheorem{example}[theorem]{Example}
\newtheorem*{thm}{Theorem}

\def \mc{\mathcal}

\begin{document}

\title[Generalized complex Stein manifold]{Generalized Complex Stein Manifold}
\author[D. Pal]{Debjit Pal}

\address{Department of Mathematics, Indian Institute of Science Education and Research, Pune, India}

\email{\href{mailto:debjit.pal@students.iiserpune.ac.in}{debjit.pal@students.iiserpune.ac.in}}



\begin{abstract} 
We introduce the notion of a generalized complex (GC) Stein manifold and provide complete characterizations in three fundamental aspects. First, we extend Cartan's Theorem A and B within the framework of GC geometry. Next, we define $L$-plurisubharmonic functions and develop an associated $L^2$ theory. This leads to a characterization of GC Stein manifolds using $L$-plurisubharmonic exhaustion functions. Finally, we establish the existence of a proper GH embedding from any GC Stein manifold into $\mathbb{R}^{2n-2k} \times \mathbb{C}^{2k+1}$, where $2n$ and $k$ denote the dimension and type of the GC Stein manifold, respectively. This provides a characterization of GC Stein manifolds via GH embeddings. Several examples of GC Stein manifolds are given.
\end{abstract}

\maketitle

\renewcommand\contentsname{\vspace{-1cm}}
{
  \hypersetup{linkcolor=black}
  \tableofcontents
}

\section{Introduction}
Stein manifolds (\cite{stein51}) play a crucial role in resolving Cousin's problems in complex geometry, extending the concept of domains of holomorphy and Riemann surfaces. Their study is based on three fundamental characterizations, namely, $(1)$ Cartan's Theorems A and B, which involve holomorphic coherent sheaves on Stein manifolds \cite{cartan53, siu68}; $(2)$ Grauert's characterization via strictly plurisubharmonic exhaustion functions \cite{grauert60, grauert58}; and $(3)$ the holomorphic embeddings of Stein manifolds into complex Euclidean spaces \cite{bishop, narasimhan60}.

\vspace{0.2em}
Generalized complex (GC) geometry provides a broad framework encompassing various geometric structures, with complex and symplectic structures representing its two extreme cases. Introduced by Hitchin \cite{Hit} and further developed by his students Gualtieri \cite{Gua, Gua2} and Cavalcanti \cite{Cavth}, GC geometry prompts a natural question: is there an appropriate analog of Stein manifolds within this framework, and how can we fully characterize it in a way that extends the classical notions of Stein manifolds?

\vspace{0.2em}
In this article, we introduce a new concept that extends the classical notion of Stein manifolds, which we call \textit{generalized complex Stein or GC Stein} manifolds. A regular GC structure induces a regular foliation with symplectic leaves and a transverse complex structure. The GC Stein manifolds are intuitively characterized by their transverse Stein structure, though they include a wider class than the classical Stein manifolds, see Example \ref{egg1}. The main contribution of this article is the comprehensive characterization of GC Stein manifolds in three fundamental aspects by adapting methods from complex geometry and extending them to the framework of GC geometry. A detailed outline of the paper is provided below.

\vspace{0.2em}
In Section \ref{prelim}, we describe the basic facts on GCS, generalized holomorphic (GH) maps, and GC submanifolds. In Section \ref{gc stein}, we introduce the notion of GC Stein manifold $X$ (cf. Definition \ref{main def}), and show that $\mathcal{O}^{m}(X)$ (with $m\in\N$) is a Fr\'{e}chet space when endowed with the topology of uniform convergence on compact subsets, see Theorem \ref{f spc}. We also provide several examples of GC Stein manifolds, see Example \ref{e1} and Example \ref{e2}. In Section \ref{cartan}, we present a complete characterization of GC Stein manifolds by providing an analog of Cartan's Theorems A and B. Specifically, we establish the following extension of Cartan's Theorem A and Theorem B is GC geometry. 
\begin{thm}(Theorem \ref{main})
 Let $X$ be a GC Stein manifold and $\mathcal{F}$ be a coherent sheaf of\,\,\,$\mathcal{O}$-modules where $\mathcal{O}$ is the sheaf of GH functions on $X$. Then, we have
 \begin{enumerate}
 \setlength\itemsep{0.2em}
     \item (Theorem B) $H^{p}(X,\mathcal{F})=0$ for $p\geq 1\,.$
     \item (Theorem A) $H^{0}(X,\mathcal{F})$ generates $\mathcal{F}_{x}$ for all $x\in X\,.$
 \end{enumerate}
\end{thm}
To prove Theorem \ref{main}, we first extend Cartan's Theorems A and B to the product of balls in $\mathbb{R}^{2M} \times \mathbb{C}^N$ (where $M, N > 0$). According to the generalized Darboux theorem (see Theorem \ref{darbu thm}), focusing on the product of balls in $\mathbb{R}^{2M} \times \mathbb{C}^N$ is natural. We establish Cartan's Theorem B for this setting in Theorem \ref{imp claim}, and the proof is detailed in Subsection \ref{pf}. Following that, we prove Theorem A in Theorem \ref{thm}.
Now, to transition from the product of balls in $\mathbb{R}^{2M} \times \mathbb{C}^N$ to GC Stein manifolds, we present a GH embedding (see Definition \ref{GH map}) theorem for relatively compact open sets in GC Stein manifolds, see Proposition \ref{prop}.

\vspace{0.2em}
In Theorem \ref{imp thm2}, we precisely describe the image of the embedding presented in Proposition \ref{prop} using GC submanifolds (cf. Definition \ref{sub def mfld}). This description is based on the properties of GC maps and GC submanifolds provided in Theorem \ref{imp thm1} and Corollary \ref{imp cor1}. Additionally, we prove a denseness result in Proposition \ref{prop2}, essential for establishing Theorem \ref{main}.
Using Remark \ref{poi rmk}, Definition \ref{main def} and Theorem \ref{main}, we then establish the following complete characterization of a regular GC manifold as a GC Stein manifold.
\begin{thm}(Theorem \ref{main1})
    Let $X$ be a regular GC manifold and $\mathcal{F}$ be a coherent sheaf of $\mathcal{O}$-modules. Then $X$ is a GC Stein Manifold if and only if the following holds:
    \begin{enumerate}
    \setlength\itemsep{0.2em}
     \item $H^{p}(X,\mathcal{F})=0$ for $p\geq 1\,,$ and $H^{0}(X,\mathcal{F})$ generates $\mathcal{F}_{x}$ for all $x\in X\,.$
     \item $X$ satisfies the property: Let $f:U\longrightarrow(\R^{2},\omega^{-1}_0)$ be a Poisson map on an open set $U\subseteq X$. Then,
    for each $x\in U$, there exists an open set $x\in U^{'}\subset U$ and a GH map    $\tilde{f}:X\longrightarrow(\R^{2},\omega^{-1}_0)$ such that $\tilde{f}|_{U^{'}}=f|_{U^{'}}\,.$
    \end{enumerate}
\end{thm}
Section \ref{l2} presents the characterization of GC Stein manifolds using $L^2$-theory. In this section, we introduce a concept of plurisubharmonic function called $L$-plurisubharmonic functions (see Definition \ref{l-pluri}), which extends classical smooth plurisubharmonic functions (see Remark \ref{rmk pluri}). We also provide the existence of a strictly $L$-plurisubharmonic exhaustion function on GC Stein manifolds, see Theorem \ref{pluri thm}. Following this, in Subsection \ref{l2-1} we present an $L^2$-theory for complex vector bundles, based on the approach in \cite{batu17}. We provide an estimation result for GC Stein manifolds, namely,
\begin{thm} (Theorem \ref{est thm})
    Let $\psi\in C^{\infty}(X,\C)$ be a real function. For any $s\in C^{\infty}_{c}(X,E^{p,q})$, there exists a positive real smooth function $\hat{C}\in C^0(X,\C)$ such that the estimate
    $$\int\Big(\gamma_{\psi}-\hat{C}\Big)\,|s|^{2}\,e^{-\psi}\,d\lambda\leq 4\Big(|\tilde{d}^{*}_{L,q}s|^2_{\psi}+|\tilde{d}_{L,q+1}s|^2_{\psi}\Big)\,,$$ holds, where $\gamma_{\psi}$ is a positive continuous function that provides the smallest eigenvalues of the hermitian metric $e^{-\psi}h^{'}$ over a system of local trivializations of $E^{p,q}$. Here $L$ is the $i$-eigen bundle of the GCS and $E^{p,q}:=\wedge^{p}\overline{L}^{*}\otimes\wedge^{q}L^{*}$ ($p,q\geq 0$).
\end{thm}
Using Theorem \ref{est thm}, we establish an analog of the Poincar\'{e} lemma with the $d_{L}$ operator (cf. \eqref{dL-dL bar}), as follows.
\begin{thm} (Theorem \ref{pluri thm1})
    For a regular GC manifold $X$ of positive type, if there is a strictly $L$-plurisubharmonic function $\psi$ on $X$ such that $\big\{x\in X\,\big|\,\psi(x)<c\big\}\subset\subset X$ for every $c\in\R$, then the equation $d_{L}s=s'$ (in the sense of distribution theory) has a solution $s\in L^2(X, E^{p,q}, loc)$ for every $s'\in L^2(X, E^{p,q+1}, loc)$ such that $d_{L}s'=0\,.$
\end{thm}
Additionally, applying Theorem \ref{pluri thm1}, we provide results on the approximation by global GH functions in $L^2$-norms (see Proposition \ref{apprx prop1}) and uniformly over compact subsets (see Proposition \ref{apprx prop}). Finally, we present the following converse of Theorem \ref{pluri thm}.
\begin{thm}(Theorem \ref{main2})
For a regular GC manifold $X$ to be a GC Stein manifold, it is necessary and sufficient that the following holds:
    \begin{enumerate}
    \setlength\itemsep{0.2em}
        \item $X$ satisfies the property $(2)$ in Theorem \ref{main1}.
        \item There exists a strictly $L$-plurisubharmonic exhaustion function $\psi$ on $X$. 
    \end{enumerate}
\end{thm}
We address the characterization problem through GH embeddings into Euclidean spaces. By utilizing the hemicompactness of GC Stein manifolds, we show that the spaces $\mathcal{GH}_{r}^N(X)$ and $\mathcal{GH}_{r,1-1}^N(X)$ (cf. \eqref{eqnn3}) are dense in the Whitney topology, see Theorem \ref{dns thm}. Additionally, we introduce the concept of a GC polyhedron in a GC Stein manifold and prove its existence in Proposition \ref{poly prop}. Consequently, we establish the existence of a proper injective GH immersion and provide a complete characterization, which is outlined as follows.
\begin{thm} (Theorem \ref{main3} and Theorem \ref{main4})
  Let $X^{2n}$ be a regular manifold of type $K>0$. Then $X$ is a GC Stein manifold if and only if a proper GH map $F: X\longrightarrow\R^{2n-2k}\times\C^{2k+1}$ exists which is also a smooth embedding. Moreover, $\img(F)$ is a closed embedded GC submanifold in $\R^{2n-2k}\times\C^{2k+1}$.
\end{thm}
It may be noted that $\mathbb{R}^{2M} \times \mathbb{C}^N$ is a natural example of a GC Stein manifold, and by Theorem \ref{main4}, only closed embedded GC submanifolds of\,\, $\mathbb{R}^{2M} \times \mathbb{C}^N$ can be GC Stein manifolds. Ben-Bassat et al. \cite{ben04} provide a detailed description of GC subspaces within $\mathbb{R}^{2M} \times \mathbb{C}^N$, enabling the construction of a large class of examples of GC Stein manifolds.

\section{Preliminaries}\label{prelim}
We first start by recalling some basic notions of generalized complex (in short GC) geometry. We shall rely upon  \cite{Gua, Gua2} for most of the basic definitions and results.  

\medskip
Given any smooth manifold $X$, the direct sum of tangent and cotangent bundles of $X$, which we denote by $TX\oplus T^{*}X$, is endowed with a natural symmetric bilinear form,
\begin{equation}\label{bilinear}
    \langle A_1+\xi,A_2+\eta\rangle\,:=\,\frac{1}{2}(\xi(A_2)+\eta(A_1))\,.
\end{equation}
It is also equipped with the \textit{Courant Bracket}  defined as follows.
\begin{definition}
The Courant bracket is a skew-symmetric bracket defined on smooth sections of $TX\oplus T^{*}X$, given by
\begin{equation}\label{bracket}
    [A_1+\xi,A_2+\eta] := [A_1,A_2]_{Lie}+\mathcal{L}_{A_1}\eta-\mathcal{L}_{A_2}\xi-\frac{1}{2}d(i_{A_1}\eta-i_{A_2}\xi),
\end{equation}  
where $A_1,A_2\in C^{\infty}(TX)$, $\xi,\eta\in C^{\infty}(T^{*}X)$, $[\,,\,]_{Lie}$ is the usual Lie bracket on $C^{\infty}(TX)$, and $\mathcal{L}_{A_1},\,\, i_{A_1}$ denote the Lie derivative and the interior product of forms with respect to the vector field $A_1$, respectively.
\end{definition}
Consider the action of $TX\oplus T^{*}X$  on $\wedge^{\bullet}T^{*}X$ defined by $$(A+\xi)\cdot\psi = i_A\varphi+\xi\wedge\psi \,.$$  This action can be extended to the Clifford algebra of $TX\oplus T^{*}X$ corresponding to the natural pairing \eqref{bilinear}. This gives a natural choice for spinors, namely, the exterior algebra of cotangent bundle, $\wedge^{\bullet}T^{*}X$.

\vspace{0.2em}
We are now ready to define the notion of GCS in a smooth manifold $X$ in three equivalent ways.
\begin{definition}(cf. \cite{Gua})\label{gcs}
A \textit{generalized complex structure (GCS)} is determined by any of the following three equivalent sets of data:

\vspace{0.2em}
\begin{enumerate}
\setlength\itemsep{0.4em}
    \item  A bundle automorphism $\mathcal{J}$ of $TX\oplus T^{*}X$ which satisfies the following conditions:

    \vspace{0.3em}
    \begin{itemize}
 \setlength\itemsep{0.4em}
        \item[(a)] $\mathcal{J}^{2}=-1\,,$ and $\mathcal{J}^{*}=-\mathcal{J}\,,$ that is, $\mathcal{J}$ is orthogonal with respect to the natural pairing in \eqref{bilinear}\,,
        \item[(b)] $\mathcal{J}$ has vanishing {\it Nijenhuis tensor}, that is, for all $C, D \in C^{\infty}(TX\oplus T^{*}X)$,
   $$ N(C, D) :=[\mathcal{J}C, \mathcal{J}D]-\mathcal{J} [\mathcal{J}C, D] - \mathcal{J} [C, \mathcal{J} D] 
    - [C, D]=0\,. $$ 
    \end{itemize}
     \item A subbundle, say $L_{X}$, of $(TX\oplus T^{*}X)\otimes\C$ which is maximal isotropic with respect to the natural bilinear form \eqref{bilinear}, involutive with respect to the Courant bracket \eqref{bracket}, and satisfies $L_{X}\cap\overline{L_{X}}=\{0\}$.
     \item A line subbundle $U_{X}$ of $\wedge^{\bullet}T^{*}X\otimes\C$ which generated locally at each point by a form of the form $\varphi=e^{(B+i\omega)}\wedge\Omega$, such that it satisfies the following non-degeneracy condition $$\omega^{n-k}\wedge\Omega\wedge\overline{\Omega}\neq 0,$$ where $B$ and $\omega$ are real 2-forms and $\Omega$ is a decomposable complex $k$-form, and  $\varphi$ satisfies the following integrability condition
    \begin{equation}
        d\varphi=u\cdot\varphi,
    \end{equation}
   for some $u\in (TX\oplus T^{*}X)\otimes\C$, where $d$ is the exterior derivative.
\end{enumerate}
\end{definition}
In Definition \ref{gcs}, the equivalent conditions $(1)$ and $(2)$ are related to each other by the fact that the subbundle $L_X$ may be obtained as the $+i$-eigenbundle of the automorphism $\mathcal{J}$. The line subbundle $U_{X}$ in Definition \ref{gcs}, is called the {\it canonical line bundle} associated to the GCS, and it is annihilated by $L_{X}$ under the above Clifford action.

\vspace{0.5em}
Given any GC manifold  $(X,\,\mathcal{J})$, we can deform $\mathcal{J}$ by a real closed $2$-form $B$ to get another GCS on $M$,  \begin{equation}\label{B transformation}
        (\mathcal{J})_{B}=e^{-B}\circ\mathcal{J}\circ e^{B}\quad\text{where}\quad e^{B}=\begin{pmatrix} 
	       1 & 0 \\
	       B & 1 \\
	    \end{pmatrix}\,.
    \end{equation}
This operation is called a \textit{$B$-field transformation} (or \textit{$B$-transformation}). Then, the corresponding $+i$-eigenbundle of $(\mathcal{J})_{B}$ is 
\begin{equation}\label{L-B}
(L_{X})_{B}=\big\{A+\xi-B(A,\,\cdot)\,|\,A+\xi\in L_{X}\big\}\,.    
\end{equation}
Let us consider some simple examples of GCS.
\begin{example}\label{complx eg}
Let $(X,\, J_{X})$ is a complex manifold with a complex structure $J_{X}$. Then the natural GCS on $X$ is given by the bundle automorphism  
\[
\mathcal{J}:=
\begin{pmatrix}
 -J_{X}     &0 \\
    
    0        &J^{*}_{X}
\end{pmatrix}: TX\oplus T^{*}X\longrightarrow TX\oplus T^{*}X\,.
\] The corresponding $+i$-eigen bundle is
$$L_{X}=T^{0,1}X\oplus(T^{1,0}X)^{*}\,.$$
\end{example}
\begin{example}\label{symplectic eg}
Let $(X,\,\omega)$ be a symplectic manifold with a symplectic structure $\omega$. Then, the bundle automorphism   
\[
\mathcal{J}:=
\begin{pmatrix}
    0    &-\omega^{-1}\\
    \omega    &0
\end{pmatrix}: TX\oplus T^{*}X\longrightarrow TX\oplus T^{*}X\,,
\] gives a natural GCS on $X$. The $+i$-eigen bundle of this GCS is 
$$L_{X}=\big\{A-i\omega(A)\,|\,A\in TM\otimes\C\big\}\,.$$
\end{example}  
\begin{example}\label{prdct gcs eg}
    Consider two GC manifolds denoted as $(M_1,\mathcal{J}_{M_1})$ and $(M_2,\mathcal{J}_{M_2})$. Then, the product GCS on $M_1\times M_2$, is characterized by the canonical line bundle locally given as $\bigwedge^2_{j=1}\pr^{*}_{j}\phi_{j}$. Here, for $j=1,2$, $\phi_{j}$ denote the respective local generator of the canonical line bundle for $\mathcal{J}_{M_{j}}$, while the map $\pr_{j}: M_1\times M_2\longrightarrow M_{j}$ is the natural projection map onto the $j$-th component. In particular, when $M_1=(\R^{2M},\omega_0)$ and $M_2=\C^N$ where $\omega_0$ is the standard symplectic structure and $M,N>0$, we always consider $\R^{2M}\times\C^N$ endowed with the product GCS, induced by the standard complex structure on $\C^N$ (cf. Example \ref{complx eg}) and the symplectic structure $\omega_0$ (see Example \ref{symplectic eg}).
\end{example}

\subsection{Generalized holomorphic map}  We recall some basic facts about generalized holomorphic (in short GH) maps, an analogue of holomorphic maps in the case of GC manifolds. For most of the definitions, we shall rely upon \cite{Gua, ornea2011}. 

\medskip
Let $(V,\mathcal{J}_{V})$ be a GC linear space with  $+i$-eigenspace $L_{V}$.  Given a subspace $E\leq V\otimes\C$ and an element $\sigma\in\wedge^2 E^{*}$, consider the subspace 
\begin{equation}\label{isotropic set}
    L(E,\sigma):=\big\{A+\xi\in E\oplus V^{*}\otimes\C\,| \,\xi|_{E}=\sigma(A)\big\}\,,
\end{equation}
of $(V \oplus V^{\ast}) \otimes \mathbb{C} $.
By \cite[Proposition 2.6]{Gua}, $L(E,\sigma)$ is a maximal isotropic subspace of $(V\oplus V^{*})\otimes\C$ with respect to the bilinear pairing \eqref{bilinear}, and any maximal isotropic subspace of $(V\oplus V^{*})\otimes\C$ is of this form. Consider the projection map 
$$\rho:(V\oplus V^{*})\otimes\C\longrightarrow V\otimes\C\,.$$  Let $\rho(L_{V})=E_{V}$ and let $E_{V}\cap\overline{E_{V}}=\Delta_{V}\otimes\C$ where $\Delta_{V} \le V$ is a real subspace. Then by \cite[Proposition 4.4]{Gua}, we have

\begin{enumerate}
\setlength\itemsep{0.4em}
    \item $L_{V}=L(E_{V},\sigma)$ for some $\sigma\in\wedge^2 E_{V}^{*}$ ;
    \item $E_{V}+\overline{E_{V}}=V\otimes\C$ with a non-degenerate real $2$-form $\Omega_{\Delta_{V}}:=\im(\sigma|_{\Delta_{V}\otimes\C})$ on $\Delta_{V}\otimes\C$.
\end{enumerate}
Following \cite[Section 3]{ornea2011}, $P_{V}:=L(\Delta_{V}\otimes\C,\Omega_{\Delta_{V}})$ is called the associated { \it linear Poisson structure} of $\mathcal{J}_{V}$ on $V\otimes\C$.

\begin{definition}\label{GC map} (\cite{ornea2011})
Let $\psi:(V,\mathcal{J}_{V})\longrightarrow (V^{'},\mathcal{J}_{V^{'}})$ be a linear map between two GC linear spaces. Then $\psi$ is called a \textit{generalized complex (GC) map} if

\vspace{0.3em}
\begin{enumerate}
\setlength\itemsep{0.4em}
    \item $\psi(E_{V})\subseteq E_{V^{'}}$ ,
    \item $\psi_{\star}(P_{V})=P_{V^{'}}$ where $\psi_{\star}$ denotes the pushforward of a Dirac structure, as in \cite[Section 1]{ornea2011}, namely,
      $$\psi_{\star}(P_{V})=\big\{\psi(A)+\eta\in (V^{'}\oplus V^{'*})\otimes\C\,|\,A+\psi^{*}(\eta)\in P_{V}\big\}\,.$$
\end{enumerate} 
\end{definition}

\begin{definition}\label{GH map} (\cite{ornea2011})
 A smooth map $\psi:(X,\mathcal{J}_{X})\longrightarrow (Y,\mathcal{J}_{Y})$ between two GC manifolds is called a \textit{generalized holomorphic (GH) map} if for each $x\in X$,
    $$(\psi_{*})_{x}:T_{x}X\longrightarrow T_{\psi(x)}Y$$ is a GC map. If $(Y,\mathcal{J}_{Y})=(\R^{2},\mathcal{J}_{\R^{2}})$ where $\mathcal{J}_{\R^{2}}$ is as in Example \ref{complx eg}, induced by the standard complex structure $J_0$, then $\psi$ is called a \textit{GH function}. If $\psi$ is also an immersion, a submersion, or a smooth embedding, then we refer to $\psi$ as a GH immersion, a GH submersion, or a GH embedding, respectively.
\end{definition}

\begin{remark}
When we have a $B$-field transformation of $\mathcal{J}$, we observe that $\rho((L_{X})_{B})=\rho(L_{X})$ where $(L_{X})_{B}$ is as in \eqref{L-B}. Since the imaginary part of $\sigma$ is preserved, the corresponding linear Poisson structures are the same for both GCS, implying that the concepts of GC map and GH map are not affected by $B$-field transformations.   
\end{remark}

\begin{remark}\label{poi rmk}
    Note that, any GC manifold is also a Poisson manifold (cf. \cite[Proposition 2.6, Proposition 4.4]{Gua}). Thus, by Definition \ref{GC map}, $f:X\longrightarrow(\R^{2M},\omega_0)$ is a GH map if and only if  $f:X\longrightarrow(\R^{2M},\omega^{-1}_0)$ is a Poisson map.
\end{remark}

\begin{definition}(cf. \cite{pal24})
A diffeomorphism $\phi:(X,\mathcal{J}_{X})\longrightarrow (Y,\mathcal{J}_{Y})$ between two GC manifolds is called a \textit{generalized holomorphic (GH) homeomorphism} if
\begin{equation}\label{gh homeo}
    \begin{pmatrix}
     \phi_{*}    &0 \\
     0           &(\phi^{-1})^{*} 
    \end{pmatrix}
    \circ \mathcal{J}_{X} = \mathcal{J}_{Y}\circ 
    \begin{pmatrix}
     \phi_{*}    &0 \\
     0           &(\phi^{-1})^{*} 
    \end{pmatrix}\,.
\end{equation}
When $Y=X$, $\phi$ is called GH automorphism.
\end{definition}
\begin{definition}\label{def:type} Let $\rho: (TX\oplus T^{\ast}X) \otimes \C \longrightarrow TX \otimes \C $ denote the natural projection. We denote by $E_X$ the image of 
$L_X$ under $\rho$. Let $\Delta_X \otimes \C := E_X \bigcap \overline{E_X}$. 

\vspace{0.2em}
    For each $x\in X$, type of $\mathcal{J}$ at $x$ is defined as 
    $$\typ(x):=\codim_{\C}((E_{X})_{x})=\frac{1}{2}\codim_{\R}((\Delta_{X})_{x})\,.$$ $x$ is called a regular point of $X$ if $\typ(x)$ is constant in a neighborhood of $x$ and $X$ is called a regular GC manifold if each point of $X$ is a regular point. For a regular GC manifold $X$, we denote the type of $X$ by $\typ(X)\,.$ 
\end{definition}
 Note that $\typ(X)$ is invariant under $B$-field transformation. We have the generalized Darboux theorem around any regular point.
\begin{theorem}(\cite[Theorem 4.3]{Gua2})\label{darbu thm}
For a regular point $x\in (X^{2n},\mc{J})$ of $\typ(x)=k$, there exists an open neighborhood $U_{x}\subset X$ of $x$ such that, after a $B$-transformation, $U_{x}$ is GH homeomorphic to $U_1\times U_2$, where $U_1\subset(\R^{2n-2k},\omega_0), U_2\subset\C^{k}$ are  open subsets with $\omega_0$ being the standard symplectic structure.
\end{theorem}
In a simpler terms, Theorem \ref{darbu thm} implies that, for some real closed form $2$-form $B_{\phi}\in\Omega^2(U_1\times U_2)$, there exists a GH homeomorphism $$\phi:(U_{x},\mc{J}_{U_{x}})\longrightarrow (U_1\times U_2,(\mathcal{J}_{U_1\times U_2})_{B_{\phi}})$$
where  $(\mathcal{J}_{U_1\times U_2})_{B_{\phi}})$ is  the $B$-transformation, as in \eqref{B transformation}, of the product GCS, denoted by $\mathcal{J}_{U_1\times U_2}\,.$ Let $p=(p_1,\ldots,p_{2n-2k})$ and $z=(z_1,\ldots,z_{k})$ represent coordinate systems for $\R^{2n-2k}$ and $\C^{k}$, respectively, and consider the corresponding local coordinates around $x$
\begin{equation}\label{loc coordi}
    (U_{x},\phi,p,z):=(U_{x},\phi\,;p_1,\ldots,p_{2n-2k},z_1,\ldots,z_{k})\,.
\end{equation}  
So, we have a nice description of coordinate transformations, as given in the following corollary. 
\begin{corollary}\label{cor:diffcharts}(\cite[Corollary 2.18]{pal24})
    Let $(X,\mc{J})$ be a regular GC manifold of type $k$. Let's assume that $(U,\phi,p,z)$ and $(U',\phi',p',z')$ are two local coordinate systems, as in \eqref{loc coordi}, with $U\cap U'\neq\emptyset\,.$ Then,
    $$\frac{\partial z'_{i}}{\partial\overline{z_{j}}}=\frac{\partial z'_{i}}{\partial p_{l}}=0\quad\text{for all}\quad i,j\in\{1,\ldots,k\}\,, l\in\{1,\ldots,2n-2k\}\,.$$
\end{corollary}
Furthermore, we have the following characterization of GH functions on a regular GC manifold in terms of the local coordinates in \eqref{loc coordi}.  
\begin{prop}(\cite[Proposition 2.19]{pal24})\label{prop GH}
Let $(X,\mc{J})$ be a regular GC manifold of type $k$. Then, $f: X\longrightarrow\C$ is a GH function if and only if at every point on $X$, expressed in terms of local coordinates, as shown in \eqref{loc coordi}, $f$ satisfies the following 
$$\frac{\partial f}{\partial\overline{z_{j}}}=\frac{\partial f}{\partial p_{l}}=0\quad\text{for all}\quad j\in\{1,\ldots,k\}\,, l\in\{1,\ldots,2n-2k\}\,.$$
\end{prop}

This implies that for a regular GC manifold of type $k$, the sheaf $\mathcal{O}_X$ is locally given by the ring of convergent power series in the coordinates $z=(z_1,\ldots,z_{k})$ in \eqref{loc coordi}.
\begin{remark}
    The GH maps are unaffected by $B$-field transformations, and GH homeomorphisms are GH maps according to \eqref{gh homeo}. So, when we say that an open neighborhood is GH homeomorphic to a product of open sets with a product GCS, as in Theorem \ref{darbu thm}, we always mean that they are GH homeomorphic up to a $B$-field transformation.
\end{remark}
\subsection{Generalized complex submanifold}
In this subsection, we revisit some fundamental concepts about generalized complex (GC) submanifolds. For most of the definitions, we shall follow \cite{ben04}.

\vspace{0.5em}
Let $V$ be a GC linear space and $L_{V}<(V\oplus V^{*})\otimes\C$ be its $+i$-eigenspace. Let $V'<V$ be a subspace. Consider the following subspace of $(V'\oplus V'^{*})\otimes\C$
\begin{equation}\label{ev}
    L_{V'}:=\bigg\{\rho(v)+\rho^{*}(v)|_{V'\otimes\C}\,\big |\,v\in L_{V}\cap(V'\oplus V^{*})\otimes\C\bigg\}
\end{equation}
where $\rho^{*}:(V\oplus V^{*})\otimes\C\longrightarrow V^{*}\otimes\C$ is the natural projection map. It is straightforward to compute that $\dim_{\C}L_{V'}=\dim_{\R}V'$, which implies that $L_{V'}$ is a maximal isotropic subspace.
\begin{definition}(\cite[Section 3]{ben04})\label{sub def}
    A subspace $V'<V$ is a called a \textit{generalized complex (GC) subspace} if $L_{V'}\cap\overline{L_{V'}}=0$ where $L_{V'}$ as defined in \eqref{ev}.
\end{definition}
Let $X$ be a GC manifold with the $+i$ eigenbundle $L_{X}$. Let $Y$ be a smooth embedded submanifold of $X$. Consider the natural projection maps 
\begin{align*}
    &\rho|_{Y}:(TX|_{Y}\oplus T^{*}X|_{Y})\otimes\C\longrightarrow TX|_{Y}\otimes\C\,,
    &\rho^{*}|_{Y}:(TX|_{Y}\oplus T^{*}X|_{Y})\otimes\C\longrightarrow T^{*}X|_{Y}\otimes\C\,,
\end{align*}
 and note that $TY<TX|_{Y}$ is a subbundle.
Using \eqref{ev} at each point of $Y$, we define the maximal isotropic distribution
\begin{equation}\label{ev mfld}
    L_{Y}:=\bigg\{\rho|_{Y}(v)+\rho^{*}|_{Y}(v)|_{TY\otimes\C}\,\big |\,v\in L_{X}|_{Y}\cap(TY\oplus T^{*}X|_{Y})\otimes\C\bigg\}
\end{equation}
In simpler terms, \eqref{ev mfld} implies that the subset $L_{Y}<(TY\oplus T^{\ast}Y) \otimes\C$ is a maximal isotropic subspace $L_{Y,y}<(T_{y}X\oplus T^{\ast}_{y}X) \otimes \C$ at each point $y\in Y$.
\begin{definition}(\cite[Section 5]{ben04})\label{sub def mfld}
    A submanifold $Y\subset X$ is called a \textit{generalized complex (GC) submanifold} if the following holds:
    \begin{itemize}
   \setlength\itemsep{0.4em} 
        \item $L_{Y}<(TY\oplus T^{\ast}Y) \otimes\C$ is a subbundle.
        \item $L_{Y}\cap\overline{L_{Y}}=0$.
        \item $L_{Y}$ is Courant bracket (cf. \eqref{bracket}) involutive.
    \end{itemize}
Here $L_{Y}$ is defined as in \eqref{ev mfld}. 
\end{definition}
\begin{remark}\label{rmk sub}
For a regular GC manifold $X$, note that $L_{X}$ is given by $L_{X} = L(E, \sigma)$ (cf. \eqref{isotropic set}), where $E := \rho(L_{X})$ is a subbundle of $TX \otimes \C$ and $\sigma \in \wedge^2 E^{*}$. The structure $L_{Y}$ can then be described pointwise as $$L_{Y}:=L(E_{Y},\sigma|_{Y})\,,$$ where $E_{Y}=\rho|_{Y}(L_{Y})=(TY\otimes\C)\cap E|_{Y}$ forms a Lie involutive distribution of $TY\otimes\C$. Consequently, if the GCS induced by $L{Y}$ on $Y$ is regular, the set $E_{Y}$ becomes a subbundle.
\end{remark}

\section{GC Stein Manifold}\label{gc stein}
Let $X$ be a regular GC manifold with $\mathcal{O}:=\mathcal{O}_{X}$ as the sheaf of generalized holomorphic (GH) functions on $X$. For $M,N\in\N$, let $\omega_0$ be the standard symplectic form on $\R^{2M}\,,$ and let $J_0$ be the standard complex structure on $\R^{2N}\,,$ that is $(\R^{2N},J_0)=\C^{N}\,.$ We consider some simple examples of $\mathcal{O}_{X}$ (cf. \cite[Section 2]{pal24}).

\vspace{0.2em}
\begin{enumerate}
\setlength\itemsep{0.4em}
    \item When $(X,J_{X})$ is a complex manifold with $J_{X}$ as its complex structure. Consider the natural GCS induced by $J_{X}$, as given in Example \ref{complx eg}. We can see that, given any GH map $\psi$, $d\psi\in\Omega^{1,0}(X)$ that is, $\psi$ is a holomorphic function. So $\mathcal{O}$ will be the sheaf of holomorphic functions on $X$.
\item When $(X,\omega)$ is a symplectic manifold with a symplectic structure $\omega$. Consider the induced GCS, as given in Example \ref{symplectic eg}, with the $+i$-eigenbundle $L_{X}$. Note that, $L_{X}$ is naturally identified with $TX\otimes\C$. So, for any GH map $\psi$, $d\psi=0$ which implies $\psi$ is locally constant sheaf. Hence $\mathcal{O}$ is a sheaf of locally constant $\C$-valued functions. 
\end{enumerate}
\begin{definition}\label{def}
 Let $K\subseteq X$ be a compact set. 
 \begin{enumerate}
 \setlength\itemsep{0.4em}
     \item  We associate its $\mathcal{O}(X)$-convex hull, also called $\mathcal{O}(X)$-hull, as follows
     \begin{equation}\label{cnvx}
         \widehat{K}_{\mathcal{O}(X)}:=\left\{p\in X\,\bigg|\,\sup_{x\in K}|f(x)|\geq|f(p)|\,\,\forall\,\,f\in\mathcal{O}(X)\right\}\,.
     \end{equation}
 \item $K$ is called $\mathcal{O}(X)$-convex if $K=\widehat{K}_{\mathcal{O}(X)}\,.$
 \end{enumerate}
\end{definition}
\vspace{0.2em}
\begin{prop}
 ~
  \begin{enumerate}
  \setlength\itemsep{0.4em}
      \item If $\widehat{K}_{\mathcal{O}(X)}$ is compact, then the $\mathcal{O}(X)$-hull of $\widehat{K}_{\mathcal{O}(X)}$ is 
      $\widehat{K}_{\mathcal{O}(X)}$ itself.
      \item For any two compact subsets $K_1\subseteq K_2$\,, $\widehat{K_1}_{\mathcal{O}(X)}\subseteq \widehat{K_2}_{\mathcal{O}(X)}\,.$
  \end{enumerate}  
\end{prop}
\begin{proof}
    Follows from Definition \ref{def}.
\end{proof}
\vspace{0.2em}
\begin{definition}
~
\begin{enumerate}
\setlength\itemsep{0.4em}
   \item  $X$ is called \textit{generalized holomorphically convex (in short, GH convex)} if for every compact set $K\subseteq X\,,$ its $\mathcal{O}(X)$-hull $\widehat{K}_{\mathcal{O}(X)}$ is also compact.
   \item $X$ is called \textit{generalized holomorphically regular (in short, GH regular)} with $\dim_{\R}X=2n$ and $\typ(X)=k$ if for every point $x\in X$, there exist a collection of GH maps $\{f_1,\ldots,f_{n-k},\,g_1,\ldots,g_{k}\}$ such that $\{f_1,\ldots,f_{n-k}\}$ are GH maps from $X$ to $(\R^2,\omega_0)$ and $\{g_1,\ldots,g_{k}\}\in\mathcal{O}(X)$ and satisfies the following

        \vspace{0.2em}
        \begin{itemize}
        \setlength\itemsep{0.4em}
            \item $\{df_1|_{x},\ldots,df_{n-k}|_{x},\,dg_1|_{x},\ldots,dg_{k}|_{x}\}$ are $\R$-linearly independent at $x$.
            \item $\{dg_1|_{x},\ldots,dg_{k}|_{x}\}$ are $\C$-linearly independent at $x$.
        \end{itemize}
\end{enumerate}
\end{definition}
\begin{example}
Any compact regular GC manifold is GH convex.
\end{example}
\begin{definition}\label{main def}
    Let $X^{2n}$ be a regular GC manifold with type $k\in\{1,\ldots,n\}\,.$ Then, $X$ is called a \textit{generalized complex (GC) Stein manifold} if the following conditions hold:

    \vspace{0.2em}
    \begin{enumerate}
     \setlength\itemsep{0.4em}
        \item (GH separability) For every pair of points $x\neq y$ in $X$, there exist a GH function $f\in\mathcal{O}(X)$ such that $f(x)\neq f(y)\,.$ 
        \item (GH convexity) $X$ is GH convex.
        \item (GH regularity) $X$ is GH regular.
    \end{enumerate}
\end{definition}

\begin{remark}\label{rmk stein}
~
\begin{itemize}
 \setlength\itemsep{0.4em}
    \item Condition $(2)$ implies that $X$ admits an exhaustion $K_1\subset K_2\subset\cdots\subset\bigcup^{\infty}_{j=1}K_{j}=X$ by compact $\mathcal{O}(X)$-convex subsets such that $K_{j}\subset K_{j+1}^{o}$ holds for every $j\geq 1$. Here $K_{j}^{o}$ denotes the interior of $K_{j}$.
    \item Condition $(3)$ implies that global GH maps provide local coordinate charts at each point of $X$. 
\end{itemize}
\end{remark}
\begin{example}\label{e1}
 Any Stein manifold is a GC Stein manifold. More generally, given an open set $\Omega_1\subseteq(\R^{2M},\omega_0)$ and a Stein submanifold $\Omega_2\subseteq\C^{N}$, $\Omega_1\times\Omega_2\subseteq\R^{2M}\times\C^{N}$ is a GC Stein manifold.
\end{example}
\begin{remark}
    Observe that for $\mathcal{O}_{X}$ turns into a sheaf of locally constant $\C$-valued functions. Consequently, any GC manifold of type $0$ cannot be a GC Stein manifold. On the other hand, when $k=\frac{1}{2}\dim_{\R}X$,  $\mathcal{O}_{X}$ becomes a sheaf of holomorphic functions, implying that Definition \ref{main def} recovers the classical Stein manifolds in this case.
\end{remark}
\begin{remark}
Since GH maps remain unaffected by any $B$-field transformation, properties such as GH separability, GH convexity, and GH regularity also remain unchanged under a $B$-field transformation. In other words, applying a $B$-field transformation to the GCS of a GC Stein manifold results in another GC Stein manifold.
\end{remark}
\begin{example}\label{e2}
   Consider a Lie group $G$ with its real Lie algebra denoted by $\mathfrak{g}$. Assume that $G$ is equipped with a GCS $\mathcal{J}$ of type $k>0$ and that the exponential map $\exp:\mathfrak{g}\longrightarrow G$ is a diffeomorphism. Then, there exists a GCS $\mathcal{J}'$ of type $k$ on $\mathfrak{g}$ such that the map 
    $$\exp:(\mathfrak{g},\mathcal{J}')\longrightarrow (G,\mathcal{J})$$
    becomes a GH homeomorphism. By applying \cite[Theorem 4.13]{Gua}, it follows that $(\mathfrak{g},\mathcal{J}')$ is GH homeomorphic to $\R^{2n-2k}\times\C^{k}$, up to a $B$-field transformation. Consequently, $(G,\mathcal{J})$ is a GC Stein manifold.
\end{example}
\begin{example}
    Any simply connected nilpotent Lie group admitting a left-invariant GCS (cf. \cite{cav04}) of positive type is a GC Stein manifold.
\end{example}
\begin{remark}\label{foli rmk}
Any regular GC manifold $X$ of type $k>0$ admits a symplectic foliation of complex codimension $k$, which is transversely holomorphic (cf. \cite[Proposition 3.11]{Gua}, \cite[Corollary 3.22]{Gua2}). Let $\widetilde{X}$ denote the leaf space of this foliation. When $X$ is a GC Stein manifold and $\widetilde{X}$ is a smooth manifold, by \cite[Theorem 13.4]{pal24} and Definition \ref{main def}, it follows that $\widetilde{X}$ becomes a Stein manifold. 
\end{remark}
\begin{remark}
It might seem natural to assume that if the leaf space $\widetilde{X}$, in Remark \ref{foli rmk}, is a Stein manifold, then $X$ would naturally inherit the structure of a GC Stein manifold. However, this assumption does not hold true. While $\widetilde{X}$ being a Stein manifold does ensure, by \cite[Theorem 13.4]{pal24}, that $X$ satisfies both GH convexity and GH separability as described in Definition \ref{main def}, it does not guarantee GH regularity. Even in the simplest scenarios, $X$ may fail to exhibit GH regularity. The following example demonstrates this.
\end{remark}
\begin{example}\label{egg1}
    Let $X=X_1\times X_2$, where $X_1$ is a compact connected symplectic manifold of even dimension, and $X_2$ is a Stein manifold. Then $X$ admits the product GCS of type $k$, where $k=\dim_{\C}X_2$, and so $\widetilde{X}=X_2$. By \cite[Theorem 13.4]{pal24}, $X$ satisfies both GH convexity and GH separability.  

    \vspace{0.2em}
    Now if possible let $X$ be GH regular. By Definition \ref{main def}, there exists GH maps $f_1,\ldots,f_{m}$ from $X$ to $\R^2$, where $2m=\dim_{R}X_1$. Consider the GH map 
    $$f=(f_1,\ldots,f_{m}):X\longrightarrow\R^{2m}\,.$$ The natural inclusion map $i:X_1\hookrightarrow X$, defined by $i(x)=(x,y_0)$ for a fixed $y_0\in X_2$, is a GH map. Then the composition $f\circ i:X_1\longrightarrow\R^{2m}$ is also a GH map. Consequently, $f\circ i$ is also a Poisson map (cf. Remark \ref{poi rmk}), and thus it becomes a submersion. Since $X_1$ is compact and $f\circ i$ is open map, it follows that $\img(f\circ i)=\R^{2m}$. However, $\img(f\circ i)$ is also compact, leading to a contradiction. Therefore, $X$ is not GH regular.
\end{example}
 
\subsection{Fr\'{e}chet topology of a coherent sheaf of $\mathcal{O}$-modules}
Suppose $g=(g_1,\ldots,g_{m})$ is a $m$-tuple of $\R^2$-valued smooth functions, defined on a set $K\subseteq X$. Define
 \begin{equation}\label{k-norm}
     ||g||_{K}:=\sup_{x\in K}\big\{|g_{j}(x)|\,|\,1\leq j\leq m\big\}\,.
 \end{equation}
Let $\mathcal{O}(X)$ be the algebra of GH functions on $X$. We use \eqref{k-norm} to topologize $\mathcal{O}(X)$ in the following way: a basis of open sets is provided by the collection: for any compact set $K\subset X$, $f\in\mathcal{O}(X)$ and $\epsilon>0$,
$$U(f,\epsilon,K):=\big\{g\in\mathcal{O}(X)\,\big |\,||g-f||_{K}<\epsilon\big\}\,.$$ This topology, usually called \textit{topology of uniform convergence} on compact subsets, makes $\mathcal{O}(X)$ into a topological complex vector space. Note that the topology of uniform convergence on compact subsets is the same as the compact-open topology since the codomain is a metric space. Additionally, because $X$ is separable, this topology is metrizable.
\vspace{0.5em}

To show that $\mathcal{O}(X)$ is a Fr\'{e}chet space, it is enough to show that $\mathcal{O}(X)$ is a complete metric space. Since $X$ is regular, for $x\in X$, there exists a coordinate chart $\{U,\phi\}$ around $x$ such that
$$\phi: U\longrightarrow (V_1\times V_2)\subset_{\text{open}}\R^{2n-2k}\times\C^{k}\, ,$$ is a GH homeomorphism (see Theorem \ref{darbu thm}), i.e, $\phi$ is a diffeomorphism, GH map and it respects the GCS of both sides upto a $B$-field transformation. Without loss of generality, we can take $V_1$ to be connected. Since GH maps are insensitive to $B$-field transformations, we observe that $$\mathcal{O}_{U}\cong_{\phi}\mathcal{O}_{V_1\times V_2}=\pr_2^{-1}(\mathcal{O}_{V_2})$$ where $\pr_2:\R^{2n-2k}\times\C^{k}\longrightarrow\C^{k}$ is the natural projection map. 

\vspace{0.5em}
Let $\{f_{j}\}$ be a Cauchy sequence in $\mathcal{O}(X)$. Then, for each $j\geq 1$, $\exists!\,\,g_{j}\in\mathcal{O}_{V_2}(V_2)$ such that $f_{j}\circ\phi^{-1}=\pr_2^{*}(g_{j})$ on $V_1\times V_2$. Since $\{f_{j}\}$ are Cauchy in $\mathcal{O}_{U}(U)$ and $\phi$ is GH homeomorphism, $\{g_{j}\}$ are also Cauchy in $\mathcal{O}_{V_2}(V_2)$. By \cite[Theorem 5, p. 158]{rossi65}, $\mathcal{O}_{V_2}(V_2)$ is already a Fr\'{e}chet space and so, there exist $g\in\mathcal{O}_{V_2}(V_2)$ such that $g_{j}\to g$ in uniform convergence topology. This implies $f_{j}\circ\phi^{-1}\to\pr_2^{*}(g)\in\mathcal{O}_{V_1\times V_2}(V_1\times V_2)$. As $\phi$ is a GH homeomorphism, $\pr_2^{*}(g)\circ\phi\in\mathcal{O}_{U}(U)$ and $f_{j}\to\pr_2^{*}(g)\circ\phi\,.$ Now as $\{f_{j}\}$ is also a Cauchy sequence of continuous functions on $X$, it has a continuous limit, namely 
   $$\tilde{f}=\lim_{j\to\infty}f_{j}\quad(\text{in uniform convergence on compact subsets})\,.$$ Since $C_{U}(U)$ is Hausdroff with respect to the topology of uniform convergence on compact subsets, we have $$\tilde{f}=\pr_2^{*}(g)\circ\phi\quad\text{on $U$}\,.$$ Here $C_{X}$ denotes the sheaf of germs of continuous $\C$-valued functions on $X$. This shows that $\tilde{f}\in\mathcal{O}_{x}$ and as $x$ is arbitrary, $\tilde{f}\in\mathcal{O}(X)$. Set
   \begin{equation}\label{ox}  \mathcal{O}^{m}:=\bigoplus_{m}\,\mathcal{O}\,,\quad\text{and}\quad\mathcal{O}^{m}(X):=\bigoplus_{m}\,\mathcal{O}(X)\quad\text{for any $m\in\N$}\,.
   \end{equation} Thus, we proved the following.
\begin{theorem}\label{f spc}
    For any regular GC manifold $X$ and $m\in\N$, $\mathcal{O}^{m}(X)$ is a Fr\'{e}chet space with respect to the topology of uniform convergence on compact subsets. In particular, $\mathcal{O}^{m}(X)$ is complete metric space.
\end{theorem}
\begin{corollary}
   $\mathcal{O}^{m}(X)$  $(m\in\N)$ is a Baire space.
\end{corollary}
\begin{proof}
    Follows from Baire category theorem and Theorem \ref{f spc}.
\end{proof}
Let $\mathcal{F}$ be a coherent sheaf of $\mathcal{O}$-modules on $X$ and $K\subset X$ is a compact set. Suppose, for some $m\in\N$, $\psi:\mathcal{O}^{m}\longrightarrow\mathcal{F}$ is a sheaf epimorphism such that $\psi$ induces an epimorphism $$\tilde{\psi}:\mathcal{O}^{m}(X)\longrightarrow\mathcal{F}(X)\,.$$ For $f\in\mathcal{F}(X)$, define
    \begin{equation}\label{k-norm1}
     ||f||^{\psi}_{K}:=\inf\big\{||g||_{K}\,\big |\,g\in\mathcal{O}^{m}(X)\,,\,\tilde{\psi}(g)=f\big\}\,. 
    \end{equation}
where $||g||_{K}$ as defined in \eqref{k-norm}. In a similar manner, we can topologize $\mathcal{F}(X)$ with respect to the norms $\big\{||\cdot||^{\psi}_{K}\,\big |\,K\subset X\,\,\text{is compact set}\,\big\}\,.$
Let $\ker(\psi):=\ker(\psi:\mathcal{O}^{m}\longrightarrow\mathcal{F})$ be the kernel subsheaf of $\psi$. Then $\ker(\psi)(X)$ is a closed subspace of $\mathcal{O}^{m}(X)$ and so, it is also a Fr\'{e}chet space. Now the surjectivity of $\tilde{\psi}$ implies that the induced quotient map  
$$\tilde{\tilde{\psi}}:\mathcal{O}^{m}(X)/\ker(\psi)(X)\longrightarrow\mathcal{F}(X)$$ is a homeomorphism where $\mathcal{O}^{m}(X)/\ker(\psi)(X)$ is with quotient topology which is a Fr\'{e}chet space topology. So we prove the following
\begin{theorem}\label{imp thm}
Given any coherent sheaf $\mathcal{F}$ of $\mathcal{O}$-modules on $X$, $\mathcal{F}(X)$ is a Fr\'{e}chet space with respect to the topology defined by the norms $\big\{||\cdot||^{\psi}_{K}\,\big |\, K\subset X\,\,\text{is compact set}\,\big\}\,,$ assuming that the sheaf epimorphism $\psi:\mathcal{O}^{m}\longrightarrow\mathcal{F}$ induces an epimorphism $\tilde{\psi}:\mathcal{O}^{m}(X)\longrightarrow\mathcal{F}(X)$ between global sections.
\end{theorem}
\begin{remark}
    Note that the Fr\'{e}chet space topology of $\mathcal{F}$, is independent of the choice of $\psi$ and hence, is canonical.
\end{remark}
\begin{remark}\label{imp rmk}
For any regular GC manifold $X$ of non-zero type, $\mathcal{O}$ is a coherent sheaf as $\mathcal{O}$-modules. Because at any point $x\in X$, using Theorem \ref{darbu thm}, there exists a coordinate chart $\{U,\phi\}$ around $x$ such that
$\phi: U\longrightarrow (V_1\times V_2)\subset_{\text{open}}\R^{2n-2k}\times\C^{k}\, ,$ is a GH homeomorphism. Without loss of generality, let $V_1$ be connected and this imply that, $\mathcal{O}_{U}\cong\pr_2^{-1}\mathcal{O}_{V_2}\,.$ As $\mathcal{O}_{V_2}$ is coherent as $\mathcal{O}_{V_2}$-modules, $\pr_2^{-1}\mathcal{O}_{V_2}$ is also coherent as $\pr_2^{-1}\mathcal{O}_{V_2}$-modules.
\end{remark}

\section{GC Stein manifold and Cartan's Theorems}\label{cartan}
In this section, we provide the GC geometric analog of Cartan's Theorems for GC Stein manifolds.

\subsection{Cartan's theorems for certain domains in $\R^{2M}\times\C^{N}$}
Suppose $\{\phi_{m}\}^{l}_{m=1}$ are $\R$-valued smooth functions on $\C^{N}$ satisfying, for all $i,j\in\{1,\ldots,N\}$ and $z\in\C^{N}$
\begin{equation}\label{*}
    \left|\frac{\partial\phi_{m}}{\partial z_{i}\partial z_{j}}(z)\right|<(6N^2)^{-1}\quad\text{and}\quad\left|\frac{\partial\phi_{m}}{\partial z_{i}\partial\overline{z}_{j}}(z)-\delta_{ij}\right|<(3N^2)^{-1}
\end{equation}
where $\delta_{ij}$ is the Kronecker delta. Define 
\begin{equation}\label{D def}
 D:=D^{'}\times D^{''}   
\end{equation}
where $D^{'}$ is any connected open set in $\R^{2M}$ and $D^{''}=\{z\in\C^{N}\,|\,\phi_{m}<0\,\,,1\leq m\leq l\}\,.$ Then $D$ admits the natural product GCS induced from the product GCS of $\R^{2M}\times\C^{N}$. The leaf space of the symplectic foliation induced by the product GCS of $D$ is just $D^{''}$ which is a complex submanifold of $\C^{N}$. So, we have $\mathcal{O}_{D}=\pr_2^{-1}(\mathcal{O}_{D^{''}})$ (cf. \cite[Theorem 13.4]{pal24}).
\begin{prop}\label{imp prop}
Assume $D$, as defined in \eqref{D def}, is either bounded or only $D^{''}$ is bounded. Then, for $p\geq 1$ and $m\in\N$, we have $$H^{p}(D,\mathcal{O}^{m}_{D})=0\,.$$
\end{prop}
\begin{proof}
We proceed by induction on $m$. Assume $m=1$. Let $\{U_{\alpha}=A_{\alpha}\times B_{\alpha}\}$ be a covering of $D$ such that each $A_{\alpha}\times B_{\alpha}$ is GH homeomorphic to the product of ball $B^{'}\times B^{''}$ in $\R^{2M}\times\C^{N}$ (see Theorem \ref{darbu thm}) and any finite intersection 
$$\bigcap_{\alpha<+\infty}U_{\alpha}=\bigcap_{\alpha<+\infty}A_{\alpha}\times\bigcap_{\alpha<+\infty}B_{\alpha}\,,$$ is GH homeomorphic to the product of open contractible sets $A^{'}\times A^{''}$ in $\R^{2M}\times\C^{N}$. Then, for $p\geq 0$, we get the following diagram: 
\[\begin{tikzcd}[ampersand replacement=\&]
	{H^{p}(\bigcap_{\alpha<+\infty}U_{\alpha},\mathcal{O}_{D})} \& {H^{p}(\bigcap_{\alpha<+\infty}B_{\alpha},\mathcal{O}_{D^{''}})} \\
	{H^{p}(A^{'}\times A^{''},\mathcal{O}_{A^{'}\times A^{''}})} \& {H^{p}(A^{''},\mathcal{O}_{A^{''}})}
	\arrow[Rightarrow, no head, from=1-1, to=1-2]
	\arrow[Rightarrow, no head, from=2-1, to=2-2]
	\arrow["\cong"{marking, allow upside down}, draw=none, from=2-2, to=1-2]
	\arrow["\cong"{marking, allow upside down}, draw=none, from=2-1, to=1-1]
\end{tikzcd}\]
The two rows of the above diagram are the same because $\mathcal{O}_{U}=\pr_2^{-1}(\mathcal{O}_{B})$ and $\mathcal{O}_{A^{'}\times A^{''}}=\pr_2^{-1}(\mathcal{O}_{A^{''}})$ (cf. \cite[Theorem 13.4]{pal24}) where $U=\bigcap_{\alpha<+\infty}U_{\alpha}$ and $B=\bigcap_{\alpha<+\infty}B_{\alpha}$.  Since $H^{p}(A^{''},\mathcal{O}_{A^{''}})=0$ for $p>0$, we have $$H^{p}(\bigcap_{\alpha<+\infty}U_{\alpha},\mathcal{O}_{D})=H^{p}(\bigcap_{\alpha<+\infty}B_{\alpha},\mathcal{O}_{D^{''}})=0\quad\text{for $p>0$}\,.$$ Using Leray's theorem and the fact that $\mathcal{O}_{D}=\pr_2^{-1}(\mathcal{O}_{D^{''}})$, we can show that, for $p\geq 0$  
\[\begin{tikzcd}[ampersand replacement=\&]
	{H^{p}(\{U_{\alpha}\},\mathcal{O}_{D})} \& {H^{p}(\{B_{\alpha}\},\mathcal{O}_{D^{''}})} \\
	{H^{p}(D,\mathcal{O}_{D})} \& {H^{p}(D^{''},\mathcal{O}_{D^{''}})}
	\arrow[Rightarrow, no head, from=1-1, to=1-2]
	\arrow[Rightarrow, no head, from=1-1, to=2-1]
	\arrow[Rightarrow, no head, from=1-2, to=2-2]
\end{tikzcd}\]
Due to the assumption on $D$, $D^{''}$ is bounded. Therefore, using \cite[Proposition 1]{siu68}, $H^{p}(D,\mathcal{O}_{D})$ vanishes for $p>0$. To finish the induction on $m$, we use the long exact sequence of cohomology associated with the following short exact sequence of sheaves 
\[\begin{tikzcd}[ampersand replacement=\&]
	0 \& {\mathcal{O}_{D}} \& {\mathcal{O}^{m}_{D}} \& {\mathcal{O}^{m-1}_{D}} \& {0\,.}
	\arrow[from=1-2, to=1-3]
	\arrow[from=1-3, to=1-4]
	\arrow[from=1-1, to=1-2]
	\arrow[from=1-4, to=1-5]
\end{tikzcd}\] Hence, this concludes the proof.
\end{proof}
\begin{corollary}\label{imp cor}
    For any coherent sheaf $\mathcal{F}$ of $\mathcal{O}_{D}$-modules on $D$ admitting a finite free resolution, $H^{p}(D,\mathcal{F})=0$ for $p>0$. 
\end{corollary}
\begin{proof}
    Follows from Proposition \ref{imp prop}.
\end{proof}
Consider the following open subspace of $\R^{2M}\times\C^{N}$,
\begin{equation}\label{br}
  B_{r}:=D^{'}\times\D_{r}  
\end{equation} where $D^{'}\subset\R^{2M}$ is an open ball and $\D_{r}\subset\C^{N}$ is the $r$-ball. Let $G$ be an open neighborhood of the closure of $B_{r}$. Let $\mathcal{F}$ be a coherent sheaf of $\mathcal{O}_{G}$-modules on $G$.
\begin{theorem}\label{imp claim}
$H^{p}(B_{r},\mathcal{F})=0$ for all $p\geq 1$.
\end{theorem}

\begin{theorem}\label{thm}
Assuming Theorem \ref{imp claim}, $\mathcal{F}$ is generated on $B_{r}$ by $H^0(B_{r},\mathcal{F})$.
\end{theorem}
\begin{proof}
Let $\mathcal{I}$ be the ideal sheaf of germs of GH functions vanishing at $x\in B_{r}$. Since $\mathcal{F}$ is coherent, we have an epimorphism 
$$\psi:\mathcal{O}^{m}_{B_{r},x}\longrightarrow\mathcal{F}_{x}\quad\text{for some $m\in\N$}\,.$$ Let $e_{j}:=(0,\ldots,0,1,0,\ldots,0)\in\mathcal{O}^{m}_{B_{r},x}$ where the $1$ is in the $j$-th position, $1\leq j\leq m$. Consider the exact sequence of sheaf homomorphism
\begin{equation}\label{eq1}
 \begin{tikzcd}
	0 & {\mathcal{I}\mathcal{F}} & {\mathcal{F}} & {\mathcal{F}/\mathcal{I}\mathcal{F}} & 0
	\arrow[from=1-2, to=1-3]
	\arrow["\eta", from=1-3, to=1-4]
	\arrow[from=1-4, to=1-5]
	\arrow[from=1-1, to=1-2]
\end{tikzcd}   
\end{equation}
Since $(\mathcal{F}/\mathcal{I}\mathcal{F})_{y}=0$ for $y\neq x$, $\eta\circ\psi(e_{j})\in(\mathcal{F}/\mathcal{I}\mathcal{F})_{x}$ defines a section $s_{j}\in(\mathcal{F}/\mathcal{I}\mathcal{F})(B_{r})$ for all $1\leq j\leq m$. Note that, due to the coherence of $\mathcal{F}$ and $\mathcal{I}$ (by Remark \ref{imp rmk}), $\mathcal{I}\mathcal{F}$ is also a coherent sheaf of $\mathcal{O}_{G}$-modules. Therefore,  by the assumption and using the long exact cohomology sequence of \eqref{eq1}, we conclude that $\eta$ induces an epimorphism
$$\tilde{\eta}:\mathcal{F}(B_{r})\longrightarrow(\mathcal{F}/\mathcal{I}\mathcal{F})(B_{r})\,.$$ So, for each $j$, there exists $f_{j}\in\mathcal{F}(B_{r})$ such that $\tilde{\eta}(f_{j})=s_{j}\,.$ Now, 
$$\eta(\psi(e_{j})-f_{j,x})=s_{j,x}-s_{j,x}=0\,,$$ which implies $\psi(e_{j})-f_{j,x}\in(\mathcal{I}\mathcal{F})_{x}\,.$ Note that,
$$(\mathcal{I}\mathcal{F})_{x}=\mathcal{I}_{x}\,\mathcal{F}_{x}=\mathcal{I}_{x}\,\psi(\mathcal{O}^{m}_{B_{r},x})=\psi(\mathcal{I}_{x}\,\mathcal{O}^{m}_{B_{r},x})\,.$$ So, for some $g_{j}=(g_{j1},\ldots,g_{jm})\in\mathcal{I}_{x}\,\mathcal{O}^{m}_{B_{r},x}\,,$ we have $$\psi(e_{j})-f_{j,x}=\psi(g_{j})\quad\text{for all $1\leq j\leq m$}\,.$$ Since $\det(\delta_{ij}-g_{ij})_{1\leq i,j\leq m}$ is a unit in $\mathcal{O}_{B_{r},x}\,,$ $\{e_{j}-g_{j}\}_{j=1}^{m}$ generates $\mathcal{O}^{m}_{B_{r},x}\,.$
Let $u\in\mathcal{F}_{x}\,.$ By surjectivity of $\psi$, there exists a $v\in\mathcal{O}^{m}_{B_{r},x}$ such that $\psi(v)=u\,.$ Now $v=\sum_{j}\lambda_{j}(e_{j}-g_{j})$ for some $\lambda_{j}\in\mathcal{O}_{B_{r},x}\,.$ Then 
$$u=\psi(\sum_{j}\lambda_{j}(e_{j}-g_{j}))=\sum_{j}\lambda_{j}\psi(e_{j}-g_{j})=\sum_{j}\lambda_{j}f_{j,x}\,.$$ Hence, for any $x\in B_{r}\,,$ $\mathcal{F}_{x}$ is generated by $H^0(B_{r},\mathcal{F})$ and this concludes the proof.
\end{proof}
\subsection{Cartan’s theorems for GC Stein manifolds}
In this subsection, we prove Cartan's Theorems A and B for GC Stein manifold. Let $X^{2n}$ be a GC Stein manifold of type $k>0$.
\begin{prop}\label{prop}
    Let $G$ be a relatively compact open subset of $X$. Then, there exists an injective GH immersion (cf. Definition \ref{GH map}) of\,\,$G$ into\,\, $\R^{2M}\times\C^{N}$ for $N>0$ large enough and $M=\frac{1}{2}(\dim_{\R}X-\typ(X))$. In particular, there exists a GH embedding, that is, GH map and smooth embedding, of\,\,$G$ into\,\, $\R^{2M}\times\C^{N}\,.$ Furthermore, the image is contained in some $D^{'}\times\Delta$ where $D^{'}\subset\R^{2M}$ is a suitable connected open set and $\Delta\subset\C^{N}$ is a polydisc of appropriate radius.
\end{prop}
\begin{proof}
 Let $X$ be endowed with an atlas $\{U_{\alpha},\phi_{\alpha}\}$, consisting of globally defined local charts that are GC manifolds. By Theorem \ref{darbu thm}, each map $\phi_{\alpha}: U_{\alpha}\longrightarrow U_{1\alpha}\times U_{2\alpha}$  is of the form
 $$\phi_{\alpha}=\,(f_{1\alpha},\ldots,f_{(n-k)\alpha}\,,\,g_{1\alpha},\ldots,g_{k\alpha})\,,$$ where
 $U_{1\alpha}\subset(\R^{2n-2k},\omega_0)$ and $U_{2\alpha}\subset\C^{k}$ are open subsets, $f_{j\alpha}:U_{\alpha}\longrightarrow\R^2$ $(j=1,\ldots,n-k)$, and $g_{l\alpha}:U_{\alpha}\longrightarrow\C$ $(l=1,\ldots,k)$. Due to GH regularity, we can take 
$f_{j\alpha}:U_{\alpha}\longrightarrow\R^2$ as a globally defined GH map $f_{j\alpha}:X\longrightarrow\R^2$. As noted in Remark \ref{poi rmk}, $f_{j\alpha}$ is a submersion, implying that the rank of the real Jacobian of $\big\{f_{1\alpha},\ldots,f_{(n-k)\alpha}\}$ at each point of $X$ is $2n-2k$. Consequently, we can replace $\big\{f_{1\beta},\ldots,f_{(n-k)\beta}\}$ with $\big\{f_{1\alpha},\ldots,f_{(n-k)\alpha}\}$ in $\phi_{\beta}$, after sufficiently shrinking $U_{\beta}$. In other words, $X$ can be endowed with an atlas $\{U_{\alpha},\phi_{\alpha}\}$, where the charts are GC manifolds, and $$\phi_{\alpha}=\,(f_1,\ldots,f_{n-k}\,,\,g_{1\alpha},\ldots,g_{k\alpha})\quad\text{is globally defined}\,.$$

 By compactness of $\overline{G}$, we can get a finite covering, from the atlas, $\{U_{\alpha},\phi_{\alpha}\}^{l}_{\alpha=1}$ of $\overline{G}$. Fix $x\in\overline{G}$, by GH separability, for all $y\in\overline{G}-\{x\}$, there exist a GH function $f_{xy}\in\mathcal{O}(X)$ such that $f_{xy}(x)\neq f_{xy}(y)$. By continuity of $f_{xy}$, there exist open neighborhoods $x\in A_{xy}$ and $y\in B_{xy}$ such that $f_{xy}(A_{xy})\cap f_{xy}(B_{xy})=\emptyset\,.$
 \vspace{0.5em}
 
 Suppose $A_{xy}\subset U_{\alpha}$ and $B_{xy}\subset U_{\beta}$ for some $\alpha,\beta\in\{1,\ldots,l\}$. We can do this by sufficiently shrinking $A_{xy}$ and $B_{xy}$. The collection $\{A_{xy}\cup B_{xy}\}_{y\in\overline{G}}$ is a covering of $\overline{G}$. Again, we can extract a finite covering $\{A_{xy_{j}}\cup B_{xy_{j}}\}^{m(x)}_{j=1}$. Set $$A_{x}:=\bigcap^{m(x)}_{j=1}A_{xy_{j}}\,.$$ The family $\{A_{x}\}_{x\in\overline{G}}$ is an open covering of $\overline{G}$. We again extract a finite covering $\{A_{x_{j}}\}^{q}_{j=1}\,.$ Define, for all $x\in\overline{G}$,
 \begin{equation}\label{stein map}
 \begin{aligned}
   &F:\overline{G}\longrightarrow\R^{(2n-2k)}\times\C^{\big(lk+\sum^{q}_{j=1}\,m(x_{j})\big)}\,;\\
   &F(x)=\bigg(f_1(x),\ldots,f_{n-k}(x)\,,\,g_{11}(x),\ldots,g_{k1}(x),\ldots,g_{1l}(x),\ldots,g_{kl}(x)\,,\\
   &\quad\quad\quad\quad\,f_{x_1y_1}(x),\ldots,f_{x_1y_{m(x_1)}}(x),\ldots\ldots,f_{x_{q}y_1}(x),\ldots,f_{x_{q}y_{m(x_{q})}}(x)\bigg)\,.\\ 
 \end{aligned}
 \end{equation}
One can see that $F|_{G}$ is a GH map as its coordinate maps are GH maps. Next, we need to show that $F$ is injective.

\vspace{0.5em}
To do this, assume $z\neq w\in\overline{G}$ and suppose $F(z)=F(w)$. Since $\{A_{x_{j}}\}^{q}_{j=1}$ is a covering of $\overline{G}$, there exists $\lambda\in\{1,\ldots,q\}$ such that $z\in A_{x_{\lambda}}=\bigcap^{m(x_{\lambda})}_{j=1}A_{x_{\lambda}y_{j}}\,.$ If $w\in A_{x_{\lambda}y_{j}}$ for some $j\in\{1,\ldots,m(x_{\lambda})\}\,,$ then $w,z$ are in the same chart $U_{\alpha}$ for some $\alpha\in\{1,\ldots,l\}$. This implies $\phi_{\alpha}(z)\neq\phi_{\alpha}(w)$ which is a contradiction as $F(z)=F(w)$.
As a consequence, $w\in B_{x_{\lambda}y_{j}}$ for some $j\in\{1,\ldots,m(x_{\lambda})\}$ as $\{A_{xy_{j}}\cup B_{xy_{j}}\}^{m(x)}_{j=1}$ is an open covering of $\overline{G}$. Hence $f_{x_{\lambda}y_{j}}(z)\neq f_{x_{\lambda}y_{j}}(w)$ which is again a contradiction as $F(z)=F(w)$. Thus $F$ is injective.

\vspace{0.5em}
Set $N=lk+\sum^{q}_{j=1}\,m(x_{j})$. By continuity, $F(G)$ is relatively compact in $\R^{(2n-2k)}\times\C^{N}$. Therefore, it is contained in $D^{'}\times\Delta$ for some connected open set $D^{'}$ and polydisc $\Delta$ of suitable radius.  Since $\{\phi_{\alpha}\}$ are local charts of $G$, $F|_{G}$ is an injective immersion. Since $\overline{G}$ is compact, $F$ and $F|_{G}$ both are also embedding on $\overline{G}$ and $G$, respectively. 
\end{proof}
Before describing $\img{(F|_{G})}$, $F$ as defined in \eqref{stein map}, we need to consider some properties related to GC maps. Let $V^{2n}$ and $W^{2m}$ be two GC vector spaces of type $k_1$ and $k_2$, respectively such that $n<m$. Let $f: V\longrightarrow W$ be an injective GC map. Since GC maps are $B$-transformation invariant, without loss of generality, we can assume 
$$V=V_1\oplus V_2\quad\text{and}\quad W=W_1\oplus W_2\,,$$ where $(V_1,\omega_1),(W_1,\omega_2)$ are isomorphic to $\R^{2n-2k_1},\R^{2m-2k_2}$ with standard symplectic structures and $V_2,W_2$ are isomorphic to $\R^{2k_1},\R^{2k_2}$ with standard complex structures, respectively. Here, we are considering $\R^{2n}$ as a symplectic (respectively, complex) subspace of $\R^{2m}$ for any $n\leq m$ with appropriate (standard) symplectic (respectively, complex) structure. Then we have the following:
\begin{enumerate}
 \setlength\itemsep{0.4em}
    \item $E_{V}=(V_1\otimes\C)\oplus V_2^{0,1}$ and $E_{W}=(W_1\otimes\C)\oplus W_2^{0,1}\,.$
    \item $P_{V}=L(V_1\otimes\C,\omega_1)$ and $P_{W}=L(W_1\otimes\C,\omega_2)\,.$
\end{enumerate} where $P_{V}=\big\{X+\eta\in(V_1\oplus V^{*})\,|\,\eta(X)|_{V_{1}}=\omega(X)\big\}$ and similar meaning to $P_{W}$. Since $f$ is GC linear, by Definition \ref{GC map},
\begin{equation}\label{gc map}
        f(E_{V})\subseteq E_{W}\quad\text{and}\quad f_{\star}P_{V}=P_{W}\,,
\end{equation} where $f_{\star}P_{V}=\big\{f(X)+\eta\in(W\oplus W^{*})\oplus\C\,|\,X+f^{*}\eta\in P_{V}\big\}\,.$ Immediately, from \eqref{gc map}, we can see that $f(V_1\otimes\C)=W_1\otimes\C$ and $f^{*}\omega_2 f=\omega_1$ on $V_1\otimes\C$, and as, $f$ is also injective, we have
$$f|_{V_1\otimes\C}:V_1\otimes\C\longrightarrow W_1\otimes\C\quad\text{is a linear isomorphism}\,.$$  This implies $f(V_2^{0,1})\subseteq W_2^{0,1}$. Therefor, $f(V_2\otimes\C)\subseteq W_2\otimes\C$ as $f$ is a real map. Now the maximal isotropic subspace of $(W\oplus W^{*})\otimes\C\,,$ corresponding to the product GCS, is $$L_{W}=\big\{X-i\omega_{2}\,|\,X\in W_1\otimes\C\big\}\bigoplus\left(W_2^{0,1}\oplus W_2^{1,0*}\right)\,.$$ Then 
\begin{align*}
L_{W}\cap(f(V)\oplus W^{*})\otimes\C&=L_{W}\cap\left\{(W_1\otimes\C)\oplus f(V_2\otimes\C)\oplus (W^{*}\otimes\C)\right\}\\
&=\big\{X-i\omega_{2}\,|\,X\in W_1\otimes\C\big\}\oplus\left(f(V_2^{0,1})\oplus W_2^{1,0*}\right)\,.
\end{align*}
Consider the following subspace of $(W\oplus W^{*})\otimes\C$
\begin{align*}
    L_{f(V)}&=\left\{X+\eta|_{f(V)\otimes\C}\,|\,(X+\eta)\in L_{W}\cap(f(V)\oplus W^{*})\otimes\C\right\}
    \stepcounter{equation}\tag{\theequation}\label{myeq1}\\
    &=\big\{X-i\omega_{2}\,|\,X\in W_1\otimes\C\big\}\oplus\left(f(V_2^{0,1})\oplus f(V_2^{1,0})^{*}\right)\,.
\end{align*}
Observe that $L_{f(V)}\cap\overline{L_{f(V)}}=\emptyset$. Hence, $f(V)$ is a GC subspace of $W$ by \cite[Definition 3.2]{ben04} with the GCS, $L_{f(V)}$, induced from $L_{W}$. This shows that, after a $B$-transformation, $f$ is an isomorphism between the linear GC spaces $(V,L_{V})$ and $(f(V),L_{f(V)})$ where $L_{V}$ is defined similarly as $L_{W}\,.$ Thus, we have proved the following.
\begin{prop}\label{imp prop1}
 Let $f: V\longrightarrow W$ be an injective GC linear map between two GC linear spaces. Let $\typ(V)$ denote the type of GCS on $V$. Then 

 \begin{enumerate}
 \setlength\itemsep{0.4em}
     \item $f(V)$ is a GC subspace of $W$.
     \item $f: V\longrightarrow f(V)$ induces an isomorphism of linear GC structures where the GC structure on $f(V)$ is a $B$-transformation of the GC structure of $f(V)$ as a GC subspace of $W\,.$ 
     \item  $\typ(W)-\typ(V)=\frac{1}{2}(\dim_{\R}(W)-\dim_{\R}(V))\,.$
 \end{enumerate}
\end{prop}
For a regular GC manifold, vector spaces are replaced with tangent bundles and cotangent bundles. Then the same preceding discussion follows. One caution is that whenever we have a GH embedding between GC manifolds, $f: Y\longrightarrow Z$, the set $L_{f(Y)}$, as defined in \eqref{myeq1}, will be a smooth maximal isotropic distribution of $(TZ\oplus T^{*}Z)\otimes\C$ in general. But, if both $Y, Z$ are regular then $L_{f(Y)}$ becomes a subbundle (see Remark \ref{rmk sub}). Therefore, we get the following    
\begin{theorem}\label{imp thm1}
  Let $f: Y\longrightarrow Z$ be a GH map between two regular GC manifolds which is also an embedding. Let $\dim_{\R}Y<\dim_{\R}Z$. Then 

 \begin{enumerate}
 \setlength\itemsep{0.4em}
     \item $f(Y)$ is an embedded GC manifold of $Z$.
     \item $f: Y\longrightarrow f(Y)$ induces a GH homeomorphism where the GC structure on $f(Y)$ is a $B$-transformation of the GC structure of $f(Y)$ as a GC submanifold of $Z$, that is, the GC structure induced by $f$ on $f(Y)$ and $f(Y)$ as a GC submanifold are equivalent up to a $B$-transformation. 
     \item  $\typ(Z)-\typ(Y)=\frac{1}{2}(\dim_{\R}(Z)-\dim_{\R}(Y))$ where $\typ(Y)$ is defined as in Definition \ref{def:type}.
 \end{enumerate}   
\end{theorem}
\begin{proof}
    Follows from Proposition \ref{imp prop1}.
\end{proof}
Let $y\in f(Y)\subset Z$. By Theorem \ref{imp thm1}, the leaf directions of the induced symplectic foliations (see Remark \ref{foli rmk}) on $f(Y)$ and $Z$ are locally preserved. Applying Weinstein's proof of the Darboux normal coordinate theorem for a family of symplectic structures (cf. \cite{mcduff}), we obtain a local diffeomorphism around $y\in Z$ into $\R^{2\dim_{\R}Z}$ that preserves leaves. Restricting this diffeomorphism to $f(Y)$, we similarly get a leaf-preserving local diffeomorphism around $y\in f(Y)$ into $\R^{2\dim_{\R}f(Y)}$. Since $f(Y)$ is an embedded submanifold, in a sufficiently small neighborhood around $y$, the compatibility of the GCS, up to $B$-transformations, on both $f(Y)$ and $Z$ can be described. More precisely, since $f(Y)$ is both a GC manifold and an embedded submanifold, using Theorem \ref{darbu thm}, there exists a coordinate neighborhood $U$ around $y\in Z$ such that the following diagram commutes
\[\begin{tikzcd}[ampersand replacement=\&]
	{y\in U} \&\& {\R^{dim_{\R}(Z)-2\typ(Z)}\times\C^{\typ(Z)}} \\
	\\
	{y\in U\cap f(Y)} \&\& {\R^{dim_{\R}(Y)-2\typ(Y)}\times\C^{\typ(Y)}}
	\arrow["\cong", from=1-1, to=1-3]
	\arrow[hook, from=3-1, to=1-1]
	\arrow["\cong"', from=3-1, to=3-3]
	\arrow[hook', from=3-3, to=1-3]
\end{tikzcd}\] where horizontal arrows are GH homeomorphism and vertical arrows are natural inclusion maps. Notice that, on the $\R$-component, the inclusion map is identity, and the map $$\mathcal{O}_{\C^{\typ(Z)}}\longrightarrow i_{*}\mathcal{O}_{\C^{\typ(Y)}}\,,$$ induced by the inclusion, is onto morphism of sheaves on $\C^{\typ(Z)}$. We know that  $$\mathcal{O}_{\R^{dim_{\R}(Z)-2\typ(Z)}\times\C^{\typ(Z)}}=\pr_2^{-1}\mathcal{O}_{\C^{\typ(Z)}}\,,$$ and  similarly for $\mathcal{O}_{\R^{dim_{\R}(Y)-2\typ(Y)}\times\C^{\typ(Y)}}\,.$ This implies that the map 
$$\mathcal{O}_{Z}(U)\cong\mathcal{O}_{\R^{dim_{\R}(Z)-2\typ(Z)}\times\C^{\typ(Z)}}\longrightarrow i_{*}\mathcal{O}_{\R^{dim_{\R}(Y)-2\typ(Y)}\times\C^{\typ(Y)}}\cong i_{*}\mathcal{O}_{f(Y)}(U) \,,$$ again induced by the inclusion, is also onto. Therefore, we have the following
\begin{corollary}\label{imp cor1}
Let $f$ be a GH map as in Theorem \ref{imp thm1}. 
Then, the morphism of sheaves over $Z$, 
$$i^{\#}:\mathcal{O}_{Z}\longrightarrow i_{*}\mathcal{O}_{f(Y)}$$ 
induced by the inclusion map $i:f(Y)\hookrightarrow Z$ is onto and its kernel sheaf, $\ker(i^{\#})$, is called the ideal sheaf for $f(Y)$ on $Z$.  Moreover, we have the following isomorphism. $$\mathcal{O}_{Y}\cong\mathcal{O}_{f(Y)}=\mathcal{O}_{Z}|_{f(Y)}\,.$$ Here $f(Y)$ is considered as a GC submanifold of $Z$.
\end{corollary}
\begin{proof}
    Follows from the preceding discussion and Theorem \ref{imp thm1}.
\end{proof}
As, $F(\overline{G})$ is compact and $F$ is an embedding, $F(G)$ can be considered as an embedded submanifold of some $B_{r}=D^{'}\times\D_{r}$, as defined in \eqref{br}, where $D^{'}\subset\R^{2M}$ is an open ball. Let $B_{r}$ be the GC manifold with the product GCS induced from $\R^{2M}\times\C^{N}\,.$ Then, using Proposition \ref{prop}, we proved the following by Theorem \ref{imp thm1} and Corollary \ref{imp cor1}. 
\begin{theorem}\label{imp thm2}
Let $G$ be a relatively compact open subset of a GC Stein manifold $X$. Then, there exists a GH embedding, say $F$, (cf. Definition \ref{GH map}) of $G$ into some product of open balls
 $B_{r}\subset\R^{2M}\times\C^{N}$ for $N>0$ large enough and $M=\frac{1}{2}(\dim_{\R}X-\typ(X))$. Moreover, $G$ becomes a embedded GC submanifold of $B_{r}$ for suitable $r>0$, and the morphism, induced by $F$, $$F^{\#}:\mathcal{O}_{B_{r}}\longrightarrow F_{*}\mathcal{O}_{G}\,,$$ is an onto morphism of sheaves.
\end{theorem}

\begin{center}
\textit{From now on, we will assume Theorem \ref{imp claim} for the rest of this subsection.}  
\end{center}
\begin{theorem}\label{thm2}
Suppose $\mathcal{F}$ is a coherent sheaf of $\mathcal{O}$-modules. Let $G$ be a relatively compact open subset of $X$. Then, $\mathcal{F}$ is generated on $G$ by $H^0(G,\mathcal{F})$.   
\end{theorem}
\begin{proof}
Follows from the combination of Proposition \ref{prop} and Theorem \ref{thm}.   
\end{proof}
\begin{prop}\label{prop2}
    Suppose $G\subset X$ is a (connected) open set. Let $\phi:X\longrightarrow\R^{2M}\times\C^{N}$ be a GH map such that 

    \vspace{0.2em}
    \begin{enumerate}
     \setlength\itemsep{0.4em}
        \item for some open neighborhood $G^{'}$ of \,\,$\overline{G}$, $\phi$ maps $G^{'}$ GH homeomorphically onto an embedded GC submanifold of some product of open subsets $D_1\times D_2\subset\R^{2M}\times\C^{N}$ where $D_1\subset\R^{2M}$ is (connected) open set and $M=\frac{1}{2}(\dim_{\R}X-\typ(X))$.        
        \item $\phi(G)$ is an embedded (closed) GC submanifold in a $B_{r}\subset\subset D_1\times D_2$ where $B_{r}$ as defined in \eqref{br}.
    \end{enumerate}
    
    \vspace{0.3em}
    \hspace{-1em}
    Then $\mathcal{O}(X)$ is dense in $\mathcal{O}_{G}(G)$ with the topology of uniform convergence on compact subsets.
\end{prop}
\begin{proof}
Let $\mathcal{I}$ be the ideal sheaf for $\phi(G^{'})$ on $D_1\times D_2$. Then $\mathcal{I}$ is coherent on $D_1\times D_2$. Since $H^1(B_{r},\mathcal{I})=0$ (by Theorem \ref{imp claim}), using the long exact sequence of cohomology of
\[\begin{tikzcd}
	0 & {\mathcal{I}} & {\mathcal{O}_{B_{r}}} & {\mathcal{O}_{B_{r}}/\mathcal{I}} & 0 & {\text{on \,\,$B_{r}$\,,}}
	\arrow[from=1-1, to=1-2]
	\arrow[from=1-2, to=1-3]
	\arrow["\eta", from=1-3, to=1-4]
	\arrow[from=1-4, to=1-5]
\end{tikzcd}\]
shows that the homomorphism $\tilde{\eta}:H^0(B_{r},\mathcal{O}_{B_{r}})\longrightarrow H^0(B_{r},\mathcal{O}_{B_{r}}/\mathcal{I})$, induced by $\eta$, is surjective.  Note that
\begin{align*}
H^0(B_{r},\mathcal{O}_{B_{r}}/\mathcal{I})&=H^0(B_{r},(\phi|_{G})_{*}\mathcal{O}_{G})\quad(\text{by Corollary \ref{imp cor1}})\\
&\cong H^0(G,\mathcal{O}_{G})\quad(\text{as $\phi$ is a GH embedding})\,. 
\end{align*}
This implies that the map $\widetilde{\phi|_{G}}:H^0(B_{r},\mathcal{O}_{B_{r}})\longrightarrow H^0(G,\mathcal{O}_{G})$, induced by $\phi|_{G}$, is surjective. Let $\tilde{\phi}:H^0(\R^{2M}\times\C^{N},\mathcal{O}_{\R^{2M}\times\C^{N}})\longrightarrow H^0(X,\mathcal{O})$ be the homomorphism induced by $\phi$. Consider the following commutative diagram
\[\begin{tikzcd}[ampersand replacement=\&]
	{H^0(\C^{N},\mathcal{O}_{\C^{N}})=H^0(\R^{2M}\times\C^{N},\mathcal{O}_{\R^{2M}\times\C^{N}})} \&\& {H^0(X,\mathcal{O})} \\
	\\
	{H^0(\D_{r},\mathcal{O}_{\D_{r}})=H^0(B_{r},\mathcal{O}_{B_{r}})} \&\& {H^0(G,\mathcal{O}_{G})}
	\arrow["{\tilde{\phi}}", from=1-1, to=1-3]
	\arrow["\sigma", from=1-3, to=3-3]
	\arrow["\rho"', from=1-1, to=3-1]
	\arrow["{\tilde{\phi|_{G}}}"', from=3-1, to=3-3]
\end{tikzcd}\] where $\rho,\sigma$ are restriction maps. Since $H^0(\C^{N},\mathcal{O}_{\C^{N}})$ is dense in $H^0(\D_{r},\mathcal{O}_{\D_{r}})$, by commutativity of the diagram we get that $H^0(X,\mathcal{O})$ is also dense in $H^0(G,\mathcal{O}_{G})$.
\end{proof}
\begin{theorem}\label{main}
 Let $X$ be a GC Stein manifold and $\mathcal{F}$ be a coherent sheaf of\,\,\,$\mathcal{O}$-modules. Then, we have

 \vspace{0.2em}
 \begin{enumerate}
 \setlength\itemsep{0.4em}
     \item (Theorem B) $H^{p}(X,\mathcal{F})=0$ for $p\geq 1\,.$
     \item (Theorem A) $H^{0}(X,\mathcal{F})$ generates $\mathcal{F}_{x}$ for all $x\in X\,.$
 \end{enumerate}
\end{theorem}
\begin{proof}
    \textbf{(1)}\,\, By Definition \ref{main def} and Proposition \ref{prop}, We can construct (connected) open subsets $G_{j}$ in $X$ and GH map $\phi_{j}:X\longrightarrow\R^{2M_{j}}\times\C^{N_{j}}\,,1\leq j<\infty\,,$ such that the following holds:

    \vspace{0.2em}
    \begin{enumerate}
    \setlength\itemsep{0.4em}
        \item $X=\bigcup_{j}G_{j}\,.$
        \item $G_{j}\subset\subset G_{j+1}\,.$
        \item $\phi_{j}$ maps $G_{j+1}$ GH homeomorphically onto an embedded submanifold of some product of open subsets in $\R^{2M_{j}}\times\C^{N_{j}}$ where
        $M_{j}=\frac{1}{2}(\dim_{\R}X-\typ(X))$.
        \item $\phi_{j}(G_{j})$ is an embedded (closed) GC submanifold in a some $B_{r}\subset\R^{2M_{j}}\times\C^{N_{j}}$ as in Proposition \ref{prop2} where $B_{r}$ is defined as in \eqref{br}.
    \end{enumerate}
    Since $\mathcal{F}$ is coherent, by Theorem \ref{thm2} there exist a sheaf epimorphisms $\psi_{j}:\mathcal{O}^{m_{j}}\longrightarrow\mathcal{F}$ on $G_{j}$ for $m_{j}\in\N\,\,\,\text{and}\,\,\,j\geq 1\,.$  Now, using Theorem \ref{imp claim}, Theorem \ref{imp thm1} and Corollary \ref{imp cor1}, we get $$H^1(G_{j},\ker\psi_{j+s})=0\quad\text{for $j\geq 1$ and $s\geq 1$}\,.$$ This implies that the homomorphism 
    \begin{equation}\label{imp eq}
     \widetilde{\psi_{j,s}}:H^0(G_{j},\mathcal{O}^{m_{j+s}})\longrightarrow H^0(G_{j},\mathcal{F})   
    \end{equation}
induced by $\psi_{j+s}$, is surjective for $j\geq 1$ and $s\geq 1$. By Theorem \ref{imp thm}, $H^0(G_{j},\mathcal{F})$ has canonical Fr\'{e}chet space structure for $j\geq 1$. Then, for $s\geq 1$, we have the following commutative diagram 
\[\begin{tikzcd}[ampersand replacement=\&]
	{H^0(G_{j+s},\mathcal{O}^{m_{j+s+1}})} \&\& {H^0(G_{j},\mathcal{O}^{m_{j+s+1}})} \\
	\\
	{H^0(G_{j+s},\mathcal{F})} \&\& {H^0(G_{j},\mathcal{F})}
	\arrow[from=1-1, to=1-3]
	\arrow["{\widetilde{\psi_{j,s+1}}}", from=1-3, to=3-3]
	\arrow["{\widetilde{\psi_{j+s,1}}}"', from=1-1, to=3-1]
	\arrow[from=3-1, to=3-3]
\end{tikzcd}\] where horizontal maps are restriction maps. By Proposition \ref{prop2} and \eqref{imp eq}, $H^0(G_{j+s},\mathcal{F})$ is dense in $H^0(G_{j},\mathcal{F})$. Now, using Theorem \ref{imp claim}, we get that $H^{p}(G_{j},\mathcal{F})=0$ for $p\geq 1$ and $j\geq 1$. Then, by Leray's theorem we can see that $H^{p}(X,\mathcal{F})=H^{p}(\{G_{j}\},\mathcal{F})$ for $p\geq 0$.

\vspace{0.4em}
Fix $p\geq 1$ and $\alpha\in Z^{p}(\{G_{j}\},\mathcal{F})$. Note that $H^{p}(G_{m},\mathcal{F})=H^{p}(X_{m},\mathcal{F})$ for $m\geq 1$ where $X_{m}$ is the open covering of $G_{m}$ defined as $X_{m}:=\{G_{j}\}^{m}_{j=1}$. Let $\alpha_{m}=\alpha|_{X_{m}}$. Then there exist $\beta_{m}\in C^{p-1}(X_{m},\mathcal{F})$ such that $\alpha_{m}=\delta\beta_{m}$ where $\delta$ denotes the C\v{e}ch coboundary map in C\v{e}ch cohomology. Notice that $\beta_{m+1}-\beta_{m}\in Z^{p-1}(X_{m},\mathcal{F})\,.$

\vspace{0.4em}
\hspace{-1.2em}\textbf{Case 1:} 
\,\,Let $p=1$. We use induction on $m$ to construct $\beta^{'}_{m}\in C^0(X_{m},\mathcal{F})$ such that
\begin{equation}\label{eq2}
\alpha_{m}=\delta\beta^{'}_{m}\quad\text{and}\quad\sup_{3\leq j\leq m}||(\beta^{'}_{j}-\beta^{'}_{j-1})|_{G_{j-1}}||^{\psi_{j}}_{\bar{G_{j-2}}}<\frac{1}{2^{m}}\,.    
\end{equation}
Choose $\beta^{'}_{1}=\beta_{1}$ and suppose we have chosen $\beta^{'}_{1},\beta^{'}_{2},\ldots,\beta^{'}_{m-1}\,.$ Then $\beta_{m}-\beta^{'}_{m-1}$ is a section of $\mathcal{F}$ on $G_{m-1}$. By denseness property of $H^0(G_{m},\mathcal{F})$ in $H^0(G_{m-1},\mathcal{F})$ and \eqref{imp eq}, there exist $\tau\in H^0(G_{m},\mathcal{F})$ such that 
$$\sup_{3\leq j\leq m}||\tau-(\beta_{j}-\beta^{'}_{j-1})|_{G_{j-1}}||^{\psi_{j}}_{\bar{G_{j-2}}}<\frac{1}{2^{m}}\,.$$ Set $\beta^{'}_{m}=\beta_{m}-\tau$ and this completes the induction. Define 
$\beta^{'}\in C^0(\{G_{j}\},\mathcal{F})$ by
$$\beta^{'}(G_{m})=\lim_{j\geq m}\beta^{'}_{j}(G_{m})\quad\text{for $m\geq 1$}\,.$$ One can see that this limit is well-defined by \eqref{eq2} and $\alpha=\delta\beta^{'}\,.$ Hence $H^1(\{G_{j}\},\mathcal{F})=0\,.$

\vspace{0.4em}
\hspace{-1.2em}\textbf{Case 2:}
\,\,Let $p>1$. Since $\beta_{m+1}-\beta_{m}\in Z^{p-1}(X_{m},\mathcal{F})\,,$ and $H^{p-1}(X_{m},\mathcal{F})=0\,,$ there exists $\beta^{'}_{m}\in C^{p-2}(X_{m},\mathcal{F})$ such that $\delta(\beta^{'}_{m})=\beta_{m+1}-\beta_{m}$ on $X_{m}$. Define $\eta\in C^{p-1}(\{G_{j}\},\mathcal{F})$ by 
$$\eta=\beta_{m}+\delta(\sum_{j<m+1}\beta^{'}_{j})\quad\text{on $X_{m}$}\,.$$
One can see that $\eta$ is well-defined and $\alpha=\delta\eta$. Hence $H^{p}(\{G_{j}\},\mathcal{F})=0$ for $p>1$.

\vspace{0.5em}
\hspace{-1.2em}\textbf{(2)}
\,\,Using Theorem B, the proof follows similarly as in Theorem \ref{thm}.
\end{proof}
\begin{theorem}(Characterization for GC Stein manifolds via Cartan's theorems)\label{main1}
    Let $X$ be a regular GC manifold and $\mathcal{F}$ be a coherent sheaf of $\mathcal{O}$-modules. Then $X$ is a GC Stein Manifold if and only if the following holds:

    \vspace{0.3em}
    \begin{enumerate}
    \setlength\itemsep{0.4em}
     \item $H^{p}(X,\mathcal{F})=0$ for $p\geq 1\,.$
     \item $H^{0}(X,\mathcal{F})$ generates $\mathcal{F}_{x}$ for all $x\in X\,.$
     \item $X$ satisfies the following property:
     \begin{equation}\tag{P}\label{p}
         \begin{aligned}
             &\text{Let $f:U\longrightarrow(\R^{2},\omega^{-1}_0)$ be a Poisson map on an open set $U\subseteq X$. Then,}\\
    &\text{for each $x\in U$, there exists an open set $x\in U^{'}\subset U$ and a GH map,}\\
    &\text{$\tilde{f}:X\longrightarrow(\R^{2},\omega^{-1}_0)$ such that $\tilde{f}|_{U^{'}}=f|_{U^{'}}\,.$}
         \end{aligned}
     \end{equation} 
    \end{enumerate}
\end{theorem}
\begin{proof}
    Follows from Remark \ref{poi rmk}, Definition \ref{main def}, and Theorem \ref{main}.
\end{proof}
\begin{corollary}
    Any SGH vector bundle over a GC Stein manifold is GH homeomorphically trivial.
\end{corollary}
\begin{proof}
    Follows from Theorem \ref{main1} and \cite[Proposition 4.6]{pal24}
\end{proof}

\subsection{The proof of Theorem \ref{imp claim}}\label{pf}
Before going into the proof of Theorem \ref{imp claim}, we need the following Proposition. We will follow the approach in \cite[Proposition 2]{siu68}. So, let $G$ be an open neighborhood of the closure of $B_{r}$ (cf. \eqref{br}), and  $\mathcal{F}$ be a coherent sheaf on $G$.
\begin{prop}\label{prop1}
For all $p\geq 1\,,$ $H^{p}(B_{r},\mathcal{F})$ are finite dimensional $\C$-vector space.    
\end{prop}
\begin{proof}
Let $D^{'}:=\{|x|<R\,|\,x\in\R^{2M}\}$ for some $R>0$. For $1\leq j\leq m$, consider the product of balls $U^{'}_{j}\times U^{''}_{j}\subset\subset V^{'}_{j}\times V^{''}_{j}\subset\subset G$ in $\R^{2M}\times\C^{N}$  such that
  \begin{enumerate}
  \setlength\itemsep{0.4em}
      \item $\partial D^{'}\times\partial\D_{r}\subset\bigcup_{j}(U^{'}_{j}\times U^{''}_{j})\,.$
      \item $\mathcal{F}$ admits a finite free resolution on $V^{'}_{j}\times V^{''}_{j}$.
  \end{enumerate}
  Let $\phi^{'}_{j}\,,\phi^{''}_{j}\geq 0$ be smooth real valued functions on $\R^{2M}$ and $\C^{N}$, respectively, such that $\phi^{'}_{j}\,,\phi^{''}_{j}>0$ on $U^{'}_{j}\,,U^{''}_{j}$ and $\phi^{'}_{j}\,,\phi^{''}_{j}\equiv 0$ outside $V^{'}_{j}\,,V^{''}_{j}$, respectively. Consider the smooth functions $\tilde{\phi^{'}_0}=\sum_{i=1}^{2M}|x_{i}|^{2}-R^{2}$ and $\tilde{\phi^{''}_0}=\sum_{i=1}^{N}|z_{i}|^{2}-r^{2}$. Choose $\lambda^{'}_{j}\,,\lambda^{''}_{j}\in(0,\infty)\,,$ $1\leq j\leq m$, so small such that $\pr_1^{-1}(D^{'}_{j})\subset G$ and $\tilde{\phi^{''}_{j}}=\tilde{\phi^{''}_0}-\sum^{j}_{i=1}\lambda^{''}_{i}\phi^{''}_{i}$ satisfies \eqref{*}, respectively, where  $D^{'}_{j}:=\{\tilde{\phi^{'}_{j}}<0\}$ and $\tilde{\phi^{'}_{j}}=\tilde{\phi^{'}_0}-\sum^{j}_{i=1}\lambda^{'}_{i}\phi^{'}_{i}$.  Define, for $0\leq j\leq m$,
 \[ D_{j}=
 \begin{cases}
 D^{'}\times D^{''}_{0}\,,& \text{if }  j=0\\
D^{'}_{j}\times D^{''}_{j}\,,& \text{if }  1\leq j\leq m
 \end{cases} 
 \text{where}\quad D^{''}_{j}:=\{\tilde{\phi_{j}}<0\}\,.
 \] 
 Note that $D_0=B_{r}\subset\subset D_{m}$, $D_{j}=D_{j-1}\cup(D_{j}\cap (V^{'}_{j}\times V^{''}_{j}))$, and also $D_{j-1}\cap (V^{'}_{j}\times V^{''}_{j})=D_{j-1}\cap(D_{j}\cap (V^{'}_{j}\times V^{''}_{j}))\,.$ By Corollary \ref{imp cor}, $H^{p}(D_{j-1}\cap (V^{'}_{j}\times V^{''}_{j}),\mathcal{F})=0$ for $p\geq 1$ and $1\leq j\leq m$. Therefore, using the Mayor-Vietoris sequence 
 $$\longrightarrow H^{p}(D_{j},\mathcal{F})\longrightarrow H^{p}(D_{j-1},\mathcal{F})\oplus H^{p}(D_{j}\cap (V^{'}_{j}\times V^{''}_{j}),\mathcal{F})\longrightarrow H^{p}(D_{j-1}\cap (V^{'}_{j}\times V^{''}_{j}),\mathcal{F})\longrightarrow$$
we get that $H^{p}(D_{j},\mathcal{F})\longrightarrow H^{p}(D_{j-1},\mathcal{F})$ is surjective for $p\geq 1$ and $1\leq j\leq m$. Hence the restriction induces a subjective map $H^{p}(D_{m},\mathcal{F})\longrightarrow H^{p}(B_{r},\mathcal{F})$.
\vspace{0.4em}

For $i=1,2$ and $1\leq j\leq l$, choose two finite collections of the product of balls $\{U^{'}_{ij}\times U^{''}_{ij}\}$ in $G$ such that

\vspace{0.2em}
\begin{enumerate}
\setlength\itemsep{0.3em}
    \item $U^{'}_{1j}\times U^{''}_{1j}\subset\subset U^{'}_{2j}\times U^{''}_{2j}$.
    \item $B_{r}\subset\bigcup_{j}(U^{'}_{1j}\times U^{''}_{1j})$ and $D_{m}\subset\bigcup_{j}(U^{'}_{2j}\times U^{''}_{2j})$.
    \item On $U^{'}_{2j}\times U^{''}_{2j}\,,$ we have a sheaf epimorphism $\psi_{j}:\mathcal{O}_{G}^{p_{j}}\longrightarrow\mathcal{F}$ which is a part of a finite free resolution.
\end{enumerate}
Now, consider the following:
\begin{enumerate}
\setlength\itemsep{0.3em}
    \item $V_{1j}=(U^{'}_{1j}\times U^{''}_{1j})\cap B_{r}$ and $V_{2j}=(U^{'}_{2j}\times U^{''}_{2j})\cap D_{m}$ for $1\leq j\leq l$.
    \item $V_{i}=\{V_{ij}\}^{l}_{j=1}\,,\,i=1,2\,.$
\end{enumerate}
Fix $p\geq 1$. Since $H^{p}(\cap_{j\in\{j_0,j_1,\ldots,j_{q}\}}V_{ij},\ker\psi_{j_0})=0$ for all $1\leq j_{q},\ldots,j_1,j_0\leq l$ by Corollary \ref{imp cor}, the map 
$$H^0(\cap_{j\in\{j_0,j_1,\ldots,j_{q}\}}V_{ij},\mathcal{O}_{G}^{p_{j_0}})\longrightarrow H^0(\cap_{j\in\{j_0,j_1,\ldots,j_{q}\}}V_{ij},\mathcal{F})$$ induced by $\psi_{j_0}$ is surjective for $i=1,2$. So, by Theorem \ref{imp thm}, $H^0(\cap_{j\in\{j_0,j_1,\ldots,j_{q}\}}V_{ij},\mathcal{F})$ has a canonical Fr\'{e}chet space structure. Since the cochain groups
$$C^{p}(\{V_{ij}\},\mathcal{F})=\bigoplus_{j_1,\ldots,j_{p}}H^0(\cap_{j\in\{j_1,\ldots,j_{p}\}}V_{ij},\mathcal{F})$$
inherits product topology, $ Z^{p}(\{V_{ij}\},\mathcal{F})$ and $C^{p-1}(\{V_{ij}\},\mathcal{F})$ has a canonical Fr\'{e}chet space structure for $i=1,2$. Also, note that by Leray's theorem  $$H^{p}(\{V_{1j}\},\mathcal{F})=H^{p}(B_{r},\mathcal{F})\quad\text{and}\quad H^{p}(\{V_{2j}\},\mathcal{F})=H^{p}(D_{m},\mathcal{F})\,\,\text{for $p\geq 0$}\,.$$ Let $\eta: Z^{p}(\{V_{2j}\},\mathcal{F})\longrightarrow Z^{p}(\{V_{1j}\},\mathcal{F})$ be the restriction map and $\delta$ be the coboundary map for $\{C^{p}(\{V_{1j}\},\mathcal{F})\}$. Then the map
$$\eta+\delta:Z^{p}(\{V_{2j}\},\mathcal{F})\oplus C^{p-1}(\{V_{1j}\},\mathcal{F})\longrightarrow Z^{p}(\{V_{1j}\},\mathcal{F})$$ defined by $\eta+\delta(c,d)=\eta(c)+\delta(d)$ is subjective. Since $V_{1j}\subset\subset V_{2j}$, the map $\eta+0$ is compact. Then by Schwartz's theorem (cf. \cite[p.~290]{rossi65}), $0+\delta=(\eta+\delta)-(\eta+0)$ has finite dimensional cokernal. Thus $\dim_{\C}H^{p}(B_{r},\mathcal{F})<+\infty$.
\end{proof}

\begin{remark}
    Note that we can substitute $D^{'}$ in Proposition \ref{prop1} with $D^{''}=\big\{\psi<0\big\}$ for any smooth function $\psi:\R^{2M}\longrightarrow\R$, and the proof will follow in a similar manner.
\end{remark}
\vspace{0.5em}

\hspace{-1.2em}\textit{Proof of Theorem \ref{imp claim}.}
\,\,By shrinking $G$, we can always assume $\dim\supp\mathcal{F}<+\infty$. Fix $p\geq 1$. We will use induction on $\dim\supp\mathcal{F}$. When $\dim\supp\mathcal{F}=0$, the Proposition is trivially true. Suppose that the proposition is true for any coherent sheaf $\mathcal{G}$ of $\mathcal{O}_{G}$-modules of $\dim\supp\mathcal{G}<d$. Let's assume $\dim\supp\mathcal{F}=d$.
\vspace{0.5em}

Let $E$ be the set of entire function on $\C$ and $\pr:\R^{2M}\times\C^{N}\longrightarrow\C$ be the projection map. Let $g\in E-\{0\}$. Consider the sheaf homomorphism
$$\phi_{g}:\mathcal{F}\longrightarrow\mathcal{F}$$ defined by the multiplication of $g\circ\pr$.

\vspace{0.4em}
Note that  both $\dim\supp\ker\phi_{g}<d$ and $\dim\supp\coker\phi_{g}<d$. Therefore, by the induction hypothesis, for $p\geq 1$,
$$H^{p}(B_{r},\ker\phi_{g})=H^{p}(B_{r},\coker\phi_{g})=0\,.$$
Consider the following two short exact sequences of sheaves
$$0\longrightarrow\ker\phi_{g}\hookrightarrow\mathcal{F}\longrightarrow\mathcal{F}/\ker\phi_{g}\longrightarrow 0\,,$$ and 
$$0\longrightarrow\mathcal{F}/\ker\phi_{g}\xrightarrow{\tilde{\phi_{g}}}\mathcal{F}\longrightarrow\coker\phi_{g}\longrightarrow 0\,,$$ where $\tilde{\phi_{g}}$ is induced by $\phi_{g}$.
Using the long exact sequence of the cohomology of the first and the second short exact sequence, we get the following $$H^{p}(B_{r},\mathcal{F})\xrightarrow{\cong}H^{p}(B_{r},\mathcal{F}/\ker\phi_{g})\quad\text{and}\quad H^{p}(B_{r},\mathcal{F}/\ker\phi_{g})\xrightarrow{\cong}H^{p}(B_{r},\mathcal{F})\,,$$ for $p\geq 1$, respectively. This implies that, for $p\geq 1$, $\phi_{g}$ induces an isomorphism $$\phi^{*}_{g}:H^{p}(B_{r},\mathcal{F})\xrightarrow{\cong}H^{p}(B_{r},\mathcal{F})\,.$$ Let $\alpha\neq 0\in H^{p}(B_{r},\mathcal{F})$ and consider the map
$$\widehat\phi_{\alpha}:E\longrightarrow H^{p}(B_{r},\mathcal{F})$$ defined by
$\widehat\phi_{\alpha}(g)=\phi^{*}_{g}(\alpha)$ for all $g\in E-\{0\}$ and $\widehat\phi_{\alpha}(0)=0$. One can see that $\widehat\phi_{\alpha}$ is a linear map. If possible, let $g\in E-\{0\}$ such that $\widehat\phi_{\alpha}(g)=0$. This implies $\phi^{*}_{g}(\alpha)=0$ for $\alpha\neq 0$ which is a contradiction. Thus $\ker\widehat\phi_{\alpha}=\{0\}$ and so, $\dim_{\C}E<\dim_{\C}H^{p}(B_{r},\mathcal{F})$. This contradicts Proposition \ref{prop1}. Hence $H^{p}(B_{r},\mathcal{F})=0$ for $p\geq 1$.

\section{GC Stein manifold and $L$-plurisubharmonic function}\label{l2}
In this section, we extend the $L^2$-characterization of Stein manifolds by providing a similar characterization in the context of GC geometry.

\subsection{$L$-plurisubharmonicity}
Let $(X^{2n},\mathcal{J})$ be a regular GC manifold of type $k>0$ with $i$-eigen bundle $L$. Then, by Definition \ref{gcs}, $(TX\oplus T^{*}X)\otimes\C=L\oplus\overline{L}\,,$ and  we have a differential operator, 
\begin{equation}\label{d-L}
    d_{L}:C^{\infty}(\wedge^{\bullet}L^{*})\longrightarrow C^{\infty}(\wedge^{\bullet+1}L^{*})\,,
\end{equation}
defined as follows:
For any $\omega\in C^{\infty}(\wedge^{m} {L^*})$ and $A_{i}\in C^{\infty}(L)$ for all $i\in\{1,\cdots,m+1\}$, 
\begin{align*}
d_{L}\omega(A_{1},\cdots,A_{m+1})& :=\sum_{i=1}^{m+1}(-1)^{i+1}\rho(A_{i})(\omega(A_{1},\cdots,\hat{A_{i}},\cdots,A_{m+1}))\\
&+\sum_{i<j}(-1)^{i+j}\omega([A_{i},A_{j}],A_{1},\cdots,\hat{A_{i}},\cdots,\cdots,\hat{A_{j}},\cdots,A_{m+1})\,,
\end{align*}
where $\rho:(TX\oplus T^{*}X)\otimes\C\longrightarrow TX\otimes\C$ is the projection map and $[\,,\,]$ is the Courant bracket.
Similarly, we have another differential operator
\begin{equation}\label{d-L bar}
    d_{\overline{L}}:C^{\infty}(\wedge^{\bullet}\overline{L}^{*})\longrightarrow C^{\infty}(\wedge^{\bullet+1}\overline{L}^{*})\,.
\end{equation} 
Using Theorem \ref{darbu thm} and Proposition \ref{prop GH}, we can observe the following. 
\begin{prop}(\cite[Corollary 8.4]{pal24})\label{imp corr}
    Let $X$ be a regular GC manifold. Given an open set $U\subseteq X$, a smooth map $\psi:(U,\mathcal{J}_{U})\longrightarrow\C$ is a GH function, that is, $f\in\mathcal{O}_{X}(U)$, if and only if $d_{L}f=0$ where $d_{L}$ as defined in \eqref{d-L}. 
\end{prop}
For any $p,q\geq 0$\,, define
\begin{equation}\label{A-pq}
\begin{aligned}
  &A^{p,q}:=C^{\infty}(\wedge^{p}\overline{L}^{*}\otimes\wedge^{q}L^{*})\,,\,\,\,\text{and}\,\,\,A^{\bullet}:=\bigoplus_{p+q=\bullet} A^{p,q}\,.\\
&E^{p,q}:=\wedge^{p}\overline{L}^{*}\otimes\wedge^{q}L^{*}\,.
\end{aligned}
\end{equation}
More specifically, given an open set $U\subseteq X$,
$$A^{p,q}(U)=C^{\infty}(U,\wedge^{p}\overline{L}^{*})\otimes_{C^{\infty}(U,\C)} C^{\infty}(U,\wedge^{q}L^{*})\,.$$
In particular, $A^{0,0}(U):=C^{\infty}(U,\C)$. 
We identify $L^{*}$ with $\overline{L}$ via the symmetric bilinear form defined in \eqref{bilinear}. Subsequently, we can naturally extend $d_{L}$ (cf. \eqref{d-L}), and $d_{\overline{L}}$ (cf. \eqref{d-L bar}) to $A^{\bullet,\bullet}$, again denoted by $d_{L}$ and $d_{\overline{L}}$ respectively, and get the following morphisms of sheaves
  \begin{equation}\label{dL-dL bar}
    \begin{aligned}
&d_{L}:A^{\bullet,\bullet}\longrightarrow A^{\bullet,\bullet+1}\,;\\
&d_{\overline{L}}:A^{\bullet,\bullet}\longrightarrow A^{\bullet+1,\bullet}\,;\\
&D:=d_{\overline{L}}+d_{L}\,.
    \end{aligned}
\end{equation}
Note that $d_{L}d_{\overline{L}}=-d_{\overline{L}}d_{L}$. This implies $$D^2=0\quad\text{and}\quad D(A^{\bullet,\bullet})\subseteq A^{\bullet+1,\bullet}\oplus A^{\bullet,\bullet+1}\,.$$ 
\begin{definition}
A real section $s\in A^{1,1}(X)$ (that is, $\overline{s}=s$) is called positive, denoted by $s>0$, if for any non-zero element $v\in\overline{L}\otimes L$, we have $s(v)>0$. Here $E^{1,1}$ is identified with $(\overline{L}\otimes L)^{*}$.
\end{definition}

Let $\mathscr{S}$ denote the induced regular transversely holomorphic, symplectic foliation corresponding to $\mathcal{J}$. Let $d_{\mathscr{S}}$ denote the exterior derivative along the leaves.
\begin{definition}\label{l-pluri}
    Let $f:X\longrightarrow\C$ be a smooth real function, that is $\overline{f(x)}=f(x)$ for all $x\in X$. Then, $f$ is called a \textit{$L$-plurisubharmonic} (respectively, \textit{strictly $L$-plurisubharmonic}) function if $d_{\mathscr{S}}f=0$, that is, $f$ is constant along the leaves of $\mathscr{S}$ and 
    $$id_{\overline{L}}d_{L}f\geq 0\,\,\, (\,\text{respectively,}\,\,\, id_{\overline{L}}d_{L}f>0\,)\,.$$
\end{definition}
\begin{remark}
     Definition \ref{l-pluri} is invariant under $B$-transformations because $B$-transformations are real and bundle isomorphisms, that map involutive, isotropic, maximal subbundles to involutive, isotropic, and maximal subbundles while preserving the type of a GCS.
\end{remark}
\begin{remark}\label{rmk pluri}
    Note that we can express $d_{\overline{L}}f,\, d_{L}f$ and $d_{\mathscr{S}}f$ in terms of local coordinates, as in \eqref{loc coordi}, as follows: 
    $$d_{\overline{L}}f=\sum^{k}_{j=1}\frac{\partial f}{\partial z_{j}}\,dz_{j}+d_{\mathscr{S}}f\,,\,
    d_{L}f=\sum^{k}_{j=1}+\frac{\partial f}{\partial\overline{z_{j}}}\,d\overline{z_{j}}+d_{\mathscr{S}}f\quad\text{and}\quad d_{\mathscr{S}}f=\sum^{2n-2k}_{l=1}\frac{\partial f}{\partial p_{l}}\,dp_{l}$$ for all $i,j\in\{1,\ldots,k\}\,, l\in\{1,\ldots,2n-2k\}\,.$ Now if, $d_{\mathscr{S}}f=0$, then 
    $$d_{\overline{L}}d_{L}f=\sum^{k}_{i,j=1}\frac{\partial^{2} f}{\partial z_{i}\partial\overline{z_{j}}}\,d\overline{z_{j}}\wedge dz_{i}\,.$$ Thus by Corollary \ref{cor:diffcharts}, $d_{\overline{L}}d_{L}f$ is invariant under the change of local coordinates as in \eqref{loc coordi}. This invariance implies that the notion of $L$-plurisubharmonicity is well defined for regular GC manifolds. Therefore, locally, $f$ is $L$-plurisubharmonic (respectively, strictly $L$-plurisubharmonic) iff $d_{\mathscr{S}}f=0$ and
    $$\sum^{k}_{i,j=1}\frac{\partial^{2} f}{\partial z_{i}\partial\overline{z_{j}}}\geq 0\,\,\, (\,\text{respectively,}\,\,\, \sum^{k}_{i,j=1}\frac{\partial^{2} f}{\partial z_{i}\partial\overline{z_{j}}}>0\,)\,.$$
\end{remark}
\begin{theorem}\label{pluri thm}
    Let $X^{2n}$ be GC Stein manifold of type $k$. Let $K\subset X$ be a compact subset and $U$ be an open neighborhood of\,\,\,$\widehat{K}_{\mathcal{O}(X)}$. Then there exists a smooth real function $f:X\longrightarrow\C$ such that
    \begin{enumerate}
    \setlength\itemsep{0.4em}
        \item $f$ is strictly $L$-plurisubharmonic.
        \item $f<0$ in $K$ and $f>0$ in $X\backslash U\,.$
        \item $\bigg\{x\in X\,\,\big|\,\,f(x)<c,\,\,\,\text{for $c\in\R$}\bigg\}\subset\subset X$.
    \end{enumerate}
\end{theorem}
\begin{proof}
Using GH convexity (see $(2)$ in Definition \ref{main def}), $X$ admits an exhaustion $K=K_1\subset K_2\subset\cdots\subset\bigcup_{j}K_{j}=X$ by compact $\mathcal{O}(X)$-convex subsets such that $K_{j}\subset K_{j+1}^{o}$ for every $j\geq 1$. Here $K_{j}^{o}$ denotes the interior of $K_{j}$. Let $U_{j}\subset X$ be an open set such that $K_{j}\subset U_{j}\subset K_{j+1}$ and $U_1\subset U$. Since, for every $j$, $K_{j}=\widehat{K_{j}}_{\mathcal{O}(X)}$, we can choose global GH functions $g_{jl}\in\mathcal{O}(X)$ for $l=1,\ldots,l_{j}$ with $|g_{jl}(x)|<1$ for $x\in K_{j}$ so that $\max_{l}|g_{jl}(x)|>1$ for $x\in K_{j+2}\backslash U_{j}$. Hereafter, we can rearrange the functions $|g_{jl}|$ by raising them to higher powers, as follows:
  \begin{equation}\label{eq gjl}
      \begin{aligned}
          \sum_{l}|g_{jl}(x)|^2&<2^{-j}\,,\quad\text{for $x\in K_{j}$\,,}\\
          \sum_{l}|g_{jl}(x)|^2&>j\,,\quad\text{for $x\in K_{j+2}\backslash U_{j}$\,.}
      \end{aligned}
  \end{equation}
By GH regularity of $X$ (see $(3)$ in Definition \ref{main def}), we can assert that among the GH functions $g_{jl}\,\,(l=1,\ldots,l_{j})$, there exist $k$ functions that form a part of a local coordinate system at any point in $K_{j}$. In other words, we can find $k$ GH functions among them that form a local coordinate system in the transversal direction, of the symplectic foliation $\mathscr{S}$, at any point in $K_{j}$. Define
$$f(x):=\sum_{j,l}|g_{jl}(x)|^{2}-1\,\,\,\text{for $x\in X$}\,.$$
By \eqref{eq gjl}, $f$ is well defined and $f(x)>j-1$ when $x\in X\backslash U_{j}\,.$ Set, $$\hat{f}(x,y):=\sum_{j,l}g_{jl}(x)\overline{g_{jl}(y)}\quad\text{for $(x,y)\in X\times X$}\,.$$ Again applying \eqref{eq gjl}, observe that $\hat{f}$ is an uniformly convergent series on compact subsets of $X\times X\,.$ Therefore it is a GH function in $x$-coordinate and its complex conjugate is GH function in $y$-coordinate. Following the proof of \cite[Theorem 5.1.6]{hor}, we deduce that $f$ is a smooth real function. Now, $d_{\mathscr{S}}f=0$, as $d_{\mathscr{S}}g_{jl}=0$ (cf. Proposition \ref{imp corr}). In simpler terms, $f$ is constant along the leaves of $\mathscr{S}$.

\vspace{0.5em}
Note that, by Proposition \ref{imp corr}, $id_{\overline{L}}g_{jl}\wedge d_{L}\overline{g_{jl}}\geq 0\,.$ Consequently, using Remark \ref{rmk pluri}, $f$ is $L$-plurisubharmonic function. Now, if for some $x\in X$, in local coordinates (cf. \eqref{loc coordi}), 
$$\sum^{k}_{m=1}w_{m}\frac{\partial g_{jl}(x)}{\partial z_{m}}=0\quad\text{for all $j$ and $l$}\,,$$ then $w_{m}=0$ for all $m\in\{1,\ldots,k\}$, because there exist $k$ number of GH functions $g_{jl}$ that form a local coordinate system in the transversal direction at $x$. Again following the proof of \cite[Theorem 5.1.6]{hor}, we imply that $f$ is also strictly $L$-plurisubharmonic, thereby concluding the proof.
\end{proof}
\subsection{$L^{2}$ theory and GC manifold}\label{l2-1}

In this subsection, we extend some classical results related to the $L^2$-theory of $\C$-valued differential forms on complex manifolds and present the converse of Theorem \ref{pluri thm}, following the approaches of  \cite[Chapter I]{batu17}, \cite[Chapters 4-5]{hor, krantz}.

\medskip
Consider a hermitian metric $h$ on the complex vector bundle $E^{1,1}$ (cf. \eqref{A-pq}) over a regular GC manifold $X^{2n}$. On a local trivialization $U\subset X$, $h$ is given by a function, again denoted by $h$, that associates to any $x\in U$, a positive-definite hermitian matrix $h(x)=(h_{ij}(x))$. Applying the Gram-Schmidt orthogonalization process and sufficiently shrinking $U$, we can choose $e_1,\ldots,e_{2n}\in A^{1,0}(U)$ such that $h(e_{j},e_{l})=\delta_{jl}$ for $j,l\in\{1,\ldots,2n\}$ where $\delta_{jl}$ is the Kronecker delta. 

\vspace{0.5em}
Similarly, we can get the same result by considering a hermitian metric on $E^{p,q}$. For notational purposes, whenever we reference a hermitian metric on $E^{p,q}$, we are implying a hermitian metric with preceding characteristic and local orthonormal frame. Hence, for a smooth section $s\in A^{p,q}(X)$ over a local trivialization $U$, we can express it uniquely as
\begin{equation}\label{sec}
    s=\sum_{I,J}s_{IJ}\,e_{I}\wedge \overline{e_{J}}\quad\text{for some $s_{IJ}\in C^{\infty}(U,\C)$}\,,
\end{equation}
where
$I, J$ are ordered subsets of $\{1, \ldots, 2n \}$, and
$e_{I}= \bigwedge_{i \in I} e_{i}$, $\overline{e_J} = \bigwedge_{j \in J} \overline{e_{j}} $ . Considering $h$ on $E^{p,q}$ (cf. \eqref{A-pq}), we  set
\begin{equation}\label{s2}
    |s|^2:=h(s,s)\,,
\end{equation} which is locally of the form $\frac{1}{p!q!}\sum_{IJ}|s_{IJ}|^2$. Note that \eqref{s2} is independent of the choice of smooth local orthonormal frame $\{e_1,\ldots,e_{2n}\}$.

\vspace{0.5em}
Let $C^{\infty}_{c}(X,\C)$ denote the space of compactly supported smooth $\C$-valued functions. We choose a sequence $\{\eta_{j}\}\in C^{\infty}_{c}(X,\C)$ such that 
\begin{itemize}
\setlength\itemsep{0.4em}
    \item $0\leq\eta_{j}\leq 1\,,$ and 
    \item $\eta_{j}=1$ on any compact subset of $X$ when $j$ is large.
\end{itemize}
Now given a hermitian metric $h$ on $E^{p,q}$, we can find a positive real function $\eta\in C^{\infty}(X,\C)$ such that $|d_{L}\eta_{j}|\leq \eta$ for all $j\in\N$. Thus, we can substitute $h$ with $h^{'} := \eta^2\, h$, which implies that, with respect to $h^{'}$, we obtain 
\begin{equation}\label{nj eq}
    |d_{L}\eta_{j}|\leq 1\quad\text{on $X$ for all $j\in\N$}\,.
\end{equation}

\begin{center}
\textit{From this point forward and throughout the remainder of this section, we keep the hermitian structure $h^{'}$ and the sequence $\{\eta_{j}\}$ fixed.} 
\end{center}

Let $\psi\in C^2(X,\C)$ be a real function and $\lambda$ be strictly positive smooth measure on $X$. We mean by this that $d\lambda$ is a volume element (or density) of $X$. Let $\Gamma(X, E^{p,q})$ denote the space of (not necessarily continuous) sections. In particular, $\Gamma(X,E^{0,0})$ denote the space of $\C$-valued functions.

\vspace{0.5em}
For any $p,q\geq 0$, consider the following spaces:
\begin{equation}
    \begin{aligned}
        L^{2}(X,\psi):=&\Big\{f\in C^{\infty}(X,\C)\,\big |\,\text{f is measurable and $|f|^2_{\psi}:=\int |f|^2\,e^{-\psi}\,d\lambda$}<\infty\Big\}\,;\\
        L^{2}(X,E^{p,q},\psi):=&\Big\{s\in \Gamma(X,E^{p,q})\,\big |\,\text{ The coefficients of $s$ (cf.\eqref{s2}) are measurable on}\\
        &\text{any local trivialization of $E^{p,q}$, and $|s|^2_{\psi}:=\int |s|^2\,e^{-\psi}\,d\lambda$}<\infty\Big\}\,;\\
        L^{2}(X,E^{p,q}):=&L^{2}(X,E^{p,q},0)\,;\\
        L^2(X,E^{p,q},loc):=&\Big\{s\in \Gamma(X,E^{p,q})\,\big |\,\int_{K} |s|^2\,d\lambda <\infty\,\,\,\text{on every compact set $K\subset X$}\Big\}\,;\\
        C^{\infty}_{c}(X,E^{p,q}):=&\Big\{s\in L^2(X,E^{p,q},\psi)\,\big |\,\text{On any local trivialization $U$, $s_{IJ}\in C^{\infty}_{c}(U,\C)$}\\
        &\text{where $s_{IJ}$ as defined in \eqref{s2}}\Big\}\,.
\end{aligned}
\label{spc}
\end{equation}
More generally, for any $p,q\geq 0$, $r\in [1,\infty]$, we define
\begin{equation}
\begin{aligned}
        L^{r}(X,E^{p,q},loc):=&\Big\{s\in \Gamma(X,E^{p,q})\,\big |\,\text{On any local trivialization $U$, $s_{IJ}\in L^{r}_{loc}(U)$}\\
        &\text{where $s_{IJ}$ as defined in \eqref{s2}}\Big\}\,;\\
        W^{m,r}(X,E^{p,q},loc):=&\Big\{s\in L^{r}(X,E^{p,q},loc)\,\big |\,\text{The differential order of $s$ is $m$, that is,}\\
        &\text{on any local trivialization $U$, $s_{IJ}\in W^{m,r}_{loc}(U)$ where $s_{IJ}$}\\
        &\text{as defined in \eqref{s2} and $m\in \N\cup\{0\}$}\Big\}\,;\\
        W^{r}(X,E^{p,q},loc):=&W^{0,r}(X,E^{p,q},loc)=L^{r}(X,E^{p,q},loc)\,.
    \end{aligned}
    \label{spc1}
\end{equation}
Here $L^{r}_{loc}(U)$ and $W^{m,r}_{loc}(U)$ are standard Sobolev spaces of $\C$-valued functions on $U\subset X$. Note that the subspace $C^{\infty}_{c}(X, E^{p,q})$ is dense in the Hilbert space $L^2(X, E^{p,q},\psi)$. Therefore, the differential operator $d_{L}$, as in \eqref{dL-dL bar}, defines (unique) linear, closed, and densely defined operator, 
\begin{equation}\label{tilde dl}
    \tilde{d}_{L,q}:L^2(X,E^{p,q},\psi_1)\longrightarrow L^2(X,E^{p,q+1},\psi_2)\,,
\end{equation}
for $\psi_{j}\in C^2(X,\C)$ $(j=1,2)$ real functions, as follows: 

\vspace{0.5em}
Let $s\in L^{2}(X,E^{p,q},\psi_1)$. We say $s\in\dom( \tilde{d}_{L,q})$ if $d_{L}s$, when considered in the distributional sense (see \cite[Chapter I]{batu17}), belongs to $L^2(X,E^{p,q+1},\psi_2)$. In this case, we define $\tilde{d}_{L,q}s=d_{L}s$. Since $d_{L}$ is continuous operator in distributional sense, and $C^{\infty}_{c}(X,E^{p,q})\subset\dom( \tilde{d}_{L,q})$, $\tilde{d}_{L,q}$ is closed, and densely defined.

\begin{prop}\label{lp prop}
 Considering \,$\psi:=\psi_1=\psi_2$ in \eqref{tilde dl} and the graph norm on $L^2(X,E^{p,q+1},\psi)$, namely,
 $$s\rightarrow |s|_{\psi}+|\tilde{d}^{*}_{L,q}s|_{\psi}+|\tilde{d}_{L,q+1}s|_{\psi}\,,$$
 the subspace $C^{\infty}_{c}(X,E^{p,q+1})$ of \,$\dom(\tilde{d}^{*}_{L,q})\cap\dom(\tilde{d}_{L,q+1})$ is dense where $\tilde{d}^{*}_{L,q}$ denotes the adjoint operator of\, $\tilde{d}_{L,q}$ (cf. \cite[p.~18]{batu17}) with respect to the norm $|\,\cdot\,|_{\psi}$.
\end{prop}
\begin{proof}
    Observe that for any $s\in \dom(\tilde{d}_{L,q+1})$, we have $$\tilde{d}_{L,q+1}(\eta_{j}s)-\eta_{j}\tilde{d}_{L,q+1}s=d_{L}\eta_{j}\wedge s\quad\text{for $j\in\N$}\,.$$ As a result, $|\tilde{d}_{L,q+1}(\eta_{j}s)-\eta_{j}\tilde{d}_{L,q+1}s|^2\leq |s|^2$ follows by utilizing \eqref{nj eq}. Furthermore, we have
    $$\Big\langle\big(\tilde{d}^{*}_{L,q}(\eta_{j}s^{'})-\eta_{j}\tilde{d}^{*}_{L,q}s^{'}\big),u\Big\rangle_{\psi}=-\Big\langle s^{'},\big(d_{L}\eta_{j}\wedge u\big)\Big\rangle_{\psi}$$
    for any $s^{'}\in\dom(\tilde{d}^{*}_{L,q})$ and $u\in C^{\infty}_{c}(X,E^{p,q})$ where $\Big\langle,\Big\rangle_{\psi}$ denotes the inner product on $L^{2}(X,E^{p,q},\psi)$ for any $p,q\geq 0$. Again by \eqref{nj eq}, it follows that 
    $$|\tilde{d}^{*}_{L,q}(\eta_{j}s^{'})-\eta_{j}\tilde{d}^{*}_{L,q}s^{'}|^2\leq |s^{'}|\,.$$ So $\{\eta_{j}s\}_{j\in\N}$ converges to $s$ in the graph norm for $s\in\dom(\tilde{d}^{*}_{L,q})\cap\dom(\tilde{d}_{L,q+1})\,.$ Therefore, it is sufficient to approximate elements of compact support in $\dom(\tilde{d}^{*}_{L,q})\cap\dom(\tilde{d}_{L,q+1})$, and using partition of unity, we can simplify the proof to the case where the support is contained within a local trivialization of $E^{p,q+1}$ over $X$. Without loss of generality, by Theorem \ref{darbu thm}, we can assume the local trivialization is provided by a local coordinate system (cf. \eqref{loc coordi}).
    
    \vspace{0.5em}
    Let $s\in \dom(\tilde{d}^{*}_{L,q})\cap\dom(\tilde{d}_{L,q+1})$ be an element which has compact support. Then by \cite[Section I.3]{batu17}, \cite[Lemma 5.2.2 and Lemma 4.1.4]{hor} and applying the operator $J_{\epsilon}$ to each component of $s$ in the local trivialization, we obtain, for $\epsilon\rightarrow 0$,
    \begin{equation*}
        \begin{aligned}
           &|\tilde{d}^{*}_{L,q}(J_{\epsilon}s)-J_{\epsilon}\tilde{d}^{*}_{L,q}s|_{\psi}\rightarrow 0\,,\\
          \implies &|\tilde{d}^{*}_{L,q}(J_{\epsilon}s)-\tilde{d}^{*}_{L,q}s|_{\psi}\rightarrow 0\,.   
        \end{aligned}
    \end{equation*}
    where $J_{\epsilon}$ is defined as in \cite[Lemma 5.2.2]{hor}. Similarly, we can show that $$|\tilde{d}_{L,q+1}(J_{\epsilon}s)-\tilde{d}_{L,q+1}s|_{\psi}\rightarrow 0\,.$$ This completes the proof.
\end{proof}
\begin{remark}
    Note that Proposition \ref{lp prop} can be further generalized for any two real functions $\psi_1,\psi_2\in C^2(X,\C)$ by straightforwardly modifying the proof of Proposition \ref{lp prop} using the approach of \cite[Lemma 4.1.3]{hor}.
\end{remark}

Now let $\psi\in C^{\infty}(X,\C)$ be a smooth real function.  Consider $U$ as a local trivialization of $E^{p,q}$ provided by a local coordinate system (cf. Theorem \ref{darbu thm} and \eqref{loc coordi}). Let $e_1,\ldots,e_{2n}\in A^{1,0}(U)$ be a smooth orthonormal frame with respect to $h'$. Then, following \eqref{sec}, we can express
$$Ds=\sum^{2n}_{j=1}\Big(\frac{\partial s}{\partial e_j}\,e_{j}\,+\,\frac{\partial s}{\partial\overline{ e_j}}\,\overline{e_{j}}\Big)\,,$$ as a definition of the first-order linear differential operators $\frac{\partial}{\partial e_j}$ and $\frac{\partial }{\partial\overline{ e_j}}$ for any $s\in A^{0}(U)$ where $D$ as in \eqref{dL-dL bar}. It follows that $d_{L}s=\sum^{2n}_{j=1}\frac{\partial s}{\partial\overline{ e_j}}\,\overline{e_{j}}\,.$ Again following \eqref{sec}, more generally we have, for any $s\in A^{p,q}(U)$,
\begin{align*}
  d_{L}s&= \sum_{I,J}d_{L}s_{IJ}\,e_{I}\wedge \overline{e_{J}}\,+\,\sum_{I,J}s_{IJ}\,d_{L}e_{I}\wedge \overline{e_{J}}\,+(-1)^{|I|}\,\sum_{I,J}s_{IJ}\,e_{I}\wedge d_{L}\overline{e_{J}}\,,\\ &=\hat{A}s+\sum_{I,J}s_{IJ}\,d_{L}e_{I}\wedge \overline{e_{J}}\,+(-1)^{|I|}\,\sum_{I,J}s_{IJ}\,e_{I}\wedge d_{L}\overline{e_{J}}\,.
\end{align*}
where  $$\hat{A}s:=\sum_{I,J}\sum^{2n}_{j=1}\frac{\partial s_{IJ}}{\partial\overline{ e_j}}\,\overline{e_{j}}\wedge e_{I}\wedge \overline{e_{J}}$$ Note that $d_{L}e_{i}$ and $d_{L}\overline{e_{j}}$ do not necessarily vanish, so $\hat{A}s$ may not be zero. Consequently, when support of $s$ lies in a fixed compact subset of $U$, we have
\begin{equation}\label{1}
    |d_{L}s-\hat{A}s|\leq C''\,|s|
\end{equation}
where $C''>0$ is independent of $s$. Now, consider $s\in C^{\infty}_{c}(X, E^{p,q+1})$ and any $s'\in C^{\infty}_{c}(U, E^{p,q})$. Then
\begin{equation}\label{2}
\begin{aligned}
 \int h'(\tilde{d}^{*}_{L,q}s\,,\,s')\,e^{-\psi}\,d\lambda&=\int h'(s\,,\,d_{L,q}s')\,e^{-\psi}\,d\lambda\\
 &=(-1)^p\sum_{I,J}\sum^{2n}_{j=1}\int\,s_{I,jJ}\,\overline{\left(\frac{\partial s'_{IJ}}{\partial\overline{ e_j}}\right)}\,e^{-\psi}\,d\lambda\quad\text{(by \eqref{sec})}\\
 &+\sum_{I,J}\int s_{IJ}\,\overline{s'_{IJ}}\,h'\big(e_{I}\wedge \overline{e_{J}}\,,\,d_{L}(e_{I}\wedge\overline{e_{J}})\big)\,e^{-\psi}\,d\lambda\,.\\
 \end{aligned}
\end{equation}
Observe that the expression $\sum_{I,J}\int s_{IJ}\,\overline{s'_{IJ}}\,h'\big(e_{I}\wedge \overline{e_{J}}\,,\,d_{L}(e_{I}\wedge\overline{e_{J}})\big)\,e^{-\psi}\,d\lambda$ in \eqref{2} does not include any differentiation, since $\{e_1,\ldots,e_{2n}\}$ is an orthonormal frame with respect to $h'$. So, following \cite[p.~122]{hor} and applying Integration by parts in \eqref{2}, we get
\begin{equation}\label{3}
\tilde{d}^{*}_{L,q}s=(-1)^{p-1}\sum_{I,J}\sum^{2n}_{j=1}\delta_{j}(s_{I,jJ})\,e_{I}\wedge \overline{e_{J}}+\cdots   
\end{equation}
where $\delta_{j}(s_{I,jJ}):=e^{\psi}\,\frac{\partial (s_{I,jJ}\,e^{-\psi})}{\partial e_{j}}$, and the dots represent terms that do not involve the differentiation of $s_{I,j'J}$ and are independent of $\psi$. Hence 
\begin{equation}\label{4}
    |\tilde{d}^{*}_{L,q}s-\hat{B}s|\leq C'\,|s|\quad\text{where}\quad\hat{B}s:=(-1)^{p-1}\sum_{I,J}\sum^{2n}_{j=1}\delta_{j}(s_{I,jJ})\,e_{I}\wedge \overline{e_{J}}\,,
\end{equation}
for some constant $C'>0$, provided $s$ has its support within a fixed compact subset of $U$. Using \eqref{1} and \eqref{4}, and introducing another constant $C>0$ (such as $C=\max\{C',C''\}$), independent of $s$ and $\psi$, we obtain
\begin{equation}\label{5} \big(|\hat{A}s|^2_{\psi}+|\hat{B}s|^2_{\psi}\big)\leq C\,|s|+2\big(|\tilde{d}^{*}_{L,q}s|^2_{\psi}+|d_{L,q+1}s|^2_{\psi}\big)
\end{equation}
Note that 
\begin{equation}\label{6}
\begin{aligned}
|\hat{A}s|^2_{\psi}+|\hat{B}s|^2_{\psi}&=\sum_{I,J}\sum^{2n}_{j=1}\int\bigg|\frac{\partial s_{IJ}}{\partial\overline{ e_j}}\bigg|^{2}\,e^{-\psi}\,d\lambda\\
&+\sum_{I,J}\sum^{2n}_{j,l=1}\int\left(\delta_{j}(s_{I,jJ})\overline{\delta_{l}(s_{I,lJ})}-\frac{\partial s_{I,jJ}}{\partial\overline{ e_l}}\overline{\left(\frac{\partial s_{I,lJ}}{\partial\overline{ e_j}}\right)}\right)\,e^{-\psi}\,d\lambda\,.
\end{aligned}
\end{equation}

Let $\gamma$ be the continuous function on $U$ that represents the smallest eigenvalue at each point of the hermitian metric $e^{-\psi}h^{'}$ on $E^{p,q}|_{U}$ for the real function $\psi\in C^{\infty}(X,\C)$. Since, on a local trivialization, $e^{-\psi}h^{'}$ is a positive-definite hermitian matrix at each point, $\gamma$ is positive. Therefore, by making a straightforward modification to \cite[Equations 5.2.7-5.2.9, p.~124]{hor} with $\psi$,  we obtain the following estimate on $U$, which is similar to \cite[Equation 5.2.10]{hor}, for all $s\in C^{\infty}_{c}(X, E^{p,q+1})$ with support lying in a fixed compact subset of $U$,
\begin{equation}\label{estimate1}
    \int\sum_{I,J}|s_{IJ}|^{2}\,\gamma\,e^{-\psi}\,d\lambda\,+\,\frac{1}{2}\sum_{I,J}\sum^{2n}_{j=1}\int\bigg|\frac{\partial s_{IJ}}{\partial\overline{ e_j}}\bigg|^{2}\,e^{-\psi}\,d\lambda\leq 2\big(C|s|^2_{\psi}+|\tilde{d}^{*}_{L,q}s|^2_{\psi}+|\tilde{d}_{L,q+1}s|^2_{\psi}\big)
\end{equation}
where $C>0$ is some constant, independent of $s$ and $\psi$.

\vspace{0.5em}
Let $s\in C^{\infty}_{c}(X,E^{p,q})$ and $\{U_{j}\}_{j\in\N}$ denote a collection of local trivializations of $E^{p,q}$ over $X$, provided by local coordinate systems (cf. Theorem \ref{darbu thm} and \eqref{loc coordi}) of $X$, ensuring that \eqref{estimate1} is applicable. Let $\phi_{j}\in C^{\infty}_{c}(U_{j},\C)$ be real functions such that $\sum\phi^{2}_{j}=1$ on $X$. Now applying \eqref{estimate1} to each section $\phi_{j}s$ on $U_{j}$, we get
  $$\int_{U_{j}}\phi^{2}_{j}\,|s|^{2}\,\gamma_{j}\,e^{-\psi}\,d\lambda\leq 4\Big(C_{j}\int_{U_{j}}\phi^{2}_{j}|s|^{2}\,e^{-\psi}\,d\lambda\Big)+4\Big(|\phi_{j}\tilde{d}^{*}_{L,q}s|^2_{\psi}+|\phi_{j}\tilde{d}_{L,q+1}s|^2_{\psi}\Big)\Big)\,,$$
where $C_{j}$'s are positive constants. Therefore, by summing over $j$  on both sides and applying $\sum\phi^{2}_{j}=1$, we obtain
\begin{align*}
\sum_{j}\int_{U_{j}}\big(\phi^{2}_{j}\gamma_{j}\big)\,|s|^{2}\,e^{-\psi}\,d\lambda&\leq \sum_{j}\int_{U_{j}}\big(4\phi^{2}_{j}C_{j}\big)\,|s|^{2}\,e^{-\psi}\,d\lambda + 4\sum_{j} \Big(|\phi_{j}\tilde{d}^{*}_{L,q}s|^2_{\psi}+|\phi_{j}\tilde{d}_{L,q+1}s|^2_{\psi}\Big)\,.
\end{align*}
Consider two non-negative continuous functions
$$\gamma_{\psi}:=\sum\phi^{2}_{j}\gamma_{j}:X\longrightarrow\R\,,\quad\text{and}\quad \hat{C}:=4\sum\phi^{2}_{j}C_{j}:X\longrightarrow\R\,.$$
Note that using monotone convergence theorem, we get
\begin{align*}
&\sum_{j}\int_{U_{j}}\big(\phi^{2}_{j}\gamma_{j}\big)\,|s|^{2}\,e^{-\psi}\,d\lambda=\int\gamma_{\psi}\,|s|^{2}\,e^{-\psi}\,d\lambda\,,\\ 
&\sum_{j}\int_{U_{j}}\big(4\phi^{2}_{j}C_{j}\big)\,|s|^{2}\,e^{-\psi}\,d\lambda=\int \hat{C}\,|s|^{2}\,e^{-\psi}\,d\lambda\,,\\
\text{and}\quad &\sum_{j} \Big(|\phi_{j}\tilde{d}^{*}_{L,q}s|^2_{\psi}+|\phi_{j}\tilde{d}_{L,q+1}s|^2_{\psi}\Big)=\Big(|\tilde{d}^{*}_{L,q}s|^2_{\psi}+|\tilde{d}_{L,q+1}s|^2_{\psi}\Big)\,.
\end{align*}
Hence, we have proved the following.
\begin{theorem}\label{est thm}
    Let $\psi\in C^{\infty}(X,\C)$ be a real function. For any $s\in C^{\infty}_{c}(X,E^{p,q})$, there exists a positive real smooth function $\hat{C}\in C^0(X,\C)$ such that the estimate
    $$\int\Big(\gamma_{\psi}-\hat{C}\Big)\,|s|^{2}\,e^{-\psi}\,d\lambda\leq 4\Big(|\tilde{d}^{*}_{L,q}s|^2_{\psi}+|\tilde{d}_{L,q+1}s|^2_{\psi}\Big)\,,$$ holds, where $\gamma_{\psi}$ is a positive continuous function that provides the smallest eigenvalues of the hermitian metric $e^{-\psi}h^{'}$ over a system of local trivializations of $E^{p,q}$.
\end{theorem}
\begin{lemma}\label{l2 lemma}
    Let $U\subseteq X$ be an open subset. Then, for any function $f\in L^2_{loc}(U)$, there exists a real function $\phi_{f}\in C^{\infty}(U,\C)$ such that $f\in L^{2}(U,\phi_{f})$. 
\end{lemma}
\begin{proof}
   Let $f\in L^2_{loc}(U)$. Let $\{U_{j}\}_{j\in\N}$ be a locally finite cover of $U$ such that $U_{j}\subset\subset U$ for each $j\in\N$. Set $$I_{j}(f):=\int_{U_{j}}|f|^2\,d\lambda\quad\forall\,\,j\,.$$ Now $I_{j}(f)$ is finite and $I_{j}(f)\geq 0$. Consider an increasing sequence $\{m_{j}\}_{j\in\N}$ in $(0,\infty)$, and define $$\phi_{f}:=\sum_{j}m_{j}\,\phi_{j}\geq 0\quad\text{on $U$}\,,$$ where $\{\phi_{j}\}_{j\in\N}$ be a smooth partition of unity such that $\supp(\phi_{j})\subset U_{j}$ for each $j\in\N$.

\vspace{0.4em}
 Now for each $U_{j}$, there exists an $n(j)\in\N$ such that $\phi_{f}>m_{n(j)}$. This is because there is a unique $x_0\in \overline{U_{j}}$ such that $\phi_{f}(x_0)=\inf_{x\in \overline{U_{j}}}(\phi_{f}(x))$. Therefore, $\inf_{x\in U_{j}}(\phi_{f}(x))\geq\phi_{f}(x_0)$. Given that $\{m_j\}$ is increasing, it follows that $\phi_{f}(x_0)=\sum_{j<\infty}m_{(j)}\phi_{j}(x_0)\geq m_{n(j)}\Big(\sum_{j<\infty}\phi_{j}(x_0)\Big)$, where $m_{n(j)}=\min_{j<\infty}\{m_{j}\}$. If we define $\phi'_{f}:=\sum_{j}m_{n(j)}\,\phi_{j}$ on $U$, we obtain $\phi'_{f}>m_{n(j)}$ on $U_j$. Therefore, we can choose $\{m_{j}\}$ such that
 \begin{enumerate}
 \setlength\itemsep{0.4em}
     \item $\phi_{f}>m_{j}$ on $U_{j}$ for each $j$.
     \item $\lim_{j\to\infty}\,m_{j}=\infty$, and $I_{j}(f)\,e^{-m_{j}}<\frac{1}{2^j}$ for each $j$. 
 \end{enumerate} Thus we have the following
\begin{align*}
    \int_{\cup U_{j}}|f|^2\,e^{-\phi_{f}}\,d\lambda &\leq\sum_{j}\int_{U_{j}}|f|^2\,e^{-\phi_{f}}\,d\lambda\quad\text{(due to sub-additivity)}\,,\\ 
    &\leq\sum_{j}\,I_{j}(f)\,e^{-m_{n(j)}}\leq\sum_{j}\,I_{j}(f)\,e^{-m_{j}}<\sum_{j}\frac{1}{2^j}<\infty\,.
\end{align*}
Hence $\int_{U}|f|^2\,e^{-\phi_{f}}\,d\lambda<\infty\,,$ implying $f\in L^{2}(U,\phi_{f})$.
\end{proof}
\begin{remark}
The statement in Lemma \ref{l2 lemma} can be extended to topological manifolds with a simple modification. A similar result can be found in the literature for complex manifolds in \cite[Chapter 4]{hor, krantz}. However, the explanation provided there is not very explicit.
\end{remark}
\begin{lemma}\label{l2 lemma1}
    Let $g: X\longrightarrow\R$ be a positive continuous function. Let $\psi: X\longrightarrow\R$ be a strictly $L$-plurisubharmonic function on $X$ such that the sublevel sets $\big\{x\in X\,\big|\,\psi(x)<c\big\}$ are compactly contained in $X$ for every $c\in\R$. Then there exists a smooth increasing convex function $\chi:\R\longrightarrow\R$ such that $\chi\circ \psi\geq g$.
\end{lemma}
\begin{proof}
Consider the continuous function $\eta:\R\longrightarrow\R$ defined by
   \[
   \eta(t) = \sup\big\{g(x)\big |\,\psi(x)\leq t\big\}\quad\forall\,\,t\in\R\,.
   \]
For each $t\in\R$, $\eta(t)$ is the supremum of $g$ over the compact set $\big\{x\in X\big |\,\psi(x)\leq t\}$. Now, if $t_1\leq t_2$, then $\big\{x\big |\,\psi(x)\leq t_1\big\}$ is a subset of $\big\{x\big |\,\psi(x)\leq t_2\big\}$, which implies that $\eta(t_1)\leq\eta(t_2)$. Note that $\eta$ is generally not convex. So, consider the convex envelope of $\eta$, denoted by $\chi$, defined as 
   \[
   \chi(t):= \inf \left\{ \sum_{j=1}^m \lambda_j \eta(t_j) \Bigg|\,\sum_{j=1}^m \lambda_j t_j = t, \sum_{j=1}^m \lambda_j = 1, \lambda_j \geq 0 \right\}\,.
   \]
Since $\eta$ is increasing, $\chi$ is also increasing. By definition of $\chi$ and $\eta$, we have $\chi\leq\eta$, and $\eta\circ\psi \geq g$. Fix $x\in X$. Since $x\in\big\{x\big |\,\psi(x)\leq t\big\}$ for some $t$, we have $g(x)\leq\eta(\psi(x))$. Furthermore, because $\chi(\psi(x))$ is the infimum at $\psi(x)$, it follows that $\chi(\psi(x))\geq g(x)$. Therefore, $\chi\circ\psi\geq g$.

\vspace{0.5em}
Now, the function $\chi$ does not have to be smooth. Therefore, using mollification, define the following.
   \[
   \chi_\epsilon(t'):= \int_{\R} \chi(t)\,\rho_\epsilon(t'-t)\,dt=\int_{\R} \chi(t'-t)\,\rho_\epsilon(t)\,dt\quad\forall\,\,t'\,\in\R\,,
   \]
   where $\rho_\epsilon$ is a smooth mollifier. Note that $\chi_\epsilon$ is an increasing smooth function that approximates $\chi$, that is, $\chi_{\epsilon}(t)\to\chi(t)$ uniformly as $\epsilon\to 0$. Consequently, for all $x\in X$, $\chi_{\epsilon}(\psi(x))\geq\chi(\psi(x))$ for sufficiently small $\epsilon$. Hence $\chi_{\epsilon}\circ\psi\geq g$ for sufficiently small $\epsilon$. This completes the proof.
\end{proof}
\begin{remark}
   A similar assertion to that in Lemma \ref{l2 lemma1} is also mentioned in \cite[Chapter 4]{hor, krantz} for pseudoconvex domains. However, the argument presented there is not very explicit. 
\end{remark}
\begin{theorem}\label{pluri thm1}
    Let $X^{2n}$ be a regular GC manifold of type $k>0$. Suppose there exists a strictly $L$-plurisubharmonic function $\psi$ (cf. Definition \ref{l-pluri}) on $X$ such that $$\big\{x\in X\,\big|\,\psi(x)<c\big\}\subset\subset X\,\,\,\text{for every $c\in\R$}\,.$$ Then the equation $$d_{L}s=s'\quad\text{(in the sense of distribution theory)}$$ has a solution $s\in L^2(X, E^{p,q}, loc)$ for every $s'\in L^2(X, E^{p,q+1}, loc)$ such that $d_{L}s'=0\,.$
\end{theorem}
\begin{proof}
    Let $\chi:\R\longrightarrow\R$ be a smooth convex increasing function. Let $\hat{\psi}:=\chi\circ\psi$. Replace $\psi$ with $\hat{\psi}$ in Theorem \ref{est thm}. On a local trivialization $U$ of $E^{p,q+1}$, given by a local coordinate system, we obtain from Remark \ref{rmk pluri} that for any $(w_1,\ldots,w_{k})\in\C^{k}$ (see \eqref{loc coordi}), 
    $$\sum^{k}_{i,j=1}\frac{\partial^{2}\psi}{\partial z_{i}\partial\overline{z_{j}}}w_{i}\overline{w_{j}}\geq g\sum^{k}_{j=1}w_{i}\overline{w_{j}}\quad\text{for some positive real function $g\in C^0(U,\C)$}\,.$$ Consequently,
    $$\sum^{k}_{i,j=1}\frac{\partial^{2}\hat{\psi}}{\partial z_{i}\partial\overline{z_{j}}}w_{i}\overline{w_{j}}\geq\chi'(\psi)\,g\sum^{k}_{j=1}w_{i}\overline{w_{j}}\quad\text{on $U$}\,,$$ where $\chi'(a):=\frac{d\chi}{dt}\big|_{t=a}$ for $a\in\R$. Therefore, the corresponding $\gamma_{\psi}$ in Theorem \ref{est thm} is replaced by $\chi'(\psi)\gamma_{\psi}$. Hence, by Theorem \ref{est thm}, we deduce that
    $$\int\Big(\chi'(\psi)\gamma_{\psi}-\hat{C}\Big)\,|s|^{2}\,e^{-\hat{\psi}}\,d\lambda\leq 4\Big(|\tilde{d}^{*}_{L,q}s|^2_{\hat{\psi}}+|\tilde{d}_{L,q+1}s|^2_{\hat{\psi}}\Big)\quad\text{for any $s\in C^{\infty}_{c}(X,E^{p,q+1})$}\,,$$ where $\tilde{d}^{*}_{L,q}$ denotes the adjoint of $\tilde{d}_{L,q}$ with respect to $|\cdot|_{\hat{\psi}}$. We can choose $\chi$ so that $\chi'(\psi)$ increases rapidly enough such that $\big(\chi'(\psi)\gamma_{\psi}-\hat{C}\big)\geq 4$. Then using Proposition \ref{lp prop}, we get    $$|s|^{2}_{\hat{\psi}}\leq\Big(|\tilde{d}^{*}_{L,q}s|^2_{\hat{\psi}}+|\tilde{d}_{L,q+1}s|^2_{\hat{\psi}}\Big)\quad\text{for any $s\in \dom(\tilde{d}^{*}_{L,q})\cap\dom(\tilde{d}_{L,q+1})$}\,.$$ It follows that $\ker(\tilde{d}_{L,q+1})=\img(\tilde{d}_{L,q})$ by \cite[Lemma 4.1.1]{hor}. So for any $s'\,\in L^{2}(X,E^{p,q+1},\hat{\psi})$ with $\tilde{d}_{L,q+1}s'=0$, there exists $s\in L^{2}(X,E^{p,q+1},\hat{\psi})$ such that $\tilde{d}_{L,q}s=s'$ and $s$ can be chosen so that $|s|_{\hat{\psi}}\leq |s'|_{\hat{\psi}}$. 

    \vspace{0.5em}
    Note that, by definition, $L^{2}(X,E^{p,q},\hat{\psi})\subseteq L^2(X, E^{p,q}, loc)$ for $p,q\geq 0$. We can select $\chi$ so that in addition any given $s\in L^2(X, E^{p,q}, loc)$ also belongs to $L^{2}(X,E^{p,q},\hat{\psi})$. This is feasible because every function in $L^2_{loc}(X)$ belongs to $L^{2}(X,g)$ for some positive real function $g\in C^0(X,\C)$ by Lemma \ref{l2 lemma}, and we choose $\chi$ so that $\hat{\psi}\geq g$ (cf. Lemma \ref{l2 lemma1}). This concludes the proof.
\end{proof}
 We can apply the proof of Theorem \ref{pluri thm1} to the more general space $W^{m,2}(X,E^{p,q+1},loc)$ with evident modification, following the approach of \cite[Theorem 4.2.5, Theorem 5.2.5]{hor}. For sections $s\in\dom(\tilde{d}^{*}_{L,q})\cap\dom(\tilde{d}_{L,q+1})$ with support in a local trivialization provided by a local coordinate system (cf. \ref{loc coordi}), we can obtain estimates of the operators $\frac{\partial s_{IJ}}{\partial\overline{ e_j}}$ from \ref{estimate1}. This allows us to show the following generalized version of Theorem \ref{pluri thm1}.
\begin{theorem}\label{pluri thm2}
Consider the same setting as in the preceding theorem. Then the equation $d_{L}s=s'$ (in the sense of distribution theory) has a solution $s\in W^{m+1,2}(X,E^{p,q},loc)$ for every $s'\in W^{m,2}(X,E^{p,q+1},loc)$ such that $d_{L}s'=0\,.$ In particular, every solution of $d_{L}s=s'$ has this property when $q=0$.
\end{theorem}
\begin{theorem}\label{apq thm}
   Let $X^{2n}$ be a regular GC manifold of type $k>0$. Suppose there exists a strictly $L$-plurisubharmonic function $\psi$ on $X$ such that $\big\{x\in X\,\big|\,\psi(x)<c\big\}$ is compactly contained in $X$ for every $c\in\R$. Then the equation $d_{L}s=s'$ (in the sense of distribution theory) has a solution $s\in A^{p,q}(X)$ for every $s'\in A^{p,q+1}(X)$ such that $d_{L}s'=0\,.$ 
\end{theorem}
\begin{proof}
    Follows from Theorem \ref{pluri thm2} and the Sobolev embedding theorem (cf. \cite[Theorem 6, p.~284]{evans}), that is, $W^{m,2}(X,E^{p,q},loc)\subset C^{m-n-1}(X,E^{p,q})$ for all $m>n$ where $C^{m-n-1}(X,E^{p,q})$ denotes the space of continuous sections of $E^{p,q}$ of differential order $m-n-1$.
\end{proof}
\begin{prop}\label{apprx prop1}
    Under the assumption stated in Theorem \ref{apq thm}, every GH function in a neighborhood of a compact set $K_0\subset X$ can be approximated by GH functions in $\mathcal{O}(X)$ with respect to $L^2$-norm over $K_0$ where $K_0=\big\{x\in X\,|\,\psi(x)\leq 0\big\}$.
\end{prop}
\begin{proof}
Let $f$ be a GH function defined in a neighborhood $U$ of $K_0$. For $j\in\N$, let $\epsilon_{j}>0$ be small enough such that $\psi^{-1}(-\infty,\epsilon_{j}]\subset U$, and $\epsilon_{j}\to 0^{+}$ as $j\to\infty$. Consider a real smooth bump function $\eta_{\epsilon_{j}}$ on $X$  such that
\begin{itemize}
\setlength\itemsep{0.3em}
    \item $\eta_{\epsilon_{j}}\equiv 1$ on $\psi^{-1}(-\infty,0]$,\quad and\quad $\eta_{\epsilon_{j}}\equiv 0$ on $\psi^{-1}(\epsilon_{j},\infty)$.
    \item $0\leq\eta_{\epsilon_{j}}(x)\leq 1$ for all $0\leq\psi(x)\leq\epsilon_{j}$.
\end{itemize}
Observe that there exists a smooth positive real function $\eta\in C^{\infty}(X,\C)$ such that $|d_{L}\eta_{\epsilon_{j}}|\leq \eta$ for every $j\in\N$, and moreover, $\big\{d_{L}\eta_{\epsilon_{j}}\big\}$ converges pointwise to $0$ as $j\to\infty$. Since $\eta\in L^2(X, E^{0,0}, loc)$, we can apply Lemma \ref{l2 lemma} and Lemma \ref{l2 lemma1} to choose a smooth convex increasing function $\chi$ with $\chi(0)=0$, such that 
$$\int_{X}|\eta|^2\,e^{-\chi\circ\psi}\,d\lambda<\infty\,.$$
Therefore, using the dominated convergence theorem, we have $$|d_{L}\eta_{\epsilon_{j}}|_{\chi\circ\psi}\to 0\quad\text{as $j\to\infty$}\,.$$
Applying Theorem \ref{pluri thm1} to $d_{L}(\eta_{\epsilon_{j}}f)$ for $p=0,\,q=0$, there exists $u_{j}\in L^2(X, E^{0,0}, loc)$ such that 
$$d_{L}u_{j}=d_{L}(\eta_{\epsilon_{j}}f)\,,\quad\text{and}\quad\,|u_{j}|_{\chi\circ\psi}\leq |d_{L}(\eta_{\epsilon_{j}}f)|_{\chi\circ\psi}\,.$$
Set $M_{j}:=\left(\sup_{\big\{\psi(x)\leq\epsilon_{j}\big\}}|f(x)|^2\right)$. Now $d_{L}(\eta_{\epsilon_{j}}f)=fd_{L}\eta_{\epsilon_{j}}$ as $f$ is a GH function (cf. Proposition \ref{imp corr}), and $\big\{M_{j}\big\}$ is a bounded sequence. Consequently, since $\supp(d_{L}\eta_{\epsilon_{j}})\subset\psi^{-1}[0,\epsilon_{j}]$, it follows that
\begin{align*}
    |d_{L}(\eta_{\epsilon_{j}}f)|^2_{\chi\circ\psi}&\leq\, M_{j}\int_{\psi^{-1}[0,\epsilon_{j}]}|d_{L}\eta_{\epsilon_{j}}|^2\,e^{-\chi\circ\psi}\,d\lambda\leq\, M_{j}\,|d_{L}\eta_{\epsilon_{j}}|^2_{\chi\circ\psi}\,.
\end{align*}
This implies $|d_{L}(\eta_{\epsilon_{j}}f)|_{\chi\circ\psi}\to 0$ as $j\to\infty$. So, we get 
\begin{align*}
\int_{K_0}\,|u_{j}|^2\,d\lambda &\leq\int\,|u_{j}|^2\,e^{-\chi\circ\psi}\,d\lambda\,,\,\,\,(\text{as $\chi$ is increasing and $\chi\circ\psi\leq 0$ on $K_0$})\\ 
&\leq |d_{L}(\eta_{\epsilon_{j}}f)|^2_{\chi\circ\psi}    
\end{align*}
Thus $\|u_{j}\|_{L^2(K_0)}\to 0$ as $j\to\infty$.
Define, for each $j$, $$\tilde{f}_{j}:=\eta_{\epsilon_{j}}f-u_{j}\quad\text{on $X$}\,.$$
Note that the functions $\tilde{f}_{j}$  $(j=1,2,\ldots)$
are GH functions by Proposition \ref{imp corr}, as $d_{L}\tilde{f}_{j}=0$. Then $\|f-\tilde{f}_{j}\|_{L^2(K_0)}\leq \|u_{j}\|_{L^2(K_0)}$, since $|f-\eta_{\epsilon_{j}}f|=0$ on $K_0$. Hence the sequence $\big\{\tilde{f}_{j}\big\}_{j\in\N}$ of GH functions on $X$ converges to $f$ with respect to $L^2$-norm over $K_0$.
\end{proof}

Theorem \ref{darbu thm} and Corrollary \ref{cor:diffcharts} imply that, in a local coordinate system (cf. \eqref{loc coordi}), a GH function coincides with a holomorphic function on a domain of a complex Euclidean space. Therefore, by applying a straightforward modification of \cite[Theorem 1.2.4 and Theorem 2.2.3]{hor} to those local coordinate systems and utilizing both the compactness of $K_0$ and Proposition \ref{apprx prop1},  we can establish the following.

\begin{prop}\label{apprx prop}
    Under the assumption stated in Proposition \ref{apprx prop1}, every GH function in a neighborhood of a compact set $K_0\subset X$ can be uniformly approximated by GH functions in $\mathcal{O}(X)$ over $K_0$ where $K_0=\big\{x\in X\,|\,\psi(x)\leq 0\big\}$.
\end{prop}

By Definition \ref{l-pluri}, any strictly $L$-plurisubharmonic function $\psi$ is constant along the leaves of the induced foliation. In other words, on local coordinated (cf. \eqref{loc coordi}), $\psi$ is equivalent to a smooth real function on some $U_2$ where $U_2$ defined as in Theorem \ref{darbu thm}. Thus, by a straightforward modification in \cite[Lemma 5.2.11]{hor}, we obtain
\begin{lemma}\label{lem}
    For every $x_0\in X$, there exists an open neighborhood $U_{x_0}\subset X$ of $x_0$ and a GH function $g_{x_0}$ on $U_{x_0}$ such that
    $$g_{x_0}(x_0)=0\,,\quad\text{and}\quad\re(g_{x_0}(x))<\psi(x)-\psi(x_0)\quad\text{if\,\,$x_0\neq x\in U_{x_0}$}\,.$$
\end{lemma}
\begin{theorem}($L^{2}$-characterization for GC Stein manifolds)\label{main2}
    Let $X^{2n}$ be a regular GC manifold of type $k>0$. Then $X$ is a GC Stein manifold if and only if the following holds:
    \begin{enumerate}
    \setlength\itemsep{0.4em}
        \item $X$ satisfies the property \ref{p}.
        \item There exists a strictly $L$-plurisubharmonic function $\psi$ (cf. Definition \ref{l-pluri}) on $X$ such that $$X_{c}:=\big\{x\in X\,\big|\,\psi(x)<c\big\}\subset\subset X\quad\text{for every\,\,\,$c\in\R$}\,.$$
        The sets $\overline{X_{c}}$ are then $\mathcal{O}(X)$-convex. 
    \end{enumerate}
\end{theorem}
\begin{proof}
    One way follows directly from Remark \ref{poi rmk}, Definition \ref{main def}, and Theorem \ref{pluri thm}. 
    
    \vspace{0.5em}
    For the converse, assume that $X$ satisfies the property \ref{p} and that such a function $\psi$ exists on $X$. To establish that $X$ is a GC Stein manifold, we need to show that $X$ satisfies the three conditions in Definition \ref{main def}, namely, GH separability, GH convexity, and GH regularity. We begin by verifying that $X$ satisfies the GH separability condition. 

    \vspace{0.5em}
    Let $x_0\in X$ be an arbitrary point. By Lemma \ref{lem}, we choose $U_{x_0}$ and $g_{x_0}$ such that $x\notin U_{x_0}$ for any $x\in X$ with $\psi(x)\leq\psi(x_0)$. Additionally, $U_{x_0}$ can be covered by a single set of local coordinates (see \eqref{loc coordi}). Consider two neighborhoods $U_1$ and $U_2$ of $x_0$ such that $$U_1\subset\subset U_2\subset\subset U_{x_0}\,.$$
    Let $\eta\in C^{\infty}_0(U_2,\C)$ be a real smooth bump function with $\eta\equiv 1$ on $U_1$. As $\supp(d_{L}\eta)\subset \overline{U_2}\backslash U_1$, we choose $a>\psi(x_0)$ with $\big\{x\in X\,\big|\,\psi(x)<a\big\}\subset U_{x_0}$, and $\epsilon>0$ so that
    \begin{equation}\label{eqn1}
    \re(g_{x_0}(x))<-\epsilon\quad\text{for any\,\, $x\in\supp(d_{L}\eta)$\,\,and\,\,$\psi(x)<a$}\,. 
    \end{equation}
    Observe that, since $X_{c}$ $(\text{for}\,\,c\in\R)$ is an open set and hence a GC manifold of the same type, the differential operator $d_{L}$ also exists on $X_c$. Furthermore, the hypotheses of the theorem apply to $X_{c}$ as well, with $\psi$ replaced by $\frac{1}{c-\psi}$. So, using Theorem \ref{pluri thm1} for the $d_{L}$ operator on $X_a$, there is a real function $\phi_a\in C^{\infty}(X_a,\C)$, bounded from below in $X_a$, such that the equation $d_{L}s=s'$ (in the sense of distribution theory) for every $s'\in L^2(X_a, E^{0,1},\phi_a)$
    with $d_{L}s'=0$ has a solution $s$ with $|s|_{\phi_a}\leq |s'|_{\phi_a}$. Consequently, Theorem \ref{apq thm} implies that $s$ is smooth if and only if $s'$ is smooth. 

    \vspace{0.5em}
     Now, let $g$ be a GH function on $U_{x_0}$ and set $s'_t:=g\,e^{tg_{x_0}}\,d_{L}\eta$ on $X_a$ for $t\in (0,\infty)$. Note that, for a fixed $g$, $|s'_t|_{\phi_a}=O(e^{-t\epsilon})$ by \eqref{eqn1}. Consequently, by preceding discussion, there exists a solution $s_t\in L^2(X_a, E^{0,0},\phi_a)$ with 
     \begin{equation}\label{eqn2}
       |s_t|_{\phi_a}=O(e^{-t\epsilon})\,.  
     \end{equation}
   Consider the following GH function  
   \begin{equation}\label{eqn4}
       f_{t}:=\eta\,g\,e^{tg_{x_0}}-s_t\quad\text{on $X_a$}\,.
   \end{equation}
     Fix any $x\in X$ such that $\psi(x)\leq\psi(x_0)$. Then, we observe that $f_t(x)=-s_t(x)$, which converges to $0$ (up to subsequence) as $t\to\infty$. Additionally, $f_t(x_0)=g(x_0)-s_t(x_0)$, and it converges to $g(x_0)$ (up to subsequence) when $t\to\infty$. Since we can choose $g(x_0)$ to be $1$, it follows that for sufficiently large $t$, we have $$f_{t}(x)\neq f_{t}(x_0)\,.$$
     Applying Proposition \ref{apprx prop}, there exists a GH function $F_{x}\in \mathcal{O}(X)$ that approximates $f_t$ on $\overline{X_{\psi(x_0)}}$ closely enough so that $$F_{x}(x)\neq F_{x}(x_0)\,.$$
Therefore, $X$ satisfies the GH separability condition.

\vspace{0.5em}
Next, we will show that $X$ satisfies the GH convexity condition. Notice that, by \eqref{eqn2} and the definition of $f_t$, for every $c<\psi(x_0)$, we have
\begin{equation}\label{eqn3}
    \int_{X_c}|f_t|^2\,d\lambda\to 0\quad\text{when $t\to\infty$}\,.
\end{equation}
By applying a straightforward modification of \cite[Theorem 2.2.3]{hor} to local coordinate systems (cf. \eqref{loc coordi}), and utilizing both \eqref{eqn3} and the compactness, we get that $\big\{f_t|_{K}\big\}$ converges to $0$, uniformly on any compact set $K\subset X_c$. Now, if $c'<c$, it follows that $|f_t|<\frac{1}{2}$ for sufficiently large $t$, while $f_t(x_0)\to 1$. Using Proposition \ref{apprx prop}, we can approximate $f_t$ by GH functions in $\mathcal{O}(X)$ which show that $x_0\notin\widehat{\left(\overline{X_{c'}}\right)}_{\mathcal{O}(X)}$ (cf. \eqref{cnvx}) for any $c'<\psi(x_0)$. Hence $\widehat{\left(\overline{X_{c'}}\right)}_{\mathcal{O}(X)}=\overline{X_{c'}}$ for every $c'$. This proves the GH convexity condition of $X$.

\vspace{0.5em}
Finally, we will show that $X$ satisfies the GH regularity condition. Consider a local coordinates $\{\tilde{f}_1,\ldots,\tilde{f}_{n-k},\,g_1,\ldots,g_{k}\}$ (see Theorem \ref{darbu thm} and \eqref{loc coordi}) at $x_0$, formed by GH maps that vanish at $x_0$. Let $f_{t}^{j}$ (see \eqref{eqn4}) on $X_a$ denote the GH function corresponding to $g_j$ for $j=1,\ldots,k$. By Remark \ref{rmk pluri}, note that at $x_0$, $$d_{\overline{L}}f_{t}^{j}=d_{\overline{L}}g_j-d_{\overline{L}}s_t^{j}\,,\quad\text{and}\quad d_{\overline{L}}s_t^{j}\to 0\quad\text{when $t\to\infty$}\,.$$ 
Then the Jacobian of $\big\{f_{t}^{1},\ldots,f_{t}^{k}\big\}$ at $x_0$ converge to $I_{k\times k}$ as $t\to\infty$ where $I_{k\times k}\in M_{k}(\C)$ denotes the identity matrix of order $k$. Applying Proposition \ref{apprx prop}, we can obtain GH functions $F^{j}\in\mathcal{O}(X)$ $(j=1,\ldots,k)$, that approximate $f_{t}^{j}$ sufficiently closely, ensuring that the Jacobian of $\big\{F^1,\ldots,F^{k}\big\}$ is of complex rank $k$ at $x_0$. Furthermore, by the property \ref{p}, we can construct GH maps $\tilde{F}_1,\ldots,\tilde{F}_{n-k}$ on $X$, corresponding to $\tilde{f}_1,\ldots,\tilde{f}_{n-k}$, such that each $\tilde{F}_j$ extends $\tilde{f}_j$ $(j=1,\ldots,n-k)$ in a neighborhood of $x_0$. Therefore, the collection $\big\{\tilde{F}_1,\ldots,\tilde{F}_{n-k},F^{1},\ldots,F^{k}\big\}$ of GH maps on $X$ proves the GH regularity condition at $x_0$.
Hence, $X$ satisfies the GH regularity condition. This completes the proof.
\end{proof}
\begin{theorem}(Oka-Weil Theorem of GC Stein manifolds)
Given a GC Stein manifold $X$ and a compact set $K\subset X$ such that
$K=\widehat{K}_{\mathcal{O}(X)}$ (cf. \eqref{cnvx}), any GH function in a neighborhood of $K$ can be uniformly approximated on $K$ by GH functions in $\mathcal{O}(X)$.
\end{theorem}
\begin{proof}
    Follows from Proposition \ref{apprx prop} and Theorem \ref{main2}.
\end{proof}

\section{GC Stein manifold and its embedding}\label{gh embd}

In this section, we describe the GC Stein manifold by showing that it admits a GH embedding into Euclidean space, following the approach of \cite[Chapters VII-VIII]{rossi65} and \cite[Chapter 5]{hor}. This generalizes the holomorphic embedding theorem known for classical Stein manifolds.

\vspace{0.5em}
Let $X^{2n}$ be a regular GC manifold of type $k>0$. For $M,N\in\N$, consider the collection of smooth functions $\big\{f_{l}\big\}^{M}_{l=1}$ and $\big\{g_{j}\big\}^{N}_{j=1}$ on $X$ where $f_{l}:X\longrightarrow(\R^2,\omega_0)$ are GH maps and $g_{j}:X\longrightarrow\C$ are GH functions. Then the following smooth function 
\begin{equation}\label{eqnn} 
\begin{aligned}
    &F:X\longrightarrow\R^{2M}\times\C^N\,;\\
     &F(x):=(f_1(x),\ldots,f_M(x)\,,\,g_1(x),\ldots,g_N(x))\quad\text{for any $x\in X$}\,,
\end{aligned}
\end{equation}
defines a GH map for $M\leq n-k$ (see Definition \ref{GH map}).
\begin{definition}\label{regu def}
   The map $F$ (cf. \eqref{eqnn}) is called a regular GH map if it has rank $2n$, that is, for any point in $X$, there exists a local coordinate system (cf. Theorem \ref{darbu thm} and \eqref{loc coordi}) formed by the GH maps $f_1,\ldots,f_M$, and by $k$ of the GH functions $g_1,\ldots,g_N$. 
\end{definition}
\begin{remark}
In Definition \ref{regu def}, for $F$ to be considered regular, $M$ must be $n-k$. Since the GH map $f_l$ $(l=1,\ldots, M)$ is a submersion (cf. Definition \ref{GH map} and Remark \ref{poi rmk}), the regularity of $F$ comes down to finding $k$ of the GH functions $g_1,\ldots,g_N$ so that at any point, they form the $\mathbb{C}$-direction of the local coordinate system, as in \eqref{loc coordi}.
\end{remark}
\begin{prop}\label{reg prop1}
    Let $K\subset X$ be a compact subset. Assume that, $X$ is GH regular and satisfies the GH separability condition in Definition \ref{main def}. Then, for some large $N>0$, there exists a GH map $F:X\longrightarrow\R^{2n-2k}\times\C^N$,  such that $F$ is injective and regular on $K$. Moreover, The Lebesgue measure of $F(K)$ is $0$ in $\R^{2n-2k}\times\C^N$ for $N>k$.  
\end{prop}
\begin{proof}
  Follows from Proposition \ref{prop} and Sard's Theorem.
\end{proof}
\begin{prop}\label{reg prop}
    Let $F: X\longrightarrow\R^{2n-2k}\times\C^{N+1}$ $(N>0)$ be a regular GH map, as in \eqref{eqnn}, on the compact subset $K\subset X$. Then the following holds:
    \begin{enumerate}
    \setlength\itemsep{0.4em}
        \item For $N\geq 2k$, there exists $c=(c_1,\ldots,c_N)\in\C^N$ arbitrarily close to the origin so that the GH map $$(f_1,\ldots,f_{n-k}\,,\,g_1-c_1 g_{N+1},\ldots,g_N-c_N g_{N+1}):X\longrightarrow\R^{2n-2k}\times\C^{N}$$ is a regular GH map map on $K$.
        \item If $F$ is also injective, then $c\in\C^N$ can be chosen such that the GH map $$(f_1,\ldots,f_{n-k}\,,\,g_1-c_1 g_{N+1},\ldots,g_N-c_N g_{N+1}):X\longrightarrow\R^{2n-2k}\times\C^{N}$$ is a one-to-one regular GH map map on $K$ for $N\geq 2k+1$.
    \end{enumerate}
Moreover, $(1)$ and $(2)$ hold for all $c\in\C^N$ outside a measure zero set.
\end{prop}
\begin{proof}
    Without loss of generality, let us assume that $K$ is contained in one local coordinate system $U$ with coordinates $(p_1,\ldots,p_{2n-2k},z_1,\ldots,z_k)$ (cf. \eqref{loc coordi}). Let $(p_{2n-2k+1},\ldots,p_{2n})\in\R^{2k}$ represent a coordinate system such that $z_m=p_{2n-2k+(2m-1)}+ip_{2n-2k+2m}$ for $m=1,\ldots,k$. Additionally, for $l\in\{1,\ldots,n-k\}$, set $\frac{\partial f_l}{\partial p_m}:=\frac{\partial f_{1l}}{\partial p_m}\oplus\frac{\partial f_{2l}}{\partial p_m}$ for $1\leq m\leq 2n$ where $f_l:=(f_{1l},f_{2l})$. 

    \vspace{0.5em}
    \hspace{-1.2em}\textbf{(1)} 
   \,\,Since the GH maps $\big\{f_1,\ldots,f_{n-k}\big\}$ are submersions, it suffices to choose a $c\in\C^N$ such that if 
    \begin{equation}\label{eqnn1}
\sum^{k}_{m=1}\,a_l\left(\frac{\partial g_j}{\partial z_m}-c_j\frac{\partial g_{N+1}}{\partial z_m}\right)=0\,,\quad j=1,\ldots,N\,,
    \end{equation}
    at some point in $K$ and for some $a=(a_1,\ldots,a_k)\in\C^k$, then $a=0$. Set $c_{N+1}=1$. Consequently, the equations \eqref{eqnn1} can be rewritten as
    $$\sum^{k}_{m=1}\,a_m\,\frac{\partial g_j}{\partial z_m}=c_j\left(\sum^{k}_{m=1}\,a_m\,\frac{\partial g_{N+1}}{\partial z_m}\right)\,,\quad j=1,\ldots,N+1\,.$$
    
 Now consider the smooth map $\gamma:\C^k\times U\longrightarrow\R^{2n-2k}\times\C^{N+1}$ defined as 
$$\gamma(a,x)=\left(\sum^{2n}_{m=1}\frac{\partial f_1}{\partial p_m}(x),\ldots,\sum^{2n}_{m=1}\frac{\partial f_{n-k}}{\partial p_l}(x)\,,\,\sum^{k}_{m=1}\,a_m\,\frac{\partial g_1}{\partial z_m}(x),\ldots,\sum^{k}_{m=1}\,a_m\,\frac{\partial g_{N+1}}{\partial z_m}(x)\right)\,.$$
    Since $F$ is a regular GH map on $K$, the matrix $\left(\frac{\partial f_l}{\partial p_m}\right)_{(2n-2k)\times 2n}\bigoplus\left(\frac{\partial g_{j}}{\partial z_m}\right)_{(N+1)\times k}$ has real rank $2n$ at any point in $K$. Thus, it is enough to choose $c\in\C^N$ such that $(c,1)\notin\pr_{\C^{N+1}}(\gamma(\C^k\times K))$ where $\pr_{\C^{N+1}}$ is the projection of $\R^{2n-2k}\times\C^{N+1}$ onto $\C^{N+1}$. First, restrict $a$ to a ball $B(r):=\big\{|a|\leq r\big\}$, $r\in\N$. By Sard's theorem, it then follows that the image of $\gamma$ on $B(r)\times K$ is a set of measure zero, since $N+1>2k$. Then $\pr_{\C^{N+1}}(\gamma(B(r) \times K))$ also has Lebesgue measure zero. Since its intersections with the hyperplanes $\big\{c_{N+1} = \text{constant}\big\}\subset\C^{N+1}\,,$ are isomorphic up to a homothetic transformation, each of these intersections must also be of $2N$-dimensional measure zero. This proves $(1)$.

    \vspace{0.5em}
    \hspace{-1.2em}\textbf{(2)}
    \,\,Let $x,y\in K$ and set $z'=g_{N+1}(x)-g_{N+1}(y)$. By $(1)$, It is sufficient to show that $c_1,\ldots,c_{N+1}$ can be chosen with  $c_{N+1}=1$ so that the following equations, on $K$, 
    \begin{equation*}
    \begin{aligned}
    &f_l(x)-f_l(y)=0\quad l=1,\ldots, n-k\,,\\
        &g_j(x)-g_j(y)=z'\,c_j\quad j=1,\ldots, N+1\,,
    \end{aligned}  
    \end{equation*}
    imply $z'\,=0$. Consequently, since $F$ is injective, this leads to $x=y$. Therefore, consider the smooth map $\hat\gamma:\C\times U\times U\longrightarrow\R^{2n-2k}\times\C^{N+1}$ defined as $$\hat\gamma(a,x,y)=\big(f_1(x)-f_1(y),\ldots,f_{n-k}(x)-f_{n-k}(y)\,,\,a(g_1(x)-g_1(y)),\ldots,a(g_k(x)-g_k(y))\big)\,.$$
    Now $\hat\gamma(\C\times K\times K)$ is a set of measure zero since $N+1>2k+1$. Thus, by preceding discussion in $(1)$, we have proved $(2)$. 
\end{proof}
Now, let $X^{2n}$ be a GC Stein manifold of type $k>0$. By Theorem \ref{main2},  there exists a sequence of compact $\mathcal{O}(X)$-convex subsets 
\begin{equation}\label{hemi}
  K_1\subset K_2\subset\cdots\subset K_j\subset\cdots\,,\quad\text{such that}\quad \bigcup^{\infty}_{j=1}K_{j}=X\quad\text{and}\quad K_{j}\subset K_{j+1}^{o}\,,  
\end{equation}
for every $j\geq 1$. Here $K_{j}^{o}$ is the interior of $K_{j}$. Additionally, for any compact subset $K\subset X$, there exists $j\in\N$ such that $K\subset K_j$. For $N>0$, consider the following topological space
    \begin{equation}\label{eqnn2}
       \mathcal{GH}^N(X):=\big\{F:X\longrightarrow\R^{2n-2k}\times\C^N\,\big |\,\text{$F$ is a GH map}\,\big\}\,, 
    \end{equation}
endowed with Whitney topology (cf. \cite[Chapter 2]{hirsch}), and the following subspaces within it: 
    \begin{equation}\label{eqnn3}
        \begin{aligned}
            &\mathcal{GH}_{r}^N(X):=\big\{F:X\longrightarrow\R^{2n-2k}\times\C^N\,\big |\,\text{$F$ is a regular GH map}\,\big\}\,,\\
            &\mathcal{GH}_{r,1-1}^N(X):=\big\{F:X\longrightarrow\R^{2n-2k}\times\C^N\,\big |\,\text{$F$ is a one-to-one regular GH map}\,\big\}\,.
        \end{aligned}
    \end{equation}
Define 
\begin{equation}\label{k-norm2}
    \begin{aligned}
        |G|_{K,l}:=\sup_{x\in K}\big\{|D^{l}g_{j}(x)|\,\big |\,1\leq j\leq n\big\}\,,\quad\text{$K\subset$ compact}\,,
    \end{aligned}
\end{equation}
where $G=(g_1,\ldots,g_{n})$ is a $n$-tuple of $\R^2$-valued smooth functions and $D^lg_j(x)$ denotes the $l$-th derivative of $g_j$ at $x$. 
Consider a sequence $0<\epsilon_j\leq 1$ with $\lim_{j\to\infty}\epsilon_j=0$ monotonically, and $\sum_{j\in\N}\epsilon_j<\infty$. Then the metric
\begin{equation}\label{metric}
\delta(F,G):=\sum_{j\in\N}\,\epsilon_j\,\sum^{\infty}_{l=0}\frac{|F-G|_{K_j,l}}{1+|F-G|_{K_j,l}}\,,\text{for}\quad F,G\in\mathcal{GH}^N(X).
\end{equation}
induces the Whitney topology. Note that in the case of $\mathcal{O}^N(X)$ (cf. \eqref{ox}), by Theorem \ref{darbu thm}, any GH function in $\mathcal{O}^N(X)$ is locally equivalent to a holomorphic map. This implies that the Whitney topology on $\mathcal{O}^N(X)$ coincides with the topology of uniform convergence on compact subsets induced by $K$-norms (cf. \eqref{k-norm}), which is the same as the compact-open topology. 

\vspace{0.5em}
Let $C^{\infty}_{P}(X,\mathbb{R}^2)$ denote the space of Poisson maps from $X$ to $(\mathbb{R}^2,\omega_0)$ with the Whitney topology. By modifying both equations \eqref{k-norm2} and \eqref{metric}, it can be shown that $C^{\infty}_{P}(X,\mathbb{R}^2)$ is metrizable. Note that $$\mathcal{GH}^N(X)=C^{\infty}_{P}(X,\mathbb{R}^2)^{n-k}\oplus\mathcal{O}^N(X)$$ where $C^{\infty}_{P}(X,\mathbb{R}^2)^{n-k}:=\bigoplus_{n-k}C^{\infty}_{P}(X,\mathbb{R}^2)$.
\vspace{0.3em}
\begin{lemma}
~
\begin{enumerate}
\setlength\itemsep{0.4em}
\item $\mathcal{GH}^N(X)$ is a complete metrizable space with respect to the Whitney topology.
    \item The space $C^{\infty}_{P}(X,\R^2)$ is both metrizable and complete with the Whitney topology.
\end{enumerate}
\end{lemma}
\begin{proof}
    Metrizability is a direct consequence of the preceding discussion, and completeness follows straightforwardly.
\end{proof}
\begin{theorem}\label{dns thm}
~
    \begin{enumerate}
    \setlength\itemsep{0.4em}
        \item The subspace $\mathcal{GH}^N(X)\backslash\mathcal{GH}_{r}^N(X)$ is a set of the first category if $N\geq 2k$. In simpler terms, $\mathcal{GH}_{r}^N(X)$ is dense in $\mathcal{GH}^N(X)$ for $N\geq 2k$.
        \item The subspace $\mathcal{GH}^N(X)\backslash\mathcal{GH}_{r,1-1}^N(X)$ is a set of the first category if $N\geq 2k+1$. In other words, $\mathcal{GH}_{r,1-1}^N(X)$ is a dense subspace in $\mathcal{GH}^N(X)$ for $N\geq 2k+1$.
    \end{enumerate}
\end{theorem}
\begin{proof}
By \eqref{hemi}, it is enough to prove the theorem for every compact set $K\subset X$, that is, for the set $\big(\mathcal{GH}^N(X) \backslash \mathcal{GH}_{r}^N(X)\big)_K$ (respectively, $\big(\mathcal{GH}^N(X) \backslash \mathcal{GH}_{r,1-1}^N(X)\big)_K$) of all $F\in\mathcal{GH}^N(X)$ which are not regular (respectively, regular and injective) on $K$. 

\vspace{0.5em}
\hspace{-1.2em} \textbf{(1)}
\,\,The subspace $\big(\mathcal{GH}^N(X) \backslash \mathcal{GH}_{r}^N(X)\big)_K$ is closed. This can be seen as follows:

\vspace{0.3em}
Let $\big\{F_{j}\big\}_{j\in\N}$ be a sequence in $\big(\mathcal{GH}^N(X) \backslash \mathcal{GH}_{r}^N(X)\big)_K$ such that $F_{j}\to F$ and at $x_j\in K$, the GH map $F_j$ is not regular. Consequently, the GH map $F$ is not GH regular at any limit point of the sequence $\{x_j\}$.

\vspace{0.5em}
Let $(f_1,\ldots,f_{n-k}\,,\,g_1,\ldots,g_N,)\in\big(\mathcal{GH}^N(X) \backslash \mathcal{GH}_{r}^N(X)\big)_K$. By Proposition \ref{reg prop1}, there exists a regular GH map $(f'_1,\ldots,f'_{n-k}\,,\,g'_1,\ldots,g'_{N'})\in\mathcal{GH}_{r}^{N'}(X)$ on $K$ for some large $N'>0$. Without loss of generality, we take $f_l=f'_l$ for all $l\in\{1,\ldots,n-k\}$. Then the GH map $$(f_1,\ldots,f_{n-k}\,,\,g_1,\ldots,g_N\,,\,g'_1,\ldots,g'_{N'})\in\mathcal{GH}^{N+N'}(X)$$ is regular on $K$. Now, apply $(1)$ in Proposition \ref{reg prop} repeatedly to $\hat{G}$, and conclude that the GH map 
$$(f_1,\ldots,f_{n-k}\,,\,g''_1,\ldots,g''_N):X\longrightarrow\R^{2n-2k}\times\C^N\,,\,\,\,\,\text{where}\,\,\,\, g''_j=g_j+\sum^{N'}_{m=1}\,c_{jm}g'_{m}\,\,(j=1,\ldots,N)\,,$$ is a regular map on $K$ for suitable and sufficiently small coefficients $c_{jm}\in\C$. Therefore $$(f_1,\ldots,f_{n-k}\,,\,g''_1,\ldots,g''_N)\notin\big(\mathcal{GH}^N(X) \backslash \mathcal{GH}_{r}^N(X)\big)_K\,,$$ implying that $G$ is not an interior point of $\big(\mathcal{GH}^N(X) \backslash \mathcal{GH}_{r}^N(X)\big)_K$. Thus $\mathcal{GH}^N(X)\backslash\mathcal{GH}_{r}^N(X)$ is nowhere dense. 

\vspace{0.5em}
\hspace{-1.2em}\textbf{(2)} 
\,\,Using $(2)$ in Proposition \ref{reg prop}, the proof follows in a similar manner as for $(1)$.
\end{proof}
\begin{definition}
    An open set $P\subset\subset X$ is called a \textit{GC polyhedron of order} $N\in\N$ if there exist some GH functions $g_j\in\mathcal{O}(X)$, $j=1,\ldots,N$, such that $$P=\big\{x\in X\,\big |\,|g_j(x)|< 1\,,\,\,j=1,\ldots,N\big\}\,.$$
\end{definition}
\begin{prop}\label{poly prop}
    Let $X^{2n}$ be a GC Stein manifold of type $k>0$. Let $K\subset X$ be a compact set.

    \vspace{0.2em}
    \begin{enumerate}
    \setlength\itemsep{0.4em}
        \item If $K=\widehat{K}_{\mathcal{O}(X)}$ and $U$ is a neighborhood of $K$, then there exists a GC polyhedron $P$ of some order $N$ such that $K\subset P\subset\subset U$.
        \item If $P$ is a GC polyhedron of order $N+1$ in $X$ such that $K\subset P$, then for $N\geq 2k$, there exists a GC polyhedron $\widetilde{P}$ of order $N$ such that $K\subset\widetilde{P}\subset P$.
    \end{enumerate}
\end{prop}
\begin{proof}
    \textbf{(1)}
    \,\,We can assume that $U$ is compactly contained in $X$. Now, for every $x\in\partial U$, we can find $g\in\mathcal{O}(X)$ so that $|g|< 1$ on $K$ but $|g(x)|>1$. Here $\partial U$ is the boundary of $U$. Using compactness, there exists $N\in\N$ such that for each $j\in\{1,\ldots,N\}$, the following open set 
    $$\big\{x\in X\,\big |\,|g_j(x)|< 1\,, g_j\in\mathcal{O}(X)\big\}$$ contains $K$ but does not meet $\partial U$. Then the open set
    $$P:=\bigcap^N_{j=1}\,\big\{x\in X\,\big |\,|g_j(x)|< 1\,, g_j\in\mathcal{O}(X)\big\}\,\bigcap U\,,$$ contains $K$ and is compactly contained within $U$.

    \vspace{0.5em}
\hspace{-1.2em}\textbf{(2)}
\,\,Let $P=\big\{x\in X\,\big |\,|g_j(x)|< 1\,,g_j\in\mathcal{O}(X)\,,\,\,j=1,\ldots,N+1\big\}\,.$ Choose positive real numbers $c_1<c_2<c_3<c_4<1$ such that $|g_j|<c_1$ on $K$. We can find GH functions $g'_j\in\mathcal{O}(X)$ for $j=1,\ldots,N+1$, with $g'_{N+1}=g_{N+1}$, so that: 
\begin{itemize}
 \setlength\itemsep{0.4em}
    \item The tuple $\bigg(\frac{g'_1}{g_{N+1}},\ldots,\frac{g'_N}{g_{N+1}}\bigg)$ is of rank $k$ on $\big\{x\in\overline{P}\,\big |\,|g_{N+1}(x)|\geq c_3\big\}$.
    \item Each $g'_j$ is sufficiently close to $g_j$ such that $|g_j|<c_1$ on $K$ for $j=1,\ldots,N$.
    \item The open set $V=\big\{x\in P\,\big |\,|g'_j(x)|< c_4\,,\,\,j=1,\ldots,N+1\big\}$ is compactly contained in $P$.
\end{itemize}
This is possible because, since $N\geq 2k$, using Theorem \ref{dns thm}, we can choose $\frac{g'_j}{g_{N+1}}$ as $\frac{g_j}{g_{N+1}}$ plus a linear combination of suitably small GH functions in $\mathcal{O}(X)$.

\vspace{0.5em}
Let $m\in\N$ and define the open set $V_m$ as 
$$V_m:=\big\{x\in X\,\big |\,|g'_j(x)^m-g_{N+1}(x)^m|< c_2^m\,,\,\,j=1,\ldots,N\big\}\,.$$
Let $x\in K$. Then
\begin{align*}
    |g'_j(x)^m-g_{N+1}(x)^m|&\leq |g'_j(x)^m|+|g_{N+1}(x)^m|\\
    &< 2c_1^m\quad\text{(as $|g'_j|<c_1$ and $|g_{N+1}|<c_1$ on $K$ )}\,.
\end{align*}
This implies that for sufficiently large $m$, $K\subset V_m$ because $c_1<c_2<1$. Define
\begin{equation}\label{pv}
    \widetilde{P}_m:=\bigcup_{\alpha}\,\big\{C_{\alpha}\subset V_m\,\big |\,\text{$C_{\alpha}$ is a component of $V_m$ and $C_{\alpha}\cap K\neq\emptyset$}\big\}\,.
\end{equation}
Observe that for sufficiently large $K\subset \widetilde{P}_{m}$. Hence, if we can show that $\widetilde{P}_{m}\subset V$ for large $m$, we will obtain a GC polyhedron of order $N$ that satisfies all the required properties. 

\vspace{0.5em}
If $\widetilde{P}_{m}\not\subset V$, then there must be a point $x\in\widetilde{P}_m$ that lies on the boundary $\partial V$, since every component of $\widetilde{P}_m$ intersects $K$ and consequently includes points within $V$. Now, if $|g_{N+1}(x)|<c_3$, then for any $1\leq j\leq N$, $|g'_j(x)^m|<c_3^m + c_2^m < c_4^m$ when $m$ is large enough. This contradicts the fact that $x\in\partial V$. Thus, the point $x$ belongs to the following compact set
$$\widetilde{K}:=\big\{x\in\partial V\,\big |\,|g_{N+1}(x)|\geq c_3\big\}\,.$$
Let $\widetilde{K}_0\subset\widetilde{K}$ be a compact set contained in a local coordinate system $\widetilde{U}\subset\R^{2n-2k}\times\C^k$ with coordinates $(p_1,\ldots,p_{2n-2k},z_1,\ldots,z_k)$ (cf. \eqref{loc coordi}). If $x\in\widetilde{K}_0\cap V_m$, then with $\tilde{g}_j:=\frac{g'_j}{g_{N+1}}$, we have
$$|\tilde{g}_j(x)-1|<\left(\frac{c_2}{c_3}\right)^m\,,\quad j=1,\ldots,N\,.$$
Let $\xi\in\R^{2n-2k}\times\C^k$ with $|\xi|=\frac{1}{m^2}$, and consider the following 
\begin{align*}
&|g'_j(x+\xi)^m-g_{N+1}(x+\xi)^m|=|\tilde{g}_j(x+\xi)^m-1|\,\,|g_{N+1}(x+\xi)^m|\,,\\
&\tilde{g}_j(x+\xi)^m-1=\tilde{g}_j(x)^m\left(\left(\frac{\tilde{g}_j(x+\xi)}{\tilde{g}_j(x)}\right)^m-1\right)-1\,.
\end{align*}
Note that $|g_{N+1}(x+\xi)^m|\geq c_3^m\bigg(1+O\bigg(\frac{1}{m}\bigg)\bigg)>\frac{c_3^m}{2}$ for sufficiently large $m$. This follows from the fact that $|g_{N+1}(x+\xi)^m|\geq c_3\bigg(1+O\bigg(\frac{1}{m^2}\bigg)\bigg)$ when $m$ is large. Since the $\R$-directions of $\tilde{g}_j$ are constant (see Proposition \ref{prop GH}), applying the Taylor expansion yields
$$\left(\frac{\tilde{g}_j(x+\xi)}{\tilde{g}_j(x)}\right)=1+T_j(\xi)+O\bigg(\frac{1}{m^4}\bigg)\,,\quad\text{and}\quad\left(\frac{\tilde{g}_j(x+\xi)}{\tilde{g}_j(x)}\right)^m=1+m\,T_j(\xi)+O\bigg(\frac{1}{m^2}\bigg)\,.$$  Here, the linear forms $T_j(\xi)$ do not vanish simultaneously because the collection $\big\{\tilde{g}_1,\ldots,\tilde{g}_N\big\}$ has rank $k$ on $\widetilde{K}_0$. Consequently, $\max_{1\leq j\leq N}\,|T_j(\xi)|\geq l\,|\xi|$ for some $l>0$. Thus, by summing the preceding estimates, for $|\xi|=\frac{1}{m^2}$ and sufficiently large $m$, we get
\begin{equation}\label{eqnn4}
\max_{1\leq j\leq N}\,|g'_j(x+\xi)^m-g_{N+1}(x+\xi)^m|\,>\,\frac{c_3^m}{4}\left(\frac{l}{m}+O\bigg(\frac{1}{m^2}\bigg)\right)\,>\,c_2^m\,.    
\end{equation}
Since the estimates are uniform in $x\in\widetilde{K}_0$, the equation \eqref{eqnn4} implies that no point $x \in\widetilde{K}_0$ can belong to a component of $V_m$ that intersects $K$. Hence, for sufficiently large $m\in\N$, we have $\widetilde{P}_{m}\subset\subset V$, which proves $(2)$.
\end{proof}
\begin{theorem}\label{main3}
  Let $X^{2n}$ be a GC Stein manifold of type $k>0$. Then there exists a GH map $\widetilde{F}:X\longrightarrow\R^{2n-2k}\times\C^{2k+1}$, which is one-to-one, proper, and regular, that is, $\widetilde{F}\in\mathcal{GH}_{r,1-1}^{2k+1}(X)$ where $\mathcal{GH}_{r,1-1}^{2k+1}(X)$ as defined in \eqref{eqnn3}. 
\end{theorem}
\begin{proof}
    By Theorem \ref{dns thm}, there exists a regular one-to-one GH map
    $$F=(f_1,\ldots,f_{n-k},g_1,\ldots,g_{2k+1}):X\longrightarrow\R^{2n-2k}\times\C^{2k+1}\,.$$
 Set $$|F(x)|:=\max_{\substack{1\leq l\leq (2n-2k)\\1\leq j\leq (2k+1)}}\,\big\{|f_l(x)|\,,|g_j(x)|\big\}\quad\text{and}\quad |g(x)|:=\max_{1\leq j\leq (2k+1)}\,\big\{|g_j(x)|\big\}\,,$$ where $g$ denotes the GH map, given by $g=(g_1,\ldots,g_{2k+1}):X\longrightarrow\C^{2k+1}\,.$  Now, if we can construct GH functions $g'_1,\ldots,g'_{2k+1}\in\mathcal{O}(X)$ such that
 \begin{equation}\label{equ1}
    \big\{x\in X\,\big |\,|g'(x)|\leq m+|g(x)|\big\}\subset\subset X 
 \end{equation}
 for every $m\in\N$ where $g'=(g'_1,\ldots,g'_{2k+1}):X\longrightarrow\C^{2k+1}$ is the GH map, then we are done. By repeatedly applying Proposition \ref{reg prop} to the regular injective GH map $$(f_1,\ldots,f_{n-k},g_1,\ldots,g_{2k+1},g'_1,\ldots,g'_{2k+1}):X\longrightarrow\R^{2n-2k}\times\C^{4k+2}\,,$$ we can find arbitrarily small constants $c_{jl}\in\C$ such that the GH map 
 $$\widetilde{F}=(f_1,\ldots,f_{n-k},g''_1,\ldots,g''_{2k+1}):X\longrightarrow\R^{2n-2k}\times\C^{2k+1}\,,$$ defines a one-to-one regular map where $$g''_j=g'_j+\sum^{2k+1}_{l=1}\,c_{jl}\,g_l\,,\quad j=1,\ldots,2k+1\,.$$ If $\sum^{2k+1}_{l=1}\,|c_{jl}|\leq 1$, then
 $$y\in\big\{x\in X\,\big |\,|\widetilde{F}(x)|\leq m\big\}\,\implies\,|(g''_1,\ldots,g''_{2k+1})(y)|\leq m\,,$$
 which leads to $y\in\big\{x\in X\,\big |\,|g'(x)|\leq m+|g(x)|\big\}$, since $|g'(y)|\leq |g(y)|+|(g''_1,\ldots,g''_{2k+1})(y)|$. Consequently, $$\{x\in X\,\big |\,|\widetilde{F}(x)|\leq m\big\}\subset\big\{x\in X\,\big |\,|g'(x)|\leq m+|g(x)|\big\}\,,$$ implying that $\widetilde{F}$ is also a proper map.

 \vspace{0.5em}
 Since $X$ is a GC Stein manifold,  there exists a sequence of compact $\mathcal{O}(X)$-convex subsets $\big\{K_j\big\}_{j\in\N}$ such that $X=\bigcup^{\infty}_{j=1}K_{j}$ and $K_{j}\subset K_{j+1}^{o}\subset K_{j+1}$  
for every $j\geq 1$ (see Theorem \ref{main2}). Using Proposition \ref{poly prop}, there exist a sequence of GC polyhedrons $\big\{P_j\big\}_{j\in\N}$ of order $2k$ such that $K_j\subset P_j\subset K_{j+1}$. Define $$M_j:=\sup_{x\in P_j}\,\big\{|g(x)|\big\}\,.$$ Then it suffices to show that 
\begin{equation}\label{equ2}
    |g'|\geq m+M_{j+1}\quad\text{on $P_{m+1}\backslash P_m$}
\end{equation} for every $m\in\N$. Because \eqref{equ2} would imply $|g'|\geq m+|g|$ on $P_{m+1}\backslash P_m$, which in turn implies $|g'|\geq m+|g|$ on $\bigcup_{j\geq m}\,(P_{j+1}\backslash P_j)=X\backslash P_m$, and thus satisfies $\eqref{equ1}$.

\vspace{0.5em}
We now construct the GH functions $\big\{g'_1,\ldots,g'_{2k+1}\big\}\subset\mathcal{O}(X)$ that satisfy the condition in \eqref{equ2}. To do so, consider that by the definition of a GC polyhedron of order $2k$, for a fixed $m\in\N$, we can find GH functions $\tilde{g}^m_1,\ldots,\tilde{g}^m_{2k}\in\mathcal{O}(X)$ such that 
\begin{itemize}
\setlength\itemsep{0.4em}
    \item $\max_{1\leq j\leq 2k}\,|\tilde{g}^m_j|<\,1$ on $\overline{P_{m-1}}$.
    \item $\max_{1\leq j\leq 2k}\,|\tilde{g}^m_j|=\,1$ on $\partial P_m$.
\end{itemize}
Define $\hat{g}^m_j:=(a_m\,\tilde{g}^m_j)^{l_m}$, where $a_m$ is slightly greater than $1$ and $l_m$ is a sufficiently large integer. By choosing $a_m$ and $l_m$ suitably, for each $m$, we have
\begin{equation}\label{equ3}
    \begin{aligned}
        &\max_{1\leq j\leq 2k}\,|\hat{g}^m_j|\leq\frac{1}{2^m}\quad\text{on $P_{m-1}$}\,.\\
        &\max_{1\leq j\leq 2k}\,|\hat{g}^m_j|>\,M_{m+1}+m+1+\max_{1\leq j\leq 2k}\,\big |\sum^{m-1}_{l=1}\,\hat{g}^l_j\big |\quad\text{on $\partial P_m$}\,.
    \end{aligned}
\end{equation}
Define 
\begin{equation}\label{equ4}
    g'_j:=\sum^{\infty}_{m=1}\hat{g}^m_j\,\quad j=1,\ldots,2k\,.
\end{equation}
Given that $P_m\subset P_{m+1}\subset K_{m+2}$, and applying \eqref{equ3}, the series \eqref{equ4} converges. Also, since $\mathcal{O}(X)$ is a complete metric space (see Theorem \ref{f spc}), each $g'_j\in\mathcal{O}(X)$ for $j=1,\ldots,2k$. Note that the construction of the GH functions $g'_1,\ldots,g'_{2k}$ provides 
\begin{equation}\label{equ5}
  \max_{1\leq j\leq 2k}\,|g'_j|>\,M_{m+1}+m\quad\text{on \,$\partial P_m$\,\,\, for each $m\in\N$}\,. 
\end{equation}
 Consider the following sets
 \begin{align*}
     &A_m:=\big\{x\in P_m\,\big |\,\max_{1\leq j\leq 2k}\,|g'_j(x)|\leq\,M_{m+1}+m\big\}\,,\\
     &B_m:=\big\{x\in P_{m+1}\backslash P_m\,\big |\,\max_{1\leq j\leq 2k}\,|g'_j(x)|\leq\,M_{m+1}+m\big\}\,.
 \end{align*}
 Observe that, $A_m, B_m$ are compact subsets and by \eqref{equ5}, we get $A_m\cap B_m=\emptyset$. The $\mathcal{O}(X)$-convex hull of $A_m\cup B_m$ is then contained in $K_{m+2}$, since $K_{m+2}$ is $\mathcal{O}(X)$-convex, and can be written as $A'_m\cup A_m\cup B_m$ where $A'_m\subset K_{m+2}\backslash P_{m+1}$. By applying Proposition \ref{apprx prop} to approximate a function that is sufficiently close to $0$ on $A'_m\cup A_m$ and sufficiently close to a large constant on $B_m$, we construct a sequence of GH functions $\big\{h_j\big\}_{j\in\N}\subset\mathcal{O}(X)$ such that, for each $m\in\N$
 \begin{equation}\label{equ6}
    \begin{aligned}
        &|h_m|\leq\frac{1}{2^m}\quad\text{on $A_m$}\,.\\
        &|h_m|>\,M_{m+1}+m+1+\big |\sum^{m-1}_{j=1}\,h_j\big |\quad\text{on $\partial B_m$}\,.
    \end{aligned}
\end{equation}
Define 
\begin{equation}\label{equ7}
    g'_{2k+1}:=\sum^{\infty}_{m=1}h_m\,\,.
\end{equation}
Since $B_m\subset A_{m+1}\subset A_{m+2}\subset\cdots$, the sum \eqref{equ7} converges by \eqref{equ6}, and $g'_{2k+1}\in\mathcal{O}(X)$. Hence, by \ref{equ4} and \eqref{equ7}, we obtain the GH function $g'=(g'_1,\ldots,g'_{2k+1})\in\mathcal{O}^{2k+1}(X)$, satisfying \eqref{equ2}. Thus, the theorem is proved.
\end{proof}
\begin{corollary}\label{emb cor}
    Let $X^{2n}$ be a GC Stein manifold of type $K>0$. Then there exists a proper GH embedding $F:X\longrightarrow\R^{2n-2k}\times\C^{2k+1}$. Moreover, the image $\img(F)$ is a closed embedded GC submanifold of\,\,\, $\R^{2n-2k}\times\C^{2k+1}$.
\end{corollary}
\begin{proof}
    Follows from Theorem \ref{imp thm1} and Theorem \ref{main3}.
\end{proof}
Consider $\R^{2M}\times\C^N$ (with $M,N>0$) endowed with the product GCS (see Example \ref{prdct gcs eg}). Let $Y^{2n}\subset\R^{2M}\times\C^N$ be a closed embedded GC submanifold (cf. Definition \ref{sub def mfld}) of type $k=n-M>0$ and $N>k$. By Theorem \ref{imp thm1} and Corollary \ref{imp cor1}, $Y$ is both GH convex and GH separable (see Definition \ref{main def}). So, to prove that $Y$ is a GC Stein manifold, we just need to show that $Y$ is GC regular.

\vspace{0.5em}
Now observe that the inclusion map $i:Y\hookrightarrow\R^{2M}\times\C^N$ is a GH map (see Definition \ref{GH map}). By identifying $\R^{2M}=\bigoplus^{M}_{l=1}\R^2$, consider the natural projections, which are also GH maps:
\begin{align*}
&pr_l:\R^{2M}\times\C^N\longrightarrow\R^2\,,\quad l=1,\ldots M\,.\\ 
&pr'_j:\R^{2M}\times\C^N\longrightarrow\R^2\,,\quad j=1,\ldots N\,.
\end{align*}
Then, for any point $y\in Y$, the set of GH maps $\big\{\pr_1\circ i,\ldots,\pr_{M}\circ i,\,\pr'_{j_1}\circ i,\ldots,\pr'_{j_{k}}\circ i\big\}$ on $Y$, where $\{j_1,\ldots,j_k\}\subset\{1,\ldots,N\}$ represents $k$ indices depending on $y$, satisfies the GH regularity condition in Definition \ref{main def}. Thus, we have proved the following.
\begin{prop}\label{prop sub}
    Let $Y^{2n}\subset\R^{2n-2k}\times\C^N$ be a closed embedded GC submanifold of type $k>0$ and $N>k$. Then $Y$ is a GC Stein manifold.
\end{prop}
\begin{theorem}\label{main4}
Let $X^{2n}$ be a regular manifold of type $K>0$. Then $X$ is a GC Stein manifold if and only if a proper GH embedding $F:X\longrightarrow\R^{2n-2k}\times\C^{2k+1}$ exists. Moreover, $\img(F)$ is a closed embedded GC submanifold in $\R^{2n-2k}\times\C^{2k+1}$.
\end{theorem}
\begin{proof}
    Follows from Corollary \ref{emb cor} and Proposition \ref{prop sub}.
\end{proof}

{\bf Acknowledgement.} I would like to thank my PhD supervisor Prof. Mainak Poddar for many stimulating and helpful discussions.

\end{document}